\documentclass{amsart}

\usepackage{amsmath,enumerate,enumitem,amscd,amsfonts,amsthm,amssymb,amscd,mathrsfs}
\allowdisplaybreaks[4]
\usepackage[colorlinks,linkcolor=blue,citecolor=red]{hyperref}
\hypersetup{CJKbookmarks}
\usepackage{appendix}
\usepackage{bm}
\usepackage{cite}
\usepackage{enumerate}
\usepackage{comment}

\usepackage{geometry}
\usepackage{graphicx}
\usepackage{subfigure}
\usepackage{setspace}
\usepackage{stmaryrd}
\usepackage{url}
\usepackage{float}

\usepackage{tikz}
\usetikzlibrary{math}
\usepackage{pgfplots}
\pgfplotsset{compat=1.10}
\usepgfplotslibrary{fillbetween}
\usetikzlibrary{patterns}
\usepackage{caption} 
\newtheorem{theorem}{Theorem}[section]
\newtheorem{lemma}[theorem]{Lemma}
\newtheorem{proposition}[theorem]{Proposition}
\newtheorem{corollary}[theorem]{Corollary}
\newtheorem{example}[theorem]{Example}
\newtheorem{notation}[theorem]{Notation}
\theoremstyle{definition}
\newtheorem{definition}[theorem]{Definition}
\newtheorem{remark}[theorem]{Remark}

\newtheorem*{note*}{Notation}
\newtheorem*{convention*}{Convention}

\numberwithin{equation}{section}

\DeclareMathOperator{\diff}{d}
\DeclareMathOperator{\Op}{\mathbf{Op}}

\DeclareFontEncoding{FMS}{}{}
\DeclareFontSubstitution{FMS}{futm}{m}{n}
\DeclareFontEncoding{FMX}{}{}
\DeclareFontSubstitution{FMX}{futm}{m}{n}
\DeclareSymbolFont{fouriersymbols}{FMS}{futm}{m}{n}
\DeclareSymbolFont{fourierlargesymbols}{FMX}{futm}{m}{n}
\DeclareMathDelimiter{\VERT}{\mathord}{fouriersymbols}{152}{fourierlargesymbols}{147}

\newcommand{\A}{\mathfrak{A}}
\newcommand{\sgn}{\mathrm{sgn}}
\newcommand{\size[1]}{\langle{#1}\rangle}
\newcommand{\E[1]}{\big\VERT{#1}\big\VERT}
\newcommand{\T}{\mathsf{T}}
\newcommand{\proj}{\bm{\Pi}}

\newcommand{\PM}{\mathrm{PM}}
\newcommand{\PL}{\mathrm{PL}}

\newcommand{\DN}{\mathrm{DN}}

\newcommand{\N}{\mathcal{N}}
\newcommand{\R}{\mathcal{R}}

\newcommand{\Ta}{\mathcal{T}}
\newcommand{\M}{\mathcal{M}}
\newcommand{\DG}{\mathrm{DG}}
\newcommand{\GU}{\mathrm{GU}}
\newcommand{\Id}{\mathrm{Id}}

\newcommand{\Ext}{\mathrm{Ext}}

\newcommand{\grow}{\mathtt{g}}
\newcommand{\low}{\mathtt{low}}
\newcommand{\high}{\mathtt{high}}
\newcommand{\inv}{\mathtt{inv}}

\def\st{\mathtt{s}}
\def\unst{\mathtt{u}}
\def\disp{\mathtt{d}}
\def\cen{\mathtt{c}}

\def\RE{\mathtt{Re}}
\def\IM{\mathtt{Im}}

\def\der{\mathbf{d}}
\def\del{\bm{\delta}}
\def\1{\mathbf{1}}

\def\xZ{\mathbb{Z}}
\def\xC{\mathbb{C}}
\def\xG{\mathbb{G}}
\def\DuG{\widehat{\mathbb{G}}}
\def\xR{\mathbb{R}}
\def\xS{\mathcal{S}}
\def\xN{\mathbb{N}}
\def\xT{\mathbb{T}}

\def\Id{\mathrm{Id}}

\def\tame{\mathrm{tame}}

\def\loc{\rm{loc}}

\def\dx{\mathrm{d}\, \! x}

\def\dt{\mathrm{d}\, \! t}
\def\dtau{\mathrm{d}\, \! \tau}
\def\dxi{\mathrm{d}\, \! \xi}

\title{Local Invariant Structures in the Dynamics of Capillary Water Jet}

\author{Chengyang Shao}
\address{Institut des Hautes Études Scientifiques}
\email{shao@ihes.fr}

\author{Haocheng Yang}
\address{Department of Mathematics, New York University Abu Dhabi}
\email{hy3448@nyu.edu}

\begin{document}
\begin{spacing}{1.2}
\begin{abstract}
Physical experiments show that a capillary water jet is exponentially unstable under long wave perturbations, while remaining stable under short wave perturbations. Measurements indicate that the exponential growth rate in the long wave regime agrees quantitatively with the classical predictions of Rayleigh and Plateau, known as \emph{Rayleigh-Plateau instability}. In this paper, we provide a mathematical justification of these experimental observations. The motion of the water jet is modeled by an axisymmetric irrotational Eulerian free-boundary system governed by surface tension. We prove that the (un)stable directions in the linearized system, corresponding to long wave perturbations, are indeed tangent to a (un)stable invariant manifold of the full nonlinear system. On the other hand, the elliptic directions, corresponding to short wave perturbations, are indeed tangent to a center invariant set in a generalized sense. These results give a positive answer to the question raised by Lin-Zeng concerning the existence of invariant manifolds for Eulerian free-boundary systems. A notable finding is the existence of infinite dimensional hyperbolic invariant manifolds in the dynamics of infinitely long capillary water jet, in absence of spectral gap. The major methodological contribution is the construction of ``paradifferential propagator" corresponding to linear paradifferential hyperbolic systems, effectively balancing the loss of regularity due to quasilinearity. The method can be generalized to a broad class of quasilinear problems.
\end{abstract}
\maketitle

\section{Introduction}\label{Sec1}
\subsection{Equations Governing a Capillary Fluid Jet}
Consider the motion of an axially symmetric jet of incompressible, irrotational ideal fluid under zero gravity, either infinitely long or under $2\pi$-periodic boundary condition. We are interested in the local invariant structures in the dynamics of the corresponding partial differential system. 

Suppose the fluid has constant density $\varrho_0$, and the surface tension coefficient at the fluid-gas interface is a constant $\sigma_0$. Suppose also, with out loss of generality, that the air pressure of the ambient space is zero. Denote by $\Omega_t$ the spatial region occupied by the fluid at time $t$, and $M_t=\partial \Omega_t$ its boundary. We shall use $\bm{N}$ to denote the outward pointing normal vector field of $M_t$. Since the fluid is incompressible and irrotational, there is a harmonic velocity potential $\Psi(t)$ inside the region $\Omega_t$. The free boundary naturally imposes the kinematic condition
$$
\frac{\partial M_t}{\partial t}\cdot\bm{N}=\bar\nabla\Psi\cdot\bm{N}.
$$
Here $\partial M_t/\partial t$ denotes the velocity vector field on the interface $M_t$. Integrating the potential Euler system along any path in the fluid domain, we obtain the Bernoulli equation. Restricting to the boundary, taking into account the Young-Laplace law for mean curvature and boundary pressure jump, we obtain the pressure balance condition (also known as dynamic condition),
$$
\left.\left( \frac{\partial\Psi}{\partial t}+\frac{1}{2}|\bar\nabla\Psi|^2 \right)\right|_{M_t}
=-\frac{\sigma_0}{\varrho_0} H(M_t)+\text{function of time}.
$$
Here $\bar\nabla\Psi$ denotes the (3-dimensional) gradient of the scalar function $\Psi$, and $H(M_t)$ denotes the (scalar) mean curvature of $M_t$. 

\begin{figure}[h]
\centering
\includegraphics[width=0.6\textwidth,angle=0]{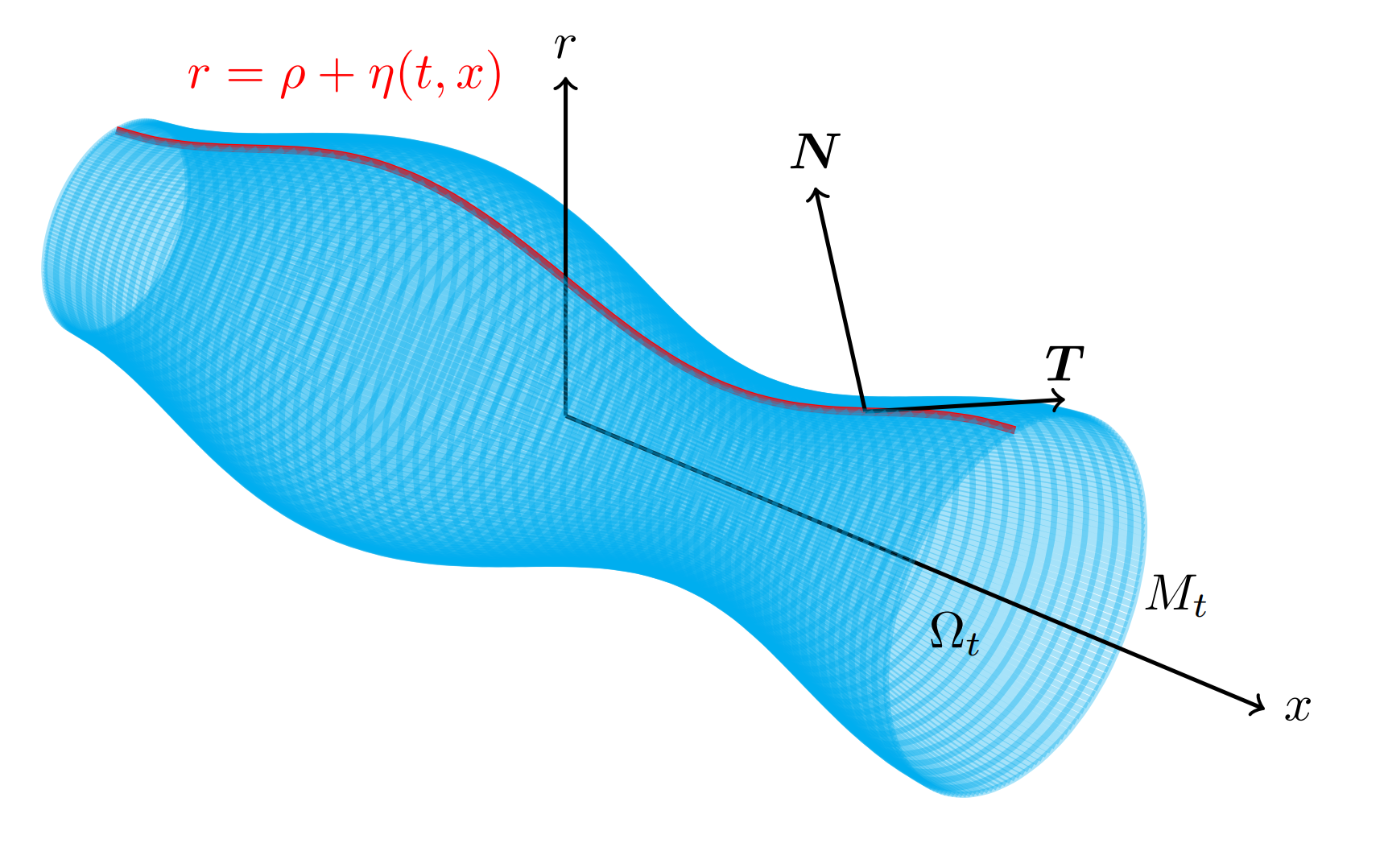}
\caption{Geometry of the water jet. Code adapted from \cite{HK2023}.}
\label{Cylinder}
\end{figure}

We consider the perturbation of the steady solution, i.e. a horizontal perfect cylinder of jet  moving at a uniform speed. Suppose the unperturbed cylinder has radius $\rho$. We shall set a cylindrical coordinate in the 3D space so that the axial coordinate $x$ coincides with the axis of the cylinder, and the radial coordinate $r$ coincides with the radial direction of the unperturbed cylinder. Dependence on the angular coordinate is abbreviated due to axial symmetry. The perturbed interface shall be the revolution surface generated by the graph $r=\rho+\eta(t,x)$, where $\eta$ has small magnitude. If $(\bm{T},\bm{N})$ is the tangent-outward normal frame in compatible orientation with the coordinate frame $(\bm{e}_x,\bm{e}_r)$, then
$$
\bm{T}[\eta]=\frac{1}{\sqrt{1+|\partial_x\eta|^2}}\bm{e}_x
+\frac{\partial_x\eta}{\sqrt{1+|\partial_x\eta|^2}}\bm{e}_r,
\quad
\bm{N}[\eta]=\frac{-\partial_x\eta}{\sqrt{1+|\partial_x\eta|^2}}\bm{e}_x
+\frac{1}{\sqrt{1+|\partial_x\eta|^2}}\bm{e}_r.
$$
See Figure \ref{Cylinder} for the geometry. The (scalar) mean curvature in this coordinate system then reads
$$
H[\eta]=-\frac{\partial_x^2\eta}{(1+|\partial_x\eta|^2)^{3/2}}+\frac{1}{(\rho+\eta)\sqrt{1+|\partial_x\eta|^2}}.
$$

We introduce the \emph{Dirichlet-Neumann operator} that maps the Dirichlet boundary value of $\Psi$ to the normal derivative $\bar\nabla\Psi|_{M_t}\cdot\bm{N}$. Under this given parameterization, the velocity potential is considered as a function $\Psi(t,x,r)$. We write $\psi(t,x)=\Psi(t,x,\rho+\eta(x))$ for its trace on the boundary $M_t$, and write
$$
G[\eta]:\psi\mapsto \sqrt{1+|\partial_x\eta|^2}(\bar\nabla\Psi)(t,x,\rho+\eta(x))\cdot \bm{N}[\eta(t,x)].
$$
The kinematic condition then becomes
$$
\partial_t\eta=G[\eta]\psi.
$$

To convert the pressure balance condition into an explicit form, we shall express the value of $\partial\Psi/\partial t$ and $\bar\nabla\Psi$ on the boundary in terms of $\partial_t\psi$, $\partial_x\psi$ and $G[\eta]\psi$. Restricting to the interface $(x,\eta(x,t))$, there holds
$$
\bar\nabla\Psi\cdot \bm{N}
=\frac{-\partial_x\eta}{\sqrt{1+|\partial_x\eta|^2}}\partial_x\Psi
+\frac{1}{\sqrt{1+|\partial_x\eta|^2}}\partial_r\Psi,
$$
while by the chain rule,
$$
\partial_x\psi(t,x)=\partial_x\Psi+\partial_r\Psi\partial_x\eta.
$$
Thus we can solve, on the interface $M_t$,
$$
\partial_r\Psi
=\frac{1}{1+|\partial_x\eta|^2}\left(\partial_x\Psi\partial_x\eta+G[\eta]\psi\right).
$$
Again by the chain rule,
$$
\partial_x\Psi
=\partial_x\psi-\partial_r\Psi\partial_x\eta,
\quad
\partial_t\Psi=\partial_t\psi-\partial_r\Psi\partial_t\eta.
$$
Substituting these back to the pressure balance condition, we arrive at the \emph{capillary water jet system}:
\begin{equation}\label{EQ}
\left\{
\begin{aligned}
\partial_t\eta&=G[\eta]\psi,\\
\partial_t\psi&+\frac{1}{2}|\partial_x\psi|^2
-\frac{\left(\partial_x\psi\partial_x\eta+G[\eta]\psi\right)^2}{2(1+|\partial_x\eta|^2)}
=\frac{1}{\sqrt{1+|\partial_x\eta|^2}}\left(\frac{\partial_x^2\eta}{1+|\partial_x\eta|^2 }-\frac{1}{\rho+\eta}\right)+\frac{1}{\rho}.
\end{aligned}
\right.
\end{equation}
Here we rescaled $t$ and $\psi$ suitably to make the ratio $\sigma_0/\varrho_0=1$. The addendum $\rho^{-1}$ is to make the right-hand side small when $\eta\simeq0$, and it results in a mere shift of $\psi$ by a function of $t$. 

The formulation (\ref{EQ}) is inspired by Zakharov-Craig-Sulem; see \cite{CSS1992,DKSZ1996}. The local theory of (\ref{EQ}) does not essentially differ from the well-known 2D capillary water waves system except for the curvature term. Local well-posedness for (\ref{EQ}) can either be inferred from the general local theory for free-boundary Eulerian systems as in \cite{CoutandShkoller2007,SZ2008a,SZ2008b,ShatahZeng2011}, or developed modeling that of capillary water waves as in \cite{HK2023,Yang2024}. Local theory for water waves system under different geometric and capillarity assumptions have been established in \cite{BG1998,Wu1997,Wu1999,Lannes2005,MingZhang2009,ABZ2011,ABZ2014,Shao2026Cauchy}, and may apply to (\ref{EQ}) after suitable adaptation.

However, due to the curved geometric nature of the cylinder, the long time behavior of (\ref{EQ}) is very different from the one in capillary water waves system for water level. In particular, the so-called \emph{Rayleigh-Plateau instability} occurs in the system under long wave perturbation. It is the goal of this paper to provide a rigorous mathematical justification of such instability.

\subsection{Experimental Facts and Formal Explanation}\label{Facts}
A number of experiments have been conducted to investigate the stability and instability of water jet under perturbations. See, for example, \cite{DG1966,GY1970,Lafrance1975} for experimental results, or Section 3 of \cite{EgVi2008} for a comprehensive review of related topics and experiments. For recent advances and engineering applications of the linear theory, see \cite{MKM2012, ZSL2019, Lohse2022}.

In these experiments, the water jet is disturbed by single-frequency waves of varying wave lengths. The temporal evolution of the jet profile is captured via high-speed photography. At a given time $t$, the deformation of the jet is measured by the difference $\Delta\eta(t):=\eta_{\max}(t,x)-\eta_{\min}(t,x)$.
Experimental observations consistently demonstrate an exponential dependence of $\Delta\eta(t)$ on time:  
$$
\Delta\eta(t)\simeq e^{\omega_et},
\quad \text{for some }\omega_e\geq0.
$$
The rate $\omega_e$ depends on the wave length -- or equivalently the wave number -- of the initial perturbation. If the wave length is $>2\pi\rho$, the ``characteristic length" of the unperturbed water jet, then $\omega_e>0$, indicating that the water jet is very unstable. Indeed, it is observed that the jet breaks up within only logarithmic time after disturbance. In contrast, if the wave length is $<2\pi\rho$, the disturbance will not be amplified within a very long time, and the jet is very stable. Typical profiles of the jet under long wave and short wave perturbations are illustrated in Figure \ref{Fig2}. Figure \ref{Fig3} compares the theoretical relation between $\omega_e$ and the dimensionless wave number $k=\rho\xi$ with experimental measurements.

A formal explanation of these experimental facts at the \emph{linear} level is available since the time of Rayleigh \cite{Rayleigh1878} and Plateau \cite{Plateau1973}, and is extensively cited in physical literature. See for example, the well-known textbooks or lecture notes \cite{Bush2013} (Section 11.2), \cite{Chandrasekhar} (Section 111), \cite{Lamb1932} (Section 273-274). For generalizations to more complicated settings, see for example the stablizing effect of vorticity \cite{LM1991}, or the review \cite{EgVi2008}. We sketch the analysis in the sequel.

\newpage
\begin{figure}[h]
\centering
\begin{subfigure}
\centering
\includegraphics[width=0.6\textwidth,angle=0]{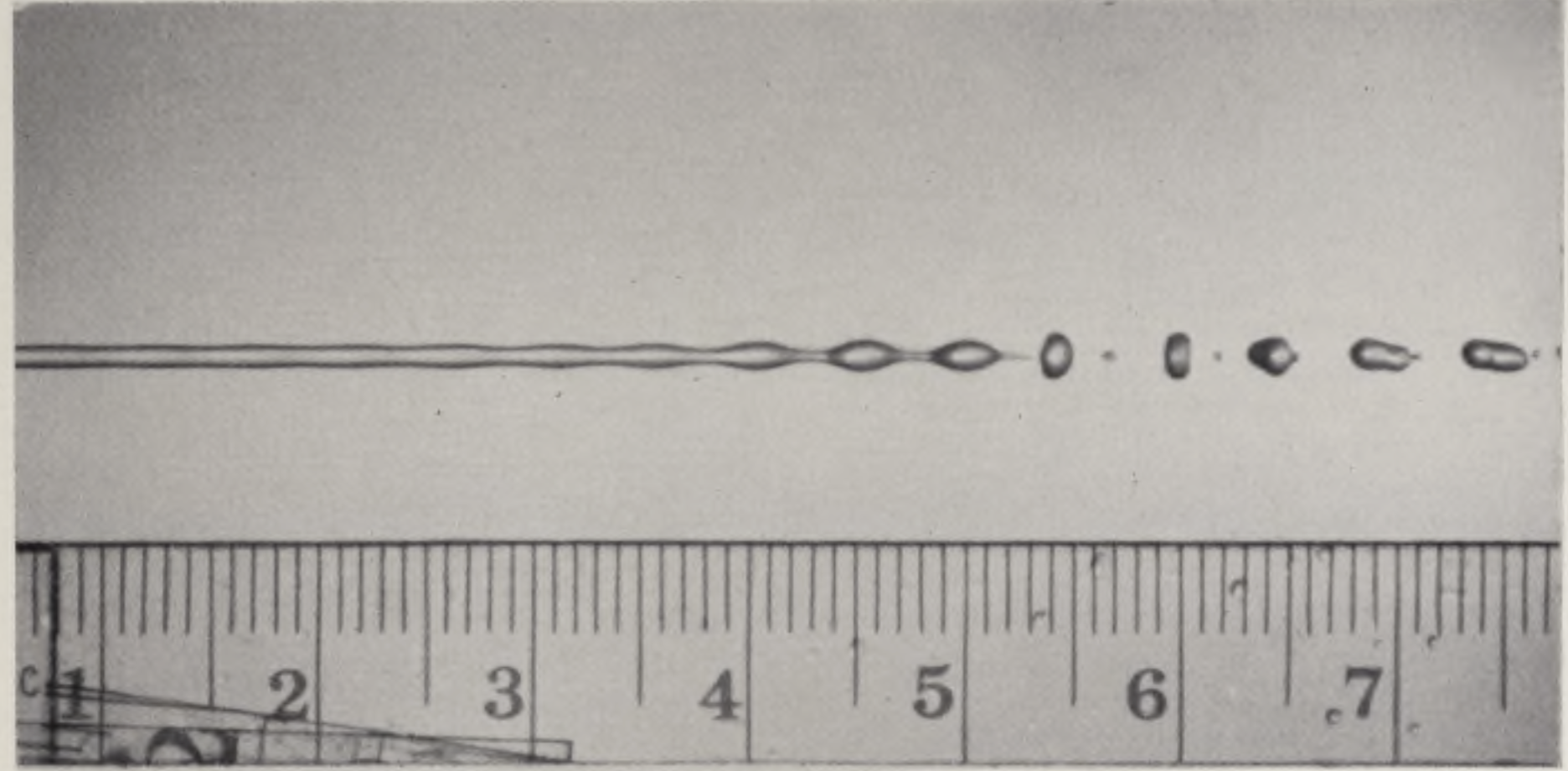}
\end{subfigure}\\
\begin{subfigure}
\centering
\includegraphics[width=0.6\textwidth,angle=0]{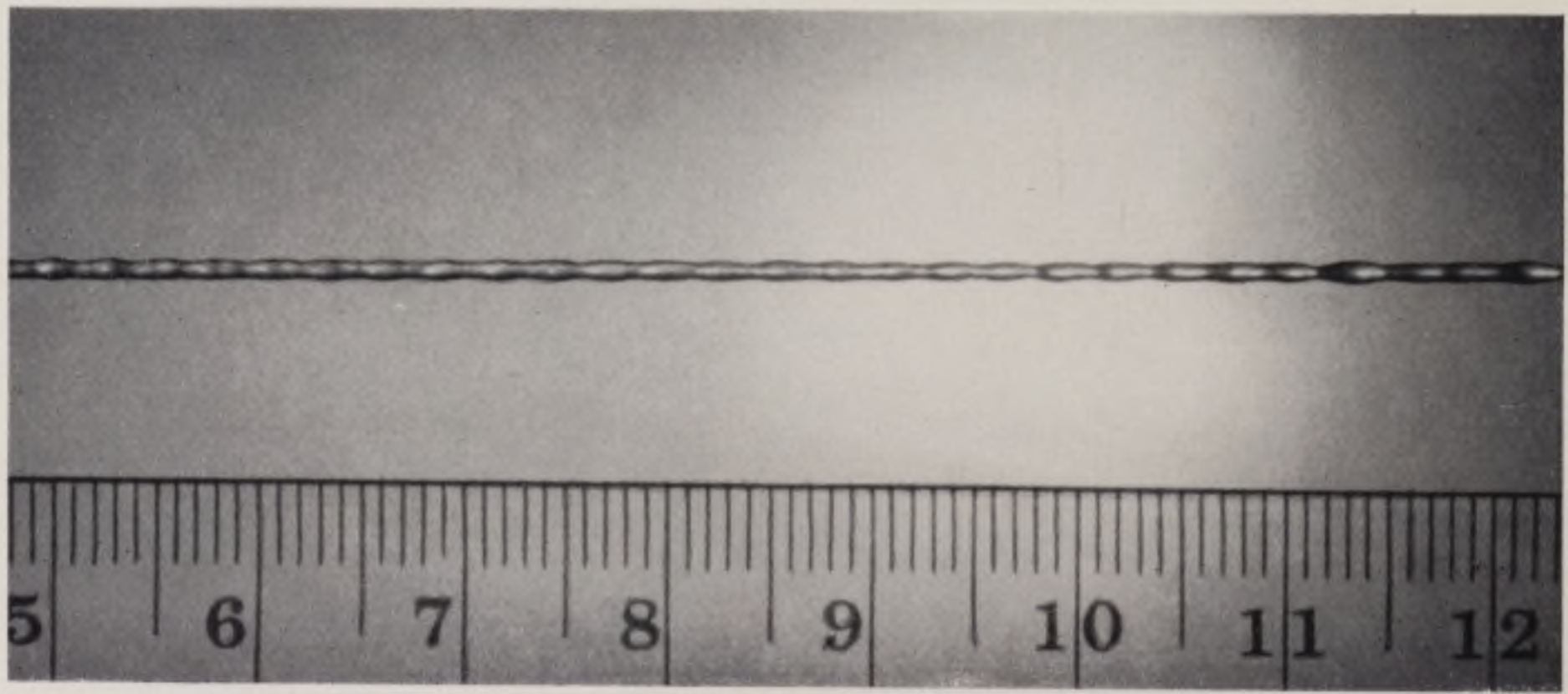}
\end{subfigure}
\caption{Above: amplifying disturbance near the most unstable mode leading to finite time break-up, corresponding to $\rho\xi\simeq0.678$. Below: non-amplifying disturbance, corresponding to $\rho\xi\simeq1.07$. Reprint from figure 8 and 10 of \cite{DG1966}.}
\label{Fig2}
\end{figure}

\begin{figure}[h]
\centering
\includegraphics[width=0.7\textwidth,angle=0]{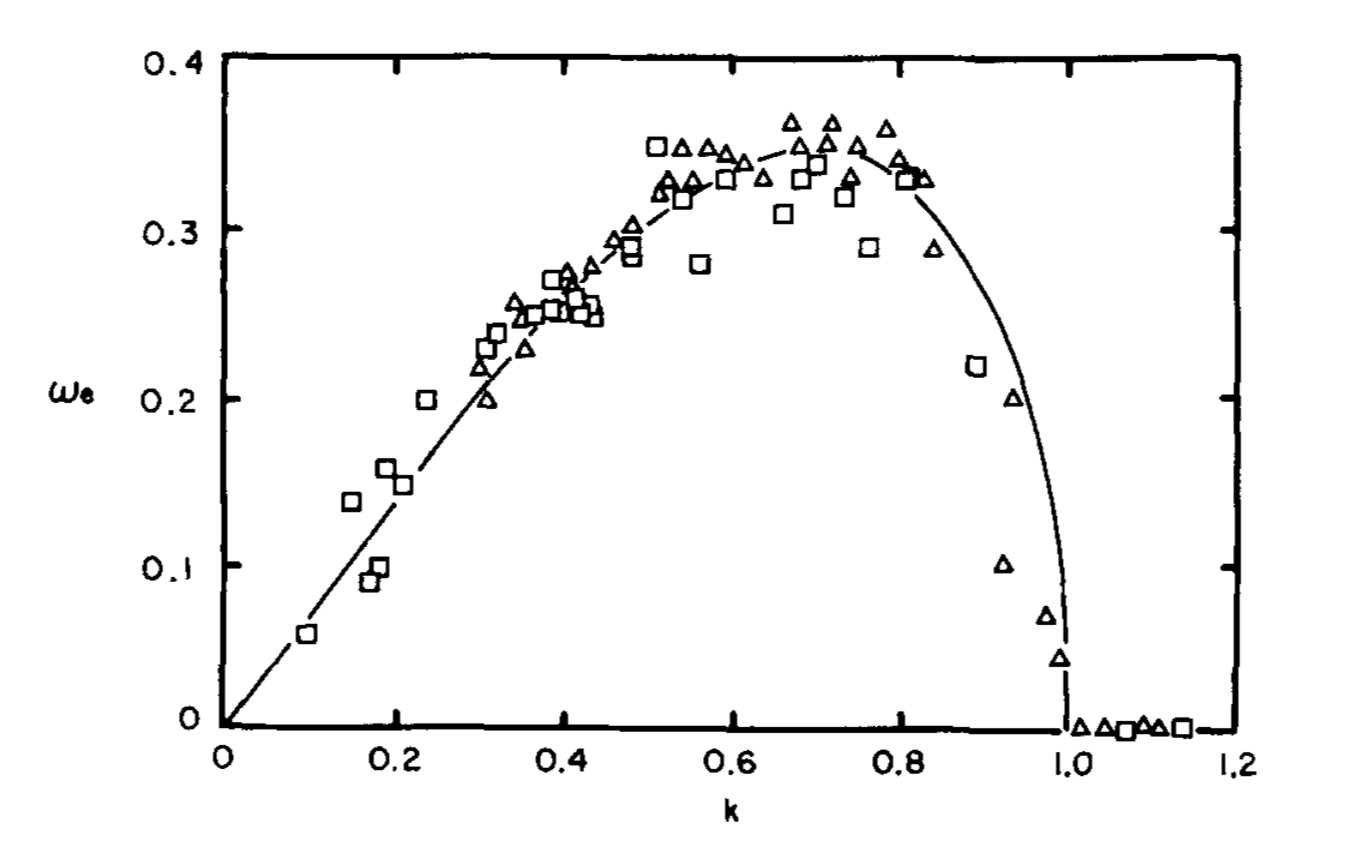}
\caption{Dependence of the exponential growth rate $\omega_e$ on the wave number $k=\rho\xi$. Solid line: Rayleigh's theoretical prediction \cite{Rayleigh1878}, which is the rescaled graph of $\Lambda_\grow(\xi)$, see (\ref{Disp_lh}). $\square$: data from \cite{DG1966}. $\triangle$: data from \cite{GY1970}. Reprint from Figure 3 of \cite{Lafrance1975}.}
\label{Fig3}
\end{figure}

\newpage
The water jet will be described by the Zakharov-Craig-Sulem system (\ref{EQ}). The system is defined for either $x\in\xR$ (infinitely long case) or $x\in\xT$ ($2\pi$-periodic case). We shall write $\xG$ for the physical space, which is either $\xR$ or $\xT$; and $\DuG$ for the frequency space, which is either $\xR$ or $\xZ$, respectively. Let us linearize the system (\ref{EQ}) around the steady solution $(\eta,\psi)=(0,0)$. To find the linearization of $G[\eta]\psi$, we must start from the Laplace equation under the cylindrical coordinate $(r,x)$:
$$
\Delta\Psi:=\left(\frac{\partial^2}{\partial r^2}+\frac{1}{r}\frac{\partial}{\partial r}+\frac{\partial^2}{\partial x^2}\right)\Psi=0.
$$
We can take Fourier transform with respect to $x\in\xG$ and find
$$
\frac{\diff^2}{\diff\!r^2}\hat\Psi(r,\xi)+\frac{1}{r}\frac{\diff}{\diff\!r}\hat\Psi(r,\xi)-\xi^2\hat\Psi(r,\xi)=0,
\quad \xi\in\DuG.
$$
The only solution smooth at $r=0$ is given by the modified Bessel function of order zero:
$$
\hat\Psi(r,\xi)=C_\xi I_0(r\xi),
\quad\text{with}\quad
I_0(r)=\sum_{k=0}^\infty\frac{r^{2k}}{4^k(k!)^2}.
$$
Therefore, the solution of the Dirichlet problem $\Delta\Psi=0,\,\Psi(\rho,x)=\psi(x)$ is formally represented as
$$
\Psi(r,x)=\int_{\DuG}\frac{I_0(r\xi)}{I_0(\rho\xi)}\hat \psi(\xi)e^{ix\xi}\dxi,
\quad \text{where }
\int_{\DuG}\cdot\dxi=\text{Haar integral on }\DuG.
$$
Consequently, the Neumann boundary value is just $\partial_r\Psi(r,x)|_{r=\rho}$, which formally equals
$$
\int_{\DuG}\frac{I_1(\rho\xi)}{I_0(\rho\xi)}\xi \hat \psi(\xi)e^{ix\xi}\dxi,
\quad\text{with}\quad
I_1(r)=I_0'(r)=\sum_{k=0}^\infty\frac{r^{2k+1}}{2^{2k+1}k!(k+1)!},
$$
provided that the Fourier integral converges in $L^2$. Standard asymptotic formula for modified Bessel functions shows that the multiplier $I_1(\rho\xi)\xi/{I_0(\rho\xi)}$ is a positive, analytic, even function in $\xi$, and equals $|\xi|+O(1)$ as $|\xi|\to\infty$. We thus find that $G[0]$ is a Fourier multiplier of order 1:
$$
G[0]\psi=\frac{I_1(\rho|D_x|)}{I_0(\rho|D_x|)}|D_x|\psi.
$$

On the other hand, linearization of the mean curvature operator $-H[\eta]$ at $\eta=0$ is simply $-|D_x|^2+\rho^{-2}$. Therefore, the linearized equation of (\ref{EQ}) at $(0,0)$ reads
\begin{equation}\label{LEQ(0,0)}
\frac{\partial }{\partial t}
\left(
\begin{matrix}
\eta\\
\psi
\end{matrix}
\right)
=L(D_x)\left(
\begin{matrix}
\eta\\
\psi
\end{matrix}
\right),
\quad L(D_x):=\left(\begin{matrix}
0 & G[0]\\
-H'[0] & 0
\end{matrix}\right).
\end{equation}
Passing to the Fourier side, the operator $L(D_x)$ in (\ref{LEQ(0,0)}) becomes a multiplier:
\begin{equation}\label{LEQ_Coeff}
L(\xi)=\left(\begin{matrix}
0 & \displaystyle{\frac{I_1(\rho|\xi|)}{I_0(\rho|\xi|)}|\xi|}\\
\displaystyle{\left(-\xi^2+\frac{1}{\rho^2}\right)} & 0
\end{matrix}\right).
\end{equation}
Let us then introduce the ``growing" and ``dispersive" eigenspeeds
\begin{equation}\label{Disp_lh}
\begin{aligned}
\Lambda_\grow(\xi)=\sqrt{\frac{I_1(\rho|\xi|)}{I_0(\rho|\xi|)}|\xi|\left(\frac{1}{\rho^2}-\xi^2\right)},
\quad
\Lambda_\disp(\xi)=\sqrt{\frac{1}{\rho^3}}\sqrt{\frac{I_1(\rho|\xi|)}{I_0(\rho|\xi|)}|\xi|\left(\xi^2-\frac{1}{\rho^2}\right)},
\end{aligned}
\end{equation}
The eigenvalues of matrix (\ref{LEQ_Coeff}) are $\pm\Lambda_\grow(\xi)$ when $0\leq|\xi|< \rho^{-1}$, and are $\pm i\Lambda_\disp(\xi)$ when $|\xi|\geq \rho^{-1}$. The solid line in Figure \ref{Fig3} sketches the graph of $\Lambda_\grow(\xi)$ for $0\leq|\xi|< \rho^{-1}$. 

Computing the matrix exponential $e^{tL(\xi)}$, the solution of (\ref{LEQ(0,0)}) is easily identified as
\begin{equation}\label{Sol_Lin}
\begin{aligned}
\begin{pmatrix}
\eta(t) \\ \psi(t)
\end{pmatrix}
&=\begin{pmatrix}
\cosh(t\Lambda_\grow(D_x)) 
& \displaystyle{G[0]}\frac{\sinh(t\Lambda_\grow(D_x))}{\Lambda_\grow(D_x)} \\ 
\displaystyle{-H'[0]}\frac{\sinh(t\Lambda_\grow(D_x))}{\Lambda_\grow(D_x)} &
\cosh(t\Lambda_\grow(D_x))
\end{pmatrix}
\1_{|\xi|<\rho^{-1}}(D_x)\begin{pmatrix}
\eta(0) \\ \psi(0)
\end{pmatrix}\\
&\quad+
\begin{pmatrix}
\cos(t\Lambda_\disp(D_x)) 
& \displaystyle{G[0]}\frac{\sin(t\Lambda_\disp(D_x))}{\Lambda_\disp(D_x)} \\ 
\displaystyle{-H'[0]}\frac{\sin(t\Lambda_\disp(D_x))}{\Lambda_\disp(D_x)} &
\cos(t\Lambda_\disp(D_x))
\end{pmatrix}
\1_{|\xi|\geq\rho^{-1}}(D_x)\begin{pmatrix}
\eta(0) \\ \psi(0)
\end{pmatrix}
\end{aligned}
\end{equation}
Here we take on the convention $z^{-1}\sinh(tz)\big|_{z=0}=z^{-1}\sin(tz)\big|_{z=0}=t$.

Now we define
\begin{equation}\label{Max_Lambda}
\lambda:=\max_{\xi\in(0,\rho^{-1})}\Lambda_\grow(\xi),
\end{equation}
and fix a real number $\mu\in(0,\lambda)$. Let $0<\xi_\mu<\xi_\mu'<\rho^{-1}$ be the two positive roots of the equation $\Lambda_\grow(\cdot)=\mu$. Corresponding to the sign of the real part of spectrum, we introduce the following $L(D_x)$-invariant closed subspaces of $L^2(\xG,\xR^2)$: the full and restricted \emph{stable subspace}
\begin{equation}\label{Stable_Space}
\begin{aligned}
E_\st&:=\left\{(\eta,\psi)\in L^2(\xG,\xR^2):\,(\hat{\eta},\hat{\psi})(\xi) \text{ supported in }
\,0<|\xi|<\frac{1}{\rho},\,\sqrt{-H'[0]}\eta+\sqrt{G[0]}\psi=0\right\},\\
E_\st^\mu&:=E_\st\cap\left\{(\hat{\eta},\hat{\psi})(\xi) \text{ supported in }
\,\xi_\mu\leq|\xi|\leq\xi_\mu'\right\};
\end{aligned}
\end{equation}
the full and restricted \emph{unstable subspace}
\begin{equation}\label{Unstable_Space}
\begin{aligned}
E_\unst&:=\left\{(\eta,\psi)\in L^2(\xG,\xR^2):\,(\hat{\eta},\hat{\psi})(\xi) \text{ supported in }
\,0<|\xi|<\frac{1}{\rho},\,\sqrt{-H'[0]}\eta-\sqrt{G[0]}\psi=0\right\},\\
E_\unst^\mu&:=E_\unst\cap\left\{(\hat{\eta},\hat{\psi})(\xi) \text{ supported in }
\,\xi_\mu\leq|\xi|\leq\xi_\mu'\right\};
\end{aligned}
\end{equation}
and the \emph{dispersive(center) subspace}
\begin{equation}\label{Disp_Space}
E_\disp:=\left\{(\eta,\psi)\in L^2(\xG,\xR^2):\,(\hat{\eta},\hat{\psi})(\xi) \text{ supported in }|\xi|=0\text{ or }
|\xi|\geq\frac{1}{\rho}\right\}.
\end{equation}

\begin{remark}\label{RT_Case}
If $\xG=\xR$, i.e. the infinitely long jet case, the spectrum of $L(D_x)$ is a continuum, whence there is \emph{no spectral gap} between the stable, unstable and elliptic part of it. If $\xG=\xT$, i.e. the periodic boundary case, the frequencies $\xi$ only take on integer values, and the stable, unstable and elliptic spectra stay away from each other. In that case, if 
$$
\mu\leq \min_{\xi\in\xZ,0<|\xi|<\rho^{-1}}\Lambda_\grow(\xi),
$$
then $E_\st^\mu=E_\st$ and $E_\unst^\mu=E_\unst$, and they are all finite-dimensional. The requirement of including zero frequency in (\ref{Disp_Space}) makes sense only when $\xG=\xT$.
\end{remark}

We then directly conclude from (\ref{Sol_Lin}) the following:
\begin{enumerate}[label=\textbf{(Lin\arabic*)}]
\item\label{Lin1} If the initial data is in $E_\st^\mu$, then the solution to (\ref{LEQ(0,0)}) stays in $E_\st^\mu$, decays like $e^{-\mu t}$ as $t\to+\infty$ while grows like $e^{\mu |t|}$ as $t\to-\infty$.

\item\label{Lin2} If the initial data is in $E_\unst^\mu$, then the solution to (\ref{LEQ(0,0)}) stays in $E_\unst^\mu$, decays like $e^{-\mu|t|}$ as $t\to-\infty$ while grows like $e^{\mu t}$ as $t\to+\infty$.

\item\label{Lin3} If the initial data is in $E_\disp$, then the solution to (\ref{LEQ(0,0)}) stays in $E_\disp$. Furthermore, if the frequency support of the initial data stays away from $0$ and $\pm\rho^{-1}$, then the solution remains bounded as $t\to\pm\infty$ and is oscillatory.
\end{enumerate}

In other words, the subspaces $E_\st,E_\unst$ are respectively the \emph{stable} and \emph{unstable} subspaces for the linear system (\ref{LEQ(0,0)}), and the system (\ref{LEQ(0,0)}) is purely \emph{dispersive} on the subspace $E_\disp$. Furthermore, the mode that grows the fastest appears at the maxima of the function $\Lambda_\grow(\xi),\,0<|\xi|<\rho^{-1}$, which corresponds to $\xi \rho\simeq 0.678$. The existence of positive spectral points is a consequence of the linearized mean curvature operator, thus physically a consequence of surface tension. 

The dispersive relation (\ref{Disp_lh}) was first obtained by Rayleigh \cite{Rayleigh1878}. The formal linear analysis above turns out to match the experimental observations very well, except near the break-up time. Indeed, as indicated in Figure \ref{Fig3}, the relation between $\omega_e$ and the wave number $\rho\xi$ measured in the experiments conducted by \cite{DG1966,GY1970} is in great consistency with (\ref{Disp_lh}): when $0<\rho\xi<1$, the data points $(\rho\xi,\omega_e)$ concentrate near the graph of the function $\Lambda_\grow(\xi)$, and when $\rho\xi>1$, the data points $(\rho\xi,\omega_e)$ simply have $\omega_e=0$, meaning that the disturbance does not amplify within a very long time.

\subsection{The Main Results}

Our main results consist of two mathematical theorems that rigorously justify the quantitative experimental facts related to water jet stability and instability. 

\begin{theorem}\label{Main1}
    Consider the water jet system~\eqref{EQ} for an infinitely long jet or under $2\pi$-periodic boundary condition. Fix the unperturbed radius $0<\rho<1$ and a reference Lyapunov exponent 
    $$
        0<\mu<\lambda:=\max_{0<|\xi|<\rho^{-1}}\Lambda_\grow(\xi).
    $$
    For any $s_0>5$, there exists a small real number $\varepsilon>0$, and a unique finite-dimensional $C^\infty$ submanifold $M_\st^\mu$ (resp. $M_\unst^\mu$) of $H^{s_0+1}\times H^{s_0+1/2}$, invariant under the flow map of~\eqref{EQ} within the $\varepsilon$-neighborhood of $(0,0)$, satisfying the following properties:
    \begin{itemize}
        \item $M_\st^\mu$ (resp. $M_\unst^\mu$) contains the point $(0,0)$ and its tangent space at $(0,0)$ coincides with the stable subspace $E_\st^\mu$ (resp. unstable subspace $E_\unst^\mu$), defined by \eqref{Stable_Space} (resp. \eqref{Unstable_Space}). 
        \item $(\eta_0,\psi_0)\in H^{s_0+1}\times H^{s_0+1/2}$ belongs to $M_\st^\mu$ (resp. $M_\unst^\mu$) if and only if the orbit $(\eta(t),\psi(t))$ starting from $(\eta_0,\psi_0)$ is attracted to zero with speed $O(e^{-\mu t})$ in     $H^{s_0+1}\times H^{s_0+1/2}$ as $t \to +\infty$ (resp. $-\infty$).
        \item For all $s\geq s_0$, $M_\st^\mu$ (resp. $M_\unst^\mu$) is a $C^\infty$ submanifold of $H^{s+1}\times H^{s+1/2}$, and the exponential decay holds in $H^{s+1}\times H^{s+1/2}$. Namely, $M_\st^\mu$ (resp. $M_\unst^\mu$) consists only of smooth functions and the attraction occurs with respect to all the Sobolev norms.
    \end{itemize}
\end{theorem}

\begin{corollary}\label{cor:Main}
    In the case of $2\pi$-periodic boundary condition, under the same assumptions as in Theorem~\ref{Main1}, the manifolds $M_\st^\mu$ (resp. $M_\unst^\mu$) are identical for
    $$
        0<\mu\leq\min_{\xi\in\xZ,0<|\xi|<\rho^{-1}}\Lambda_\grow(\xi).
    $$
    This gives the full stable manifold $M_\st$ (resp. unstable manifold $M_\unst$), which is finite dimensional with
    $$
        \dim M_\st = \dim E_\st =  \dim M_\unst = \dim E_\unst = 2\times\text{the largest integer }<\rho^{-1}.
    $$
\end{corollary}

\begin{remark}\label{No_Spec_Gap}
A notable feature in Theorem \ref{Main1} is that (un)stable submanifolds can be constructed for continuous spectrum: in the infinitely long jet case, given the Lyapunov exponent $\mu$, there is \emph{no spectral gap} between the stable spectra $[-\lambda,-\mu]$ (corresponding to $E_\st^\mu$) and the complementary center-unstable spectra $(-\mu,\lambda]\cup i\xR$. To the authors' knowledge, this is the first result on existence of hyperbolic submanifolds for a quasilinear problem in absence of both dissipation and spectral gap. 
\end{remark}

\begin{theorem}\label{Main2}
Consider \eqref{EQ} under $2\pi$-periodic boundary condition. Fix the unperturbed radius $0<\rho<1$ and $s_0$ as in Theorem \ref{Main1}. For any $s_1\geq s_0+3/2$, there exists a neighborhood $U^{s_1}$ of $0$ in $H^{s_1+1}\times H^{s_1+1/2}$, and a subset $M_\cen\subset U^{s_1}$, passing through the point $(0,0)$, with the following properties:

\begin{itemize}
\item $M_\cen$ is of first order contact with, or equivalently ``tangent to" the dispersive subspace $E_\disp\cap H^{s_1}$ at $(0,0)$, in the following sense:  
\begin{itemize}
    \item There is a smooth diffeomorphism of $U^{s_1}$ to itself, such that the image of $M_\cen$ under the diffeomorphism has surjective orthogonal projection to $E_\disp\cap U^{s_1}$.
    \item There exists a constant $C_{s_1}>0$ such that $M_\cen$ is contained in the cone defined by (note that $E_\st\oplus E_\unst$ is finite dimensional)
    \begin{equation}\label{Conical}
    \|(\proj_\st+\proj_\unst)(\eta,\psi)\|_{L^2}
    \leq C_{s_1}\|\proj_\disp(\eta,\psi)\|_{H^{s_1+1}\times H^{s_1+1/2}}^2.
    \end{equation}
    Here $\proj_\st,\proj_\unst,\proj_\disp$ are respectively the spectral projections of $L(D_x)$ to $E_\st,E_\unst,E_\disp$, well-defined by Remark \ref{RT_Case}.
\end{itemize} 

\item $M_\cen$ is invariant under the flow map of~\eqref{EQ} within the neighborhood $U^{s_1}$. Furthermore, if an orbit on $M_\cen$ has initial value $(\eta_0,\psi_0)$ satisfying
$$
\|\eta_0\|_{H^{s_1+1}}+\|\psi_0\|_{H^{s_1+1/2}}=\varepsilon\leq\delta_0,
$$
then it stays on $M_\cen$ within the time interval $[-T_\varepsilon,T_\varepsilon]$, where $T_\varepsilon\gtrsim\varepsilon^{-1}$; within this time interval, there holds
$$
\|\eta\|_{H^{s_1+1}_x}+\|\psi\|_{H^{s_1+1/2}_x}\leq C_{s_1}\varepsilon.
$$

\item Any orbit starting from $U^{s_1}\setminus M_\cen$ must exit $U^{s_1}$ at some time. In particular, any orbit starting outside of the cone defined by \eqref{Conical} must exit $U^{s_1}$ at some time, while any global-in-time orbit staying close to the equilibrium must be contained in the invariant set $M_\cen$.
\end{itemize}
\end{theorem}

\begin{remark}
We are not able to prove that $M_\cen$ in Theorem \ref{Main2} is a submanifold. This does not seem to reflect shortcomings of our approach, but rather an intrinsic difficulty of the problem. See the discussion after Proposition \ref{Brouwer_Set}. It is also not clear if there is \emph{global-in-time} existence for small data with finite codimension under periodic boundary condition. This is open even for purely dispersive quasilinear PDEs.
\end{remark}

Let us explain the physical implications of Theorem~\ref{Main1} and~\ref{Main2} on the water jet system \eqref{EQ}. 

The existence of unstable manifold $M_\unst$ (or $M_\unst^\mu$) implies that the amplitude of a solution starting from $M_\unst$ (or $M_\unst^\mu$), no matter how small the initial size $\varepsilon\simeq0$ is, shall be amplified in logarithmic time $\simeq\log(1/\varepsilon)$, so that the radius of the jet will be of comparable magnitude with $\rho$. This shows that \eqref{EQ} is \emph{unstable under long wave perturbation}. 

Moreover, since e.g. $M_\unst$ in Theorem \ref{Main1} in the periodic boundary case is a smooth graph tangent to $E_\unst$, the restriction of (\ref{EQ}) to $M_\unst$ is smoothly conjugate to an ODE system on $E_\unst$:
\begin{equation}\label{Restricted_Dynamics}
\partial_tz=L(D_x)z+N_\unst(z),
\quad N_\unst:E_\unst \to E_\unst\text{ smooth and quadratic near }z=0.
\end{equation}
By the Hartman-Grobman theorem, the restricted dynamics on $M_\unst$ is topologically conjugate to the linear dynamics indicated by \ref{Lin2}. If we demand this conjugation to be $C^1$, then some non-resonance condition should be imposed; see for example the standard reference \cite{Sternberg1957} and \cite{Hartman1960}. However, (\ref{Restricted_Dynamics}) already implies that for each frequency $\xi\in (0,\rho^{-1})\cap\xN$, there exists an orbit, repelled exponentially fast from the equilibrium, with Lyapunov exponent \emph{exactly equal to} $\Lambda_\grow(\xi)$ (see e.g. Section 4.1 of the textbook \cite{ChiconeODE}). In other words, this justifies the solid curve in Figure \ref{Fig3}. 

The existence of the center invariant set $M_\cen$ shows that \eqref{EQ} is \emph{stable under short wave perturbation}: the lifespan of solutions starting from $M_\cen$ with initial size $\varepsilon\simeq0$ is at least $\simeq1/\varepsilon$, the one predicted by energy inequalities. Indeed, stronger results are to be expected; see Subsection \ref{Perspect}. Such stability is due to the oscillatory (dispersive) nature of the system for high frequency, corresponding to the data points lying on the horizontal axis in \ref{Fig2}. On the other hand, the third item in Theorem \ref{Main2} asserts that the union of small-magnitude, globally existing orbits must be \emph{non-generic} since the conical neighborhood defined by (\ref{Conical}) ``squeezes near $E_\disp\cap H^{s_1}$". Furthermore, it asserts a strong instability result: \emph{any orbit not ``in first order contact" with the center subspace $E_\disp$, no matter how close to the equilibrium initially, will eventually become large in magnitude.}

To summarize, Theorem~\ref{Main1},~\ref{Main2}, and Corollary~\ref{cor:Main} provide mathematical justifications of experimental facts on instability and stability for a perturbed water jet governed by surface tension. 

\subsection{Limitations of Classical Lyapunov-Perron Method}
The linearization (\ref{LEQ(0,0)}) suggests that (\ref{EQ}) is a partially hyperbolic system. Inspired by finite dimensional partially hyperbolic systems, it is natural to ask how much remains true for the quasilinear problem (\ref{EQ}). 

Let us first recall the \emph{Lyapunov-Perron method} for ODEs. Let $A$ be an $n\times n$ real matrix. Consider
\begin{equation}\label{Classical}
x'=Ax+N(x),\quad N:\xR^n\to\xR^n\text{ smooth and quadratic as }x\to0.
\end{equation}
The \emph{stable set}\footnote{From the definition, it is more appropriate to name these subsets as \emph{attracting} or \emph{repelling} subsets. Historical development of the theory of dynamical systems gives this somewhat inappropriate name, which has now become conventional.} of the equilibrium zero is defined to be the set of points in $\xR^n$ whose orbit under (\ref{Classical}) is attracted to zero as $t\to+\infty$. Similarly, the \emph{unstable set} of the equilibrium zero is the set of points in $\xR^n$ whose orbit under (\ref{Classical}) is attracted to zero as $t\to-\infty$. 

Split $\sigma(A)$ into $\sigma_\st(A)\cup\sigma_\unst(A)\cup\sigma_\cen(A)$, collecting eigenvalues of negative, positive and zero real parts. They are called the \emph{stable, unstable and elliptic} parts of $\sigma(A)$ respectively; $\sigma_\st(A)\cup\sigma_\unst(A)$ is called the \emph{hyperbolic} part. The corresponding $A$-invariant subspaces $E_\st,E_\unst,E_\cen$ are defined via spectral projection. The \emph{stable manifold theorem} asserts that the stable and unstable subsets of (\ref{Classical}) turn out to be local submanifolds. It also asserts the existence of a \emph{center submanifold} tangent to $E_\cen$.

\begin{theorem}[Classical Invariant Manifold Theorem]\label{Stab_Mfd_Finite}
Let $N(x)$ be a $C^m$ ($m\geq1$) mapping defined in some neighborhood of $0\in\xR^n$, vanishing quadratically as $x\to0$. The stable and unstable subsets of (\ref{Classical}) are locally $C^m$ submanifolds containing $0$. The tangent spaces of these manifolds at $0$ are $E_\mathrm{s}$ and $E_{\mathrm{u}}$, respectively. These hyperbolic submanifolds are locally unique. On the other hand, there exists a $C^m$ center submanifold for (\ref{Classical}), namely, a local $C^m$ invariant submanifold tangent to $E_\cen$ at 0.
\end{theorem}
\begin{proof}
This is exactly the content of the \emph{Lyapunov-Perron method}. See for example, \cite{HP1970}, Chapter IX of \cite{Hartman1964}, Chapter 9 of \cite{Teschl2012}, Chapter 4 of \cite{ChiconeODE}. Here we only provide a sketch of the proof. Let $\proj_{\st},\proj_{\unst},\proj_\cen$ \footnote{In a finite dimensional dynamical system, the invariant subspace associated to purely imaginary eigenvalues is usually called the center space, which is marked by the subscript $\cen$. In this article, this subspace is identical to the dispersive subspace (\ref{Disp_Space}), marked by the subscript $\disp$.} be the complementary $A$-spectral projections to $E_\st,E_{\unst},E_\cen$, respectively. They all commute with $A$. Given $0<\mu<\min|\RE\sigma_\st(A)|$, as a consequence of the classical Duhamel formula, any solution with $x(t)=O(e^{-\mu t}),t\to+\infty$ must solve the integral equation
$$
\begin{aligned}
x(t)&=e^{tA}w+\int_0^te^{(t-\tau)A}\proj_{\st}N(x(\tau))\dtau
-\int_t^{+\infty} e^{(t-\tau)A}(\proj_{\unst}+\proj_\cen)N(x(\tau))\dtau\\
&=:F(x;w),
\end{aligned}
$$
where $w\in E_\mathrm{s}$. If $w\in E_\mathrm{s}$ is sufficiently close to zero, the mapping $F$ is a $C^m$ contraction near 0. Therefore, given any $w\in E_\mathrm{s}$ sufficiently close to 0, there is a unique solution $x(t)=S(t;w)$ to the equation $x=F(x;w)$, with $C^m$ dependence on $w$. The center manifold is then the graph $w\mapsto S(0;w)$. Reverting the time direction, the argument gives the unstable manifold. 

On the other hand, introducing a smooth bump function $\kappa(x)$ near 0, (\ref{Classical}) becomes $x'=Ax+\kappa(x)N(x)$, which is globally well-posed. A center manifold $M^\kappa_\cen$ is then defined via the integral equation
$$
x(t)=e^{tA}w+\int_0^te^{(t-\tau)A}\proj_{\cen}(\kappa N)(x(\tau))\dtau
-\int_t^{+\infty} e^{(t-\tau)A}(\proj_{\st}+\proj_\unst)(\kappa N)(x(\tau))\dtau,
$$
where $w\in E_\cen$. The equation can be solved in the space of polynomially growing functions. The center manifold $M^\kappa_\cen$ is obtained in a similar manner as above. It is an invariant set for (\ref{Classical}) in the domain $\kappa\equiv1$, but different $\kappa$ may lead to different $M^\kappa_\cen$. Therefore, $M^\kappa_\cen$ is not necessarily unique.
\end{proof}

The Lyapunov-Perron method has been substantially generalized to infinite dimensional systems to study instability and invariant manifolds for Eulerian systems with fixed boundary and semilinear dispersive PDEs. The former was addressed by \cite{LZ2013}: a highly non-trivial argument shows that the system is an infinite dimensional ODE on the Lie group of volume preserving diffeomorphisms. Therefore, the Lyapunov-Perron method still applies. On the other hand, since dispersive systems involve neither dissipation nor enough smoothing, the variant of Lyapunov-Perron method for such systems relies on highly non-trivial dispersive (Strichartz) estimates. See for example \cite{BJ1989,BLZ1998,GJLS2000} for the earlier works, and \cite{DM2009,NS20111,NS2012,KNS20121,KNS20122,JLSX2020} for more recent developments. See Chapter 3 of \cite{NS20112} for a detailed review. 

However, similar question for quasilinear equations, proposed in \cite{LZ2013}, remained open until very recently. This is because the quasilinear counterpart of $N$ typically involves loss of derivatives in $u$. Consequently, the right-hand side of the Duhamel formula is of worse regularity than $u$, failing to close a fixed point argument. As discussed in \cite{LZ2013,SZ2008a,SZ2008b,ShatahZeng2011}, Eulerian systems with free boundary are inherently quasilinear. They cannot be considered as ODE due to unboundedness of the infinite dimensional Riemann curvature. Despite this, nonlinear instability for free boundary Eulerian systems is still available; see for example \cite{RT2011,CS2023,NS2023,CNS2025,JRSY2025,BMV2022,BMV2023,BMV2024,guo2003RTI,guo2007unstable,ian2016RTI} for various results on transversal instability and modulation instability of Stokes waves, Rayleigh-Taylor instability, and Kelvin-Helmholtz instability.

After the first version of this paper was completed, we learned of a recent independent preprint by Shatah-Zeng \cite{shatah2026invariant} on construction of hyperbolic invariant submanifolds for a large class of quasilinear PDEs, including the transverse perturbation problem for capillary water waves. Their framework is general, applying to nonlinearities decomposable into the form $\mathcal{A}(u)u+f(u)$, where (1) the linear operator $\mathcal{A}(u)$ enjoys good spectral properties, according to the structural decomposition theorem for Hamiltonian PDEs developed in a previous work~\cite{LZ2022}; (2) the remainder $f(u)$ has the same regularity as in a priori estimates and loses derivatives for each differentiation in $u$. Therefore, a variant of Lyapunov-Perron method works by allowing the topology to vary. A gap between the stable and unstable spectra of $\mathcal{A}(u)$ is still an essential assumption in \cite{shatah2026invariant} (Condition C.4 in Subsection 2.1).

In comparison, the paradifferential method of the present paper provides a systematic procedure of making a more delicate decomposition. Through the paralinearization procedure in Section~\ref{Sec3}, we are able to \emph{gain back regularity} through a quadratically small remainder. See the next subsection for a summary. This refined decomposition also makes it possible to construct hyperbolic invariant submanifolds for systems \emph{with no spectral gap}. Theorem~\ref{Main1} for the infinitely long water jet is a typical example (see Remark \ref{No_Spec_Gap}); generalization to other problems with no spectral gap should follow from the same routine.

\subsection{Idea of the Proof}
We turn to outlining the ideas behind the proof of Theorem~\ref{Main1} and~\ref{Main2}. We summarize this approach as the \emph{paradifferential propagator method}. The strategy to be developed in this paper appears to be quite general and is potentially applicable to a broader class of quasilinear or fully nonlinear evolutionary PDEs. It turns out that the water jet problem is \emph{the one case that rules all} -- each step of the proof can be adjusted to other quasilinear problems without essential change. In this subsection, we describe how the paradifferential propagator method formally works, while outlining its concrete realization for the water jet problem.

\subsubsection{Paralinearization}
The key to the proof of Theorem \ref{Main1} and \ref{Main2} lies in transforming the water jet system (\ref{EQ}) with \emph{paradifferential calculus}, a powerful tool designed to study regularity loss caused by quasilinearity. The transformed system allows a Lyapunov-Perron argument. The nonlinearities can be manipulated as if they were linear: the regularity loss is balanced by gains due to \emph{paralinearization}, while the operators enjoy exactly the same algebraic structure as usual (pseudo)differential operators do. 

Let us formally consider a general autonomous evolutionary problem
\begin{equation}\label{Prototype}
\partial_tu=\N(u)
\end{equation}
on Sobolev spaces $H^s_x$, $s\geq s_0$. Here $s_0$ is a fixed, suitably large index. The spatial dimension could be arbitrary and the domain could be either $\xR^n$ or compact manifolds. The mapping $\N$ might be semilinear, quasilinear or completely nonlinear in $u$, and it is allowed to contain nonlocal derivatives of $u$. 

Let us describe how paradifferential calculus provides a systematic procedure of analyzing the regularity of $\N(u)$. Suppose $0$ is an equilibrium of (\ref{Prototype}). If for each $u\in H^\infty_x$ close to $0$ in the $H^{s_0}_x$ norm, the linearized operator $\N'(u)$ is pseudodifferential (of the usual type (1,0), which is often abbreviated in the literature), then it is possible to define the corresponding \emph{paradifferential operator} $T_{\varsigma[\N'(u)]}$, where $\varsigma[\N'(u)]$ is the symbol of $\N'(u)$. Suppose e.g. it has order $m>0$. The calculus of paradifferential operators yields the following advantages (see Section \ref{Sec:2} for precise account):

\begin{itemize}
\item The bound of $T_{\varsigma[\N'(u)]}:H^s_x\to H^{s-m}_x$ \emph{for any $s\in\xR$} depends only on low regularity of $u$, say $H^{s_0}_x$. 
\item For $s$ suitably large, the paralinearization procedure of \cite{Bony1981} yields the decomposition
$$
\N(u)=T_{\varsigma[\N'(u)]}u+\R(\N;u),
$$
where the remainder $\R(\N;u)$ has regularity $H^{2s-r}_x$ for some fixed $r>0$. In other words, the most irregular part of $\N(u)$ is described by $T_{\varsigma[\N'(u)]}u$.
\item The paradifferential operator $T_{\varsigma[\N'(u)]}$ enjoys exactly the same symbolic calculus of the pseudodifferential operator $\N'(u)$. Therefore, the worst irregularity in $\N(u)$ can be manipulated like linear differential operators, while the remainder gains almost twice the regularity.
\end{itemize}

The first step  then to re-write (\ref{Prototype}) into the paradifferential form
\begin{equation}\label{Prototype_Paradiff}
\partial_tu-T_{\varsigma[\N'(u)]}u=\R(\N;u)
\in H^{2s-r}_x.
\end{equation}
This procedure is usually quite standard in the literature (though not without technicalities), since it is essentially a generalization of the usual paraproduct decomposition (see (\ref{eq:RemPM})).

The paralinearization for water jet system (\ref{EQ}) is done in Section \ref{Sec3}. The paradifferential system is then diagonalized and reduced to a single equation for a complex-valued unknown; see (\ref{EQ_DG}). A truncation is then introduced to extend to the whole space $H^{s_0}$, see (\ref{EQ_Red_Ext}); this resembles the extension argument in Theorem \ref{Stab_Mfd_Finite} as a preparation for the construction of center invariant set.

\subsubsection{Paradifferential Propagator}
This is the key argument. 
Comparing the paradifferential system (\ref{Prototype_Paradiff}) with the classical invariant manifold theorem (Theorem~\ref{Stab_Mfd_Finite}), the major difference lies in the dispersive regime (elliptic directions), which leads to loss of derivatives. However, if for $u$ close to $0$, the operator $\N'(u)$ is \emph{elliptic and anti-self-adjoint (on $L^2_x$)}, then the symbolic calculus for paradifferential operators guarantees \emph{$T_{\varsigma[\N'(u)]}$ to be elliptic and almost anti-self-adjoint (on $L^2_x$)}, in the sense that $T_{\varsigma[\N'(u)]}+T^*_{\varsigma[\N'(u)]}$ is a bounded linear operator on $L^2_x$. 

These features suggest one to study, for any $H^{s_0}_x$-valued function $u:\xR\to H^{s_0}_x$, the linear hyperbolic paradifferential system for any index $s\in\xR$:
\begin{equation}\label{Prototype_Lin}
\partial_tv(t)- \proj_\cen T_{\varsigma[\N'(u(t))]} v(t)=\proj_\cen f(t)\in L^1_{t}H^s_x,\quad v(t_0)\in H^s \text{ given}.
\end{equation}
Here $\proj_\cen$ is the projection to the elliptic directions (center subspace). Paradifferential calculus then enables one to construct the \emph{propagator} $\bm{F}(u;t,t_0):H^s_x\to H^s_x$ for (\ref{Prototype_Lin}), so that the solution to it reads
$$
v(t)=\bm{F}(u;t,t_0)v(t_0)+\int_{t_0}^t\bm{F}(u;t,\tau)f(\tau)\dtau,
$$
verifying energy estimates in $H^s_x$ with coefficients depending \emph{only on the low $H^{s_0}_x$ norm of $u$}. 

For the concrete water jet problem, the construction of this paradifferential propagator $\bm{F}(u;t,t_0)$ is the content of Theorem \ref{Fundamental}. Item \ref{F2} shows how the energy estimate depends only on the low norm of $u$. The quasilinearity is reflected in Item \ref{F3}, where we show that linearization of the propagator in $u$ loses derivatives. This is in contrast to the classical case (\ref{Classical}) as well as semilinear problems, where the propagator does not depend on $u$ at all. In general, if the nonlinearity $\N(u)$ is of order $m$ in $u$, then linearization of $\bm{F}(u;t,t_0)$ in $u$ loses $m$ derivatives, namely $(\der_u\bm{F}\cdot\del u)(u;t,t_0):H^s_x\to H^{s-m}_x$. However, the operator norms still only depend on the $H^{s_0}_x$ norm of $u$.

\subsubsection{Twisted Duhamel Formula, Lyapunov-Perron Method}
The nonlinear system (\ref{Prototype_Paradiff}) becomes
\begin{equation}\label{Prototype_Duhamel}
u(t)=\bm{F}(u;t,t_0)u(t_0)+\int_{t_0}^t\bm{F}(u;t,\tau)\R(\N;u(\tau))\dtau
+\text{solution along hyperbolic directions.}
\end{equation}
We name this as the \emph{twisted Duhamel formula}. If the hyperbolic and elliptic directions of the problem near $0$ are identified, then the Duhamel formula (\ref{Prototype_Duhamel}) will yield integral equations that resemble the ones in Theorem \ref{Stab_Mfd_Finite}. A variant of Lyapunov-Perron method is then available: integral equations for solutions exponentially attracted to $0$ are deduced from (\ref{Prototype_Duhamel}) by pushing $t_0$ to infinity. 

The advantage of the paradifferential propagator method becomes clearer at this stage. Even though linearization in $u$ of the propagator $\bm{F}(u;t,t_0)$ might lose derivatives (i.e. $(\der_u\bm{F}\cdot\del u)(u;t,t_0):H^s_x\to H^{s-m}_x$, but the bounds of this operator depend only on the low $H^{s_0}_x$ norm of $u$), the regularizing nature $\R(\N,u)\in H^{2s-r}_x$ \emph{balances with this loss}, as long as the index $s$ is sufficiently large. This enables one to solve the integral equation using standard implicit function theorems. On the other hand, this delicate decomposition allows the absence of gap between the stable and center-unstable spectra, which was required in most of the literature (see e.g. \cite{LZ2013,shatah2026invariant}).

The concrete argument for water jet problem is presented in Theorem~\ref{GWP_EQ} and in Section~\ref{Sec5}-\ref{Sec6}. The former one focuses on the well-posedness of truncated system~\eqref{EQ_Red_Ext} and Section \ref{Sec5}-\ref{Sec6} provide the Lyapunov-Perron type argument. In particular, behavior along the hyperbolic directions corresponds to lower frequencies and studied in a straightforward manner; see Lemma \ref{exp(tL)}. Equation (\ref{Int_Eq_Stab}) is the integral equation satisfied by exponentially decaying solutions, and (\ref{Int_Eq_Center}) is the one used to construct the center invariant set. The key Proposition \ref{Exp_Contraction} states how the stable manifold arises from implicit function theorem, and the estimate (\ref{Ineq_d_uA_d(u)}) in the proof is where this balance of loss plays a role. It is also noted that the absence of spectral gap does not harm, as can be seen from (\ref{exp(tL)_mu})(\ref{Ineq_duA_su}), since the paradifferential remainders are \emph{both quadratic and smoothing}.

\subsection{Future Perspectives}\label{Perspect}
Theorem \ref{Main1} and \ref{Main2} constitute only one piece of a larger picture. It is noteworthy that, in the dynamics of (\ref{EQ}) for an infinitely long water jet, Theorem \ref{Main1} guarantees the existence of hyperbolic invariant submanifolds, even though \emph{there is no gap between the stable and center-unstable spectra}. The generality of our paradifferential propagator method may allow to construct invariant submanifolds for other quasilinear dispersive PDE systems with this feature.

The diagonalized paradifferential form of water jet system under periodic boundary condition, obtained in Section \ref{Sec3}, sets stage for further investigation into the dynamics on the center invariant set $M_\cen$. There are two types of results to anticipate: special solutions and general dynamics. 

Since we may consider the system under periodic boundary condition as a dispersive PDE when restricted to $M_\cen$, it becomes natural to ask whether periodic or quasiperiodic orbits could exist. There has been a vast literature on the construction of periodic or quasiperiodic free surface waves. Existence of traveling waves for (\ref{EQ}) follows readily from a bifurcation analysis; the method would be largely the same as in, for example, \cite{CN2000,BJL2026}. The existence of periodic standing or quasiperiodic waves is much more difficult. It is usually addressed by KAM or Nash-Moser type iterations; see for example, \cite{PT2001,IPT2005,IP2009,ABHK,AB2015,BBHM-2018,BM2020,BFM2021,FG2024,HHM2025}. However, paradifferential calculus might serve as a feasible alternative, making it possible to construct periodic standing or quasiperiodic waves for (\ref{EQ}) by a standard fixed point method. This has already been addressed as \emph{paradifferential reducibility} for quasilinear hyperbolic PDE systems in \cite{AS2024}. The construction of periodic standing waves for (\ref{EQ}) will be the theme of a forthcoming paper.

On the other hand, the general dynamics of (\ref{EQ}) under periodic boundary condition on the center invariant set is significantly stablized by its oscillatory nature, as seen in Figure \ref{Fig2}. The $1/\varepsilon$ lifespan estimate for an orbit on $M_\cen$ of initial size $\varepsilon$ in Theorem \ref{Main2} is a rough reflection of this feature. Indeed, the \emph{paradifferential descent method} necessary for construction of periodic or quasiperiodic standing waves is also a key step in finding \emph{normal forms} for the problem; see for example \cite{BD2018,BFF2021,BFP2023,BMM2024,BGMS2025} for the discussion on water waves. It is relatively easy to prove that the lifespan on $M_\cen$ is $\gtrsim1/\varepsilon^2$, and it is natural to conjecture the almost global lifespan $\gtrsim_N1/\varepsilon^N$. We will address this issue in a forthcoming paper. 

We end this subsection with a conjecture for (\ref{EQ}) in the infinitely long case. Inspired by the classification results for dynamics of semilinear dispersive PDEs (e.g. Theorem 6.1 of \cite{NS20112} for Klein-Gordon), we propose the following 

\noindent
\textbf{Conjecture.} Consider (\ref{EQ}) for an infinitely long capillary water jet. There are local center-stable and center-unstable \emph{submanifolds} $M_{\st\cen}$ and $M_{\unst\cen}$ for (\ref{EQ}) near the equilibrium. The dynamics of (\ref{EQ}) near the equilibrium is completely classified into several mutually exclusive cases: 
\begin{itemize}
\item The solution stays on $M_{\st\cen}$ or $M_{\unst\cen}$, exists globally along one time direction and develops finite time singularity along the other. 
\item The solution stays on the center manifold $M_{\st\cen}\cap M_{\unst\cen}$ and exists globally along both time directions.
\item The solution develops finite time singularity along both time directions. 
\end{itemize}

The existence of hyperbolic submanifolds, guaranteed by Theorem \ref{Main1}, is the most accessible aspect of the conjecture. To motivate the remaining statements, note that dispersive part of (\ref{EQ}) for an infinitely long capillary water jet enjoys dispersive estimates, yielding $t^{-1/2}$ decay in magnitude, at least at the linear level. The vector field method combined with normal form argument, as in \cite{DIPP2017,IoPu2019}, may be used to study the integral equation (\ref{Int_Eq_Center}) and lead to global-in-time existence. In particular, this could potentially close the fixed point argument to construct $M_{\st\cen}$ and $M_{\unst\cen}$. On the other hand, experimental facts strongly suggest the formation of neck-pinch singularities for (\ref{EQ}) under long wave perturbations. Despite the physical significance and the extensive literature on formal and numerical analysis (e.g. \cite{Eggers2000,EgVi2008,day1998self,ZSL2019}), the mathematical theory remains largely unavailable. Construction of neck-pinch singularity should be a most challenging and exciting project in the study for hydrodynamical free boundary problems.

\section{Toolset of Paradifferential Calculus}\label{Sec:2}
In this section, we review the basic results on \textit{paradifferential calculus}. For the general theory, we refer to Bony's paper \cite{Bony1981}, Meyer's paper \cite{Meyer}, Chapter X of H\"{o}rmander's book \cite{Hormander1997}, or Chapter 4 of M\'{e}tivier's textbook \cite{MePise}; here, we follow the presentation by M\'etivier in \cite{MePise}.

The advantage of applying paradifferential calculus to nonlinear problems is twofold: it preserves all the algebraic structures enjoyed by usual (pseudo)differential operators, while explicitly managing the regularity loss. This enables one to manipulate a nonlinear expression \emph{as if it were linear}. At technical level, paradifferential operators have many good features such as smooth dependence and tame estimates (see Definition~\ref{def:Reg} below). These properties are well-known (so we omit some of the proof), thus usually omitted in many references, while, in this section, we will present them rigorously in a proper framework. 

Throughout the section, the manifold $\xG$ that we are concerned with is the Euclidean space $\xR^n$ or the flat torus $\xT^n \simeq \xR^n/(2\pi\xZ)^n$. The corresponding frequency space $\DuG$ is $\xZ^n$ or $\xR^n$, respectively.

\subsection{Notations and conventions}

We first collect the notations and conventions used in this article.

\begin{notation}[Fr\'{e}chet derivative]
    Given two Banach spaces $X,Y$, open subset $\Omega\subset X$, and an integer $n\in\xN$, we denote by $C^k(\Omega;Y)$ the set of mappings $f:\Omega\to Y$ whose $k$-th derivative $\der^k_x f$ is a continuous map from $\Omega$ to $\mathcal{L}(\otimes^k X;Y)$, where $\der^k_x f$ is defined through iteration: for all $1\leq j \leq k-1$ and $\bm{\delta} x \in X$,
    $$
        \der^{j}_xf(x+\bm{\delta} x) - \der^{j}_xf(x) - \der^{j+1}_xf(x)\cdot\bm{\delta} x = o_{\mathcal{L}(\otimes^jX;Y)}\big( \|\bm{\delta} x\|_{X} \big).
    $$
    We also use the notation $C^\infty(\Omega;Y) := \cap_{k\in\xN} C^k(\Omega;Y)$. In particular, for all $(\bm{\delta} x_1,...\bm{\delta} x_{j+1}) \in \otimes^{j+1} X$, we have the following directional derivative
    $$
        \lim_{\epsilon\to 0} \frac{1}{\epsilon}\Big( \der^j_xf(x+\epsilon\bm{\delta} x_{j+1})\cdot(\bm{\delta} x_1,...\bm{\delta} x_{j}) - \der^j_xf(x)\cdot(\bm{\delta} x_1,...\bm{\delta} x_{j}) \Big) = \der^{j+1}_xf(x)\cdot(\bm{\delta} x_1,...\bm{\delta} x_{j+1}),
    $$
    where the limit exists in the space $Y$. Recall that if a mapping from $\Omega$ to $Y$ is differentiable in all directions $\bm{\delta} x\in X$ and the derivative $x\mapsto \der f(x)$ is continuous, then this mapping belongs to $C^1(\Omega;Y)$.
\end{notation}

\begin{notation}[Multilinear component]
    Given integers $0\leq k < n$ and Banach spaces $X,Y$, we consider a map $\R=\R(u) \in C^n(X;Y)$. Then we denote by $\R^{[k]}=\R^{[k]}(u)$ the $k$-th term in the Taylor expansion of $\R$ at $0$ and by $\R^{[\ge (k+1)]} = \R^{[\ge (k+1)]}(u)$ the corresponding remainder. Namely, we express the Taylor expansion of $\R$ at $0$ as
    $$
        \R(u) = \sum_{j=0}^k \R^{[j]}(u) + \R^{[\ge (k+1)]}(u).
    $$
    Clearly, $\R^{[j]}(u)$ is $j$-multilinear in $u$. Moreover, we denote by $\R^{[\leq k]} = \R^{[\leq k]}(u)$ the sum of terms of order no greater than $k$, namely
    $$
        \R^{[\leq k]} (u) = \sum_{j=0}^k \R^{[j]}(u).
    $$
    In particular, we say that $\R(u)$ decays linearly or quadratically in $u$ (as $u \to 0$), if $\R^{[\leq 1]}=0$ or $\R^{[0]}=0$, respectively.
\end{notation}

\subsection{Preliminaries from Fourier Analysis}\label{subsect-Tool:Pre}

\subsubsection{Functional Spaces and Littlewood-Paley Decomposition}\label{s1}
We introduce the \emph{Littlewood-Paley decomposition} of a tempered distribution on $\xG=\xR^n$ or $\xT^n$, as a standard harmonic analysis construction. 

We represent a tempered distribution $u$ on $\xG^n$ in terms of Fourier transform:
$$
u(x) = \int_{\DuG} e^{i\xi\cdot x} \hat{u}(\xi) \dxi, \quad \hat{u}(\xi)=(2\pi)^{-n}\int_{\DuG}e^{-i\xi\cdot x}u(x)\dx.
$$
Both integrals converge in the sense of tempered distribution.

The Littlewood-Paley decomposition is fixed as follows. Let $\varphi\in C^\infty_0(\xR^n)$ be a function that \emph{has rotation and reflection symmetry}, 
with support in an annulus $\{1/2\leq |\xi|\leq 2\}$, so that
$$
\sum_{j=1}^\infty\varphi(2^{-j}\xi)=1-\chi(\xi),
\quad\mathrm{supp} \chi\subset\{|\xi|<1\}.
$$
We can then decompose any tempered distribution $u$ on $\xG$ as
$$
    u=\Delta_{0}u+\sum_{j\ge1}\Delta_ju,
$$
where
\begin{align*}
    \Delta_0 u(x) =& \chi(D_x)u(x) = (2\pi)^{-n} \int_{\DuG} e^{i\xi\cdot x} \chi(\xi)\hat{u}(\xi) \dxi, \\
    \Delta_j u(x) =& \varphi(2^{-j}D_x)u(x) = (2\pi)^{-n} \int_{\xG} e^{i\xi\cdot x} \varphi(2^{-j}\xi)\hat{u}(\xi) \dxi, \quad j\geq1.
\end{align*}
We then define the \emph{partial sum operator} $S_j$ to be
$$
    S_j:=\sum_{l\leq j}\Delta_l,\quad j\geq0,
$$
while for $j\leq -1$ we just fix $S_j=0$.


For an index $s\in \xR$, the \emph{Sobolev space} $H^s = H^s(\xG)$ consists of those tempered distributions $u$ on $\xG$ such that
$$
    \| u\|_{H^s}:=\left( \int_{\DuG} \big(1+|\xi|^2\big)^s|\hat{u}(\xi)|^2 \dxi \right)^{1/2} <+\infty.
$$
The space $(H^s,\| \cdot\|_{H^s})$ is a Hilbert space. From Plancherel theorem, the Sobolev norm can be characterized through Littlewood-Paley decomposition,
$$
\| u\|_{H^s}^2 \simeq_{s} \sum_{j=0}^{\infty}2^{2js}\| \Delta_j u\|_{L^2}^2.
$$

We will also use \emph{H\"{o}lder spaces} $C^r$ with $r\ge 0$ (also known as \emph{Lipschitz spaces} in the literature) throughout the paper. For integer index~$r\in\xN$, $C^r$ is defined as the usual space of functions on $\xT^n$ whose derivatives of order $\leq r$ are continuous. For non-integer $r>0$, $C^r$ denotes the space of functions whose derivatives up to order $[r]$ are bounded and uniformly H\"older continuous with exponent $r-[r]$. The corresponding H\"{o}lder norm is defined as
$$
    |u|_{C^r}:=|u|_{C^{[r]}} +\sum_{|\alpha|=[r]}\sup_{x,y\in\xG}\frac{\big|\partial^{\alpha}u(x)-\partial^{\alpha}u(y)\big|}{|x-y|^{r-[r]}},
$$
where $[r]$ is the integer part of $r$. The Littlewood-Paley characterization of $C^r$ is as follows: 
$$
|u|_{C^r}\simeq_r
\sup_{j\geq0} 2^{jr}|\Delta_ju|_{L^\infty},
\quad r>0\text{ is not an integer}.
$$
Direct manipulation with series implies $H^s\subset C^{s-n/2}$ for $s>n/2$ with $s-n/2\notin\xN$. Notice that when $r$ is a natural number, the Littlewood-Paley characterization yields the so-called \emph{Zygmund space} $C^r_*$, but we shall not use it within the scope of this paper.

\subsubsection{Regular Mappings}\label{subsubsect:Reg}

As mentioned above, paradifferential operators possess many good properties at technical level. We will describe them as \emph{regular mappings} defined below in this paper.

\begin{definition}[Decreasing Family]
    A family $\{X^s\}_{s\in\xR}$ of semi-normed spaces is said to be \textit{decreasing} if $X^{s_1}\subseteq X^{s_2}$ for all $s_1\ge s_2$.
\end{definition}

\begin{definition}[Regular Mappings]\label{def:Reg}
Consider a non-negative integer $N\in\xN$, a real number $s'_0$ and two decreasing families $\{X^s\}_{s\in\xR}$, $\{Y^s\}_{s\in\xR}$ of normed spaces (e.g. Sobolev spaces). Let $f: u \mapsto f(u)$ be a mapping defined from some open subset $0\in B\subset X^{s'_0}$ to $Y^{s'_0}$. Then we say that $f(u)$ \textit{has $\Ta^N$ dependence on $u$ (from $X^\bullet$ to $Y^{\bullet}$)}, or simply $f(u)$ is \textit{regular (in $u$)}, if there exists $s_0 \ge s'_0$ such that the following properties hold true for all $s \ge s_0$:
\begin{enumerate}
    \item (Smoothness) $f$ belongs to $C^N(B\cap X^{s}; Y^{s})$;
    \item (Tame estimates) For all $u\in B\cap X^{s}$, $n\in\xN$ with $n\leq N$, and $(\bm{\delta} u_1,...\bm{\delta} u_n)\in \otimes^n X^{s}$, there exists a positive non-decreasing function $K_s\colon\xR_+\to\xR_+$ such that
    \begin{equation}\label{eq:Reg-Tame}
        \| \der_u^n f(u) \cdot (\bm{\delta} u_1,...\bm{\delta} u_n) \|_{Y^{s}} \leq  K_s\big( \|u\|_{X^{s_0}} \big) \Big( \|u\|_{X^s} \prod_{l=1}^n \|\bm{\delta} u_l\|_{X^{s_0}} + \sum_\tame \prod_{l=1}^n \|\bm{\delta} u_l\|_{X^{\sigma_l}} \Big),
    \end{equation}
    where the summation is taken over all the couples $(\sigma_1,...\sigma_{n})\in\{s_0,s\}^{n}$, with exactly one component equal to $s$ when $s>s_0$. Moreover, $K_s$ is non-decreasing in $s$.
\end{enumerate}
We denote by $\Ta^N(s_0;X^\bullet,Y^\bullet)$, or simply $\Ta^N(X^\bullet,Y^\bullet)$, the collection of all such mappings $f$. In particular, we use the notation $\Ta^\infty(s_0;X^\bullet,Y^\bullet)$, or simply $\Ta^\infty(X^\bullet,Y^\bullet)$, for the intersection $\cap_{N\in\xN} \Ta^N(s_0;X^\bullet,Y^\bullet)$.
\end{definition}

One can easily check the following properties of the class $\Ta^N$.
\begin{lemma}\label{lem:PropOfReg}
    Let $N\in\xN$ be a non-negative integer, $s_0$ be a real number, and $\{X^s\}_{s\in\xR}$, $\{Y^s\}_{s\in\xR}$, $\{Z^s\}_{s\in\xR}$ be three decreasing families of normed (function) spaces. Then the following properties hold:
    \begin{itemize}
        \item For all positive integer $k\in\xN_+$, and mapping $f = f(u_1,...u_k)$ in the class $\Ta^N(s_0;(X^\bullet)^{\otimes k},Y^\bullet)$, its diagonal $\tilde{f}(u) := f(u,...u)$ belongs to $\Ta^N(s_0;X^\bullet,Y^\bullet)$;
        \item The composition of mappings in $\Ta^N(s_0;X^\bullet,Y^\bullet)$ and $\Ta^N(s_0;Y^\bullet,Z^\bullet)$ is still regular and belongs to the class $\Ta^N(s_0;X^\bullet,Z^\bullet)$.
    \end{itemize}
\end{lemma}

\begin{notation}\label{note:Reg}
    In this paper, we frequently use Definition~\ref{def:Reg} with $X^s$ being some Sobolev spaces or their products. For the sake of simplicity, we introduce the following simplified notations:
    $$
        \Ta^N_{r}(Y^\bullet) = \Ta^N_{r}(s_0;Y^\bullet) := \Ta^N(s_0;H^{\bullet+1}\times H^{\bullet+1/2},Y^\bullet).
    $$
\end{notation}

\subsection{Paradifferential Operators}\label{s2}

\begin{definition}[Symbol class]\label{def:SymCl}
Given~$r\in [0, +\infty)$ and~$m\in\xR$, the symbol class $\Gamma_{r}^{m}$ denotes the space of locally bounded functions~$a(x,\xi)$ on~$\xG\times\DuG$, $C^\infty$ with respect to $\xi \neq 0$, such that
\begin{equation}\label{defi:norms}
\mathcal{M}_{r}^{m}(a)= 
\sup_{|\alpha|\leq 2(n+2) +r}\sup_{|\xi| \ge 1/2~}
\big| (1+|\xi|)^{|\alpha|-m}\partial_\xi^\alpha a(\cdot,\xi)\big|_{C^r_x} < +\infty.
\end{equation}
As we will see later, the low frequency part of the symbol has no contribution in paradifferential operators. Thus, we introduce the equivalent relation
$$
\forall a,b\in\Gamma^m_r, \ a\sim b \Leftrightarrow \  a(\cdot,\xi)=b(\cdot,\xi) \text{ for }|\xi| \ge 1/2.
$$
In this sense, $(\Gamma^m_r,\M^m_r)$ is a Banach space.
\end{definition}

\begin{definition}[Paradifferential Operators]
    Given a symbol~$a\in\Gamma^m_r$, we define the \emph{paradifferential operator} $T_a$ by
    \begin{equation}\label{eq.para}
        \widehat{T_a u}(\xi)=(2\pi)^{-n}\int_{\DuG}
        \chi_0(\xi-\xi',\xi')\widehat{a}(\xi-\xi',\xi')\widehat{u}(\xi') \dxi'.
    \end{equation}
    Here $\widehat{a}(\eta,\xi')$ is the Fourier transform of~$a$ with respect to the spatial variable
    $$
        \widehat{a}(\zeta,\xi')=\int_{\xG} e^{-ix\cdot\eta}a(x,\xi')\dx,
    $$
    and the cut-off function $\chi_0$ is fixed as follows:
    $$
        \chi_0(\zeta,\xi')=\sum_{j\geq0}\left(\sum_{0\leq l \leq j-3}\varphi(2^{-l}\zeta)\right)\varphi(2^{-j}\xi'),
    $$
    with $\varphi$ being the smooth truncation used to define Littlewood-Paley decomposition. It satisfies
    \begin{gather*}
        \chi_0(\zeta,\xi')=1 \quad \text{if}\quad |\zeta|\leq 0.125|\xi'|,\qquad
        \chi(\zeta,\xi')=0 \quad \text{if}\quad |\zeta|\geq 0.5|\xi'|, \\
        |\partial_\zeta^\alpha \partial_{\xi'}^\beta \chi_0(\zeta,\xi')| \leq C_{\alpha,\beta}(1+|\xi'|)^{-|\alpha|-|\beta|} \quad \forall \alpha,\beta\in\xN^n.
    \end{gather*}

    In case the expression for $a$ is lengthy, we shall also use the notation 
    \begin{equation}\label{DefOpPM}
        \Op^\PM(a)=T_a,
    \end{equation}
    where PM is the abbreviation for \emph{paramultiplication}, the terminology used in \cite{Bony1981}. 
\end{definition}

\begin{example}\label{Ex:Paradiff} Some classical operators that can be expressed in terms of paradifferential operators:
    \begin{itemize}
        \item When $a(x,\xi)$ is independent of $\xi$, the paradifferential operator $T_a$ becomes the widely used \emph{paraproduct} operator:
        \begin{equation}\label{T_au}
            T_au=\sum_{j\geq0}S_{j-3}a \Delta_ju.
        \end{equation}
        \item If $a$ is a constant, then $S_{j-3}a = a$ for $j \ge 3$, with $S_{j-3}a = 0$ otherwise, and 
        $$
            T_au=a\sum_{j\geq3}\Delta_ju = a(u - S_2 u).
        $$
        Namely, up to small frequency counterparts, $T_a$ coincides with the multiplication with $a$.
        \item If the symbol $a(x,\xi) = a_0(x)m(\xi)$, then 
        $$
            T_a u = T_{a_0} m(D_x) u = a_0 m(D_x) (u - S_2 u).
        $$
        In particular, for \emph{classical symbols} $a(x,\xi)=\sum_{\alpha}a_\alpha(x)(i\xi)^\alpha$, we have 
        $$
            T_au=\sum_{\alpha}T_{a_\alpha}\partial^\alpha u.
        $$
    \end{itemize}
\end{example}

\subsection{Boundedness and Symbolic Calculus}\label{sec:2.3}
We shall use quantitative results from \cite{MePise} about operator norms estimates in symbolic calculus. The Sobolev spaces $H^s$ are all defined on $\xG$. Since the theory is quite similar either on $\xR^n$ or $\xT^n$, we shall not specify the underlying space in the sequel. 

\begin{convention*}\label{defi:order}
    Let~$m\in\xR$. A linear operator is said to be of order~$m$ if, for all~$s\in\xR$, it is bounded from~$H^{s} $ to~$H^{s-m}$. 
\end{convention*}

The main features of symbolic calculus for paradifferential operators are given by the following propositions, where we can see that the regularity (defined in Definition~\ref{def:Reg}) is always preserved.

\begin{proposition}[Boundedness]\label{prop:PDReg}
    Consider real numbers $s_0,\sigma,m\in\xR$ and a non-negative integer $N\in\xN$. Let $\{X^s\}_{s\in\xR}$ and $\{Y^s\}_{s\in\xR}$ be two decreasing Banach spaces. Assume that $a$ belongs to $\mathcal{T}^N(s_0;X^\bullet,\Gamma^m_0)$ (recall that $\mathcal{M}^m_0$ is defined in \eqref{defi:norms}) and $u$ belongs to $\mathcal{T}^N(s_0;Y^\bullet,H^{\sigma+\bullet})$. Then we have
    $$
     T_a u \in \mathcal{T}^{N}(s_0;X^\bullet \times Y^\bullet,H^{\sigma+\bullet-m}).
    $$
\end{proposition}

\begin{proposition}[Symbolic Calculus]\label{prop:SymReg}
    Consider real numbers $s_0 \ge 0$, $m,m',\sigma\in\xR$, $r>0$, and a non-negative integer $N\in\xN$. Let $\{X^s\}_{s\in\xR}$, $\{Y^s\}_{s\in\xR}$, $\{Z^s\}_{s\in\xR}$ be three decreasing Banach spaces. Assume that symbols $a,b$ belong to $\mathcal{T}^N(s_0;X^\bullet,\Gamma^m_{r})$ and $\mathcal{T}^{N}(s_0;Y^\bullet,\Gamma^{m'}_{r})$, respectively, and that the function $u$ belongs to $\mathcal{T}^N(s_0;Z^\bullet,H^{\sigma+\bullet})$. Then, we have 
    \begin{align*}
        \text{(Composition)} \quad & \big( T_a T_b - T_{a\#_r b} \big) u \in \mathcal{T}^{N}(s_0;X^\bullet \times Y^\bullet \times Z^\bullet,H^{\sigma+\bullet-m-m'+r}), \\
        \text{(Adjoint)} \quad & \big( (T_a)^* - T_{a^{\times;r}} \big) u \in \mathcal{T}^{N}(s_0;X^\bullet \times Z^\bullet,H^{\sigma+\bullet-m+r}).
    \end{align*}
    Here the composition symbol $a\#_r b$ and the adjoint symbol $a^{\times;r}$ are defined respectively by 
    \begin{equation}\label{eq:SymCal}
        a\#_r b= \sum_{|\alpha| < r} \frac{1}{i^{|\alpha|} \alpha !} \partial_\xi^{\alpha} a \partial_{x}^\alpha b, \quad a^{\times;r} = \sum_{|\alpha| < r} \frac{1}{i^{|\alpha|} \alpha !} \partial_\xi^\alpha \partial_x^{\alpha} \bar{a},
    \end{equation}
    which are bilinear in $a,b$, and linear in $a$, respectively. Note that these symbols are still regular, namely 
    $$
        a\#_r b \in \Ta^N(s_0;X^\bullet \times Y^\bullet,\Gamma^{m+m'}_{r-\lceil r \rceil +1}), 
        \quad a^{\times;r} \in \Ta^N(s_0;X^\bullet,\Gamma^{m}_{r-\lceil r \rceil +1}),
    $$
    where $\lceil r \rceil$ is the smallest integer $\geq r$.
\end{proposition}

Another equivalent presentation of Proposition~\ref{prop:PDReg} and ~\ref{prop:SymReg} is as follows: the concerned symbols are regular mappings taking value in symbol classes independent of the index denoted by $\bullet$. That is to say, for arbitrary index $s\in\xR$, we have estimates
\begin{equation}\label{eq:PDSymEsti}
\begin{gathered}
\| T_a u \|_{H^{s-m}} \lesssim \|a\|_{\Gamma^m_0} \|u\|_{H^{s}}, \\
\left\| \big( T_a T_b - T_{a\#_r b} \big) u \right\|_{H^{s-m-m'+r}} \lesssim \|a\|_{\Gamma^m_r} \|b\|_{\Gamma^{m'}_r} \|u\|_{H^{s}}, \\
\left\| \big( (T_a)^* - T_{a^{\times;r}} \big) u \right\|_{H^{s-m+r}} \lesssim \|a\|_{\Gamma^m_r} \|u\|_{H^{s}},
\end{gathered}
\end{equation}
so as their differentiations. 
    
Meanwhile, we observe that the $\#_r$ and $\times;r$ operations are simply cut-offs of the composition and adjoint operations in the symbolic calculus for (1,0) pseudodifferential operators. Hence, we assert that \emph{paradifferential calculus retains the same algebraic structure as standard pseudodifferential calculus, modulo smoothing operators, requiring minimal regularity in the coefficients}. 

\begin{remark}\label{rmk:MultiSymCal}
We have the following corollary: if 
$$
a=\sum_{0\leq j <r}a^{(m-j)}\in \sum_{0\leq j <r} \Gamma^{m-j}_{r-j},\quad 
b=\sum_{0\leq k <r}b^{(m'-k)}\in \sum_{0\leq k <r} \Gamma^{m'-k}_{r-k},
$$
with $m,m'\in\xR$ and $r>0$, then with
$$
    c= \sum_{ |\alpha| +j+k < r} \frac{1}{i^{|\alpha|} \alpha !} \partial_\xi^{\alpha} a^{(m-j)} \partial_{x}^\alpha b^{(m'-k)}, 
    \quad 
    d = \sum_{|\alpha|+j<r} \frac{1}{i^{|\alpha|} \alpha !} \partial_\xi^{\alpha} \partial_{x}^\alpha \overline{a^{(m-j)}}
$$
the operators $T_a T_b -T_{c}$ and $T_a^* - T_{d}$ are of order $m+m'-r$ and $m-r$, respectively. Such argument will be used for the paralinearization of nonlinear operators such as Dirichlet-Neumann operators.
\end{remark}

\begin{remark}[Pluri-homogeneous symbols]\label{rmk:Pluri-homoSym}
In the theory of paradifferential operators, the calculus arises usually when $a^{(m-j)}$'s and $b^{(m'-k)}$'s are homogeneous in $\xi$. This leads to the notion of pluri-homogeneous symbols. We denote by $\dot\Gamma_{r}^{m}$ the space of \emph{homogeneous symbols} of degree $m$, defined as the collection of symbols $a\in\Gamma^m_r$ of the form
$$
a(x,\xi)=|\xi|^ma'\left(x,\frac{\xi}{|\xi|}\right),
$$
where $a'(x,\omega)$ is $C^r$ in $x\in \xT^n$ and smooth in $\omega\in S^{n-1}\subset\xR^n$. Then we define \emph{pluri-homogeneous symbols} of degree $m$ as the symbols of the form
$$
a=\sum_{0\leq j <r}a^{(m-j)}\in \sum_{0\leq j <r}\dot\Gamma^{m-j}_{r-j}.
$$
We say that $a^{(m)}$ is the principal symbol of $a$. Recall from definition~\ref{def:SymCl} that the low frequency part $|\xi|<1/2$ is negligible. The collection of such symbols are denoted by $\Sigma_r^m$. Clearly, Remark~\ref{rmk:MultiSymCal} provides the symbolic calculus for pluri-homogeneous symbols.
\end{remark}

\subsection{Paraproduct and Paralinearization}\label{sec.2.3}
Recall that if $a=a(x)$ is a function of $x$ only, the paradifferential operator $T_a$ is no more than a paraproduct. A key feature of paraproducts is that one can replace 
bilinear expressions by paraproducts up to smoothing operators.

\begin{definition}\label{R_PM}
Given two functions~$a,b$ on $\xT$, we define the paraproduct remainder 
\begin{equation}\label{eq:RemPM}
    R_\PM(a,u) = au-T_a u-T_u a = \sum_{j,k:|j-k|<3}\Delta_j a\Delta_ku.
\end{equation}
Here PM is the abbreviation for \emph{paramultiplication}.
\end{definition}

We record some results on paraproducts; see Chapter 2 of \cite{BCD}, Chapter 2 of \cite{Chemin} or Chapter X of \cite{Hormander1997}. 

\begin{proposition}\label{prop:PMReg}
Consider real numbers $s_0,\sigma,\sigma'\in\xR$ and a non-negative integer $N\in\xN$. Let $\{X^s\}_{s\in\xR}$ and $\{Y^s\}_{s\in\xR}$ be two decreasing Banach spaces. Assume that functions $a,u$ belong to $\mathcal{T}^N(s_0;X^\bullet,H^{\sigma+\bullet})$ and $\mathcal{T}^N(s_0;Y^\bullet,H^{\sigma'+\bullet})$, respectively. If $\sigma_0+\sigma_0'+2s_0>n/2$, then the paraproduct remainder~\eqref{eq:RemPM} satisfies
    $$
    \R_{\PM}(a,u) \in \mathcal{T}^{N}(s_0;X^\bullet \times Y^\bullet,H^{\sigma+\sigma'+s_0+\bullet-n/2}).
    $$
    In particular, if we further have $\sigma=\sigma'$ and $\sigma+s_0>n/2$, then the product $au$ verifies
    $$
        au \in T^N(s_0;X^\bullet \times Y^\bullet,H^{\sigma_0+\bullet}).
    $$
\end{proposition}

This proposition states that in the bilinear expression $B(a,u)=au$, the sum of the paraproducts $T_au+T_ua=T_{B'(a,u)}\cdot(a,u)$ captures the most irregularity. In the same vein, we 
have the following celebrated \emph{paralinearization theorem} due to Bony \cite{Bony1981} (a beautiful proof can be found in \cite{Meyer} or Chapter X of \cite{Hormander1997}, and a quantitative version can be found in \cite{AS2023}). 

\begin{proposition}\label{prop:PLReg}
    Consider integers $k,N\in\xN$ and real numbers $s_0, \sigma$ with $s_0+\sigma> k+n/2$. Let $F\in C^\infty(\xR^{k+1})$ be a fixed smooth function, $\{X^s\}_{s\in\xR}$ be a decreasing family of Banach spaces, and $u$ belong to $\Ta^N(s_0;X^\bullet,H^{\bullet+\sigma})$. Then the composition $F(u,\partial_xu,...\partial_x^ku)$ admits the following paralinearization
    $$
        F(u,\partial_xu,...\partial_x^ku) = F(0) + T_{\nabla F(\eta,\partial_xu,...\partial_x^ku)} \cdot (u,\partial_xu,...\partial_x^ku) + \R_\PL(F;u),
    $$
    where the remainder $\R_\PL(F;u)$, regarded as a mapping of $u$, satisfies
    $$
    \R_\PL(F;u) \in \mathcal{T}^N(s_0;X^{\bullet},H^{\bullet+s_0+2\sigma-2k-n/2}).
    $$
\end{proposition}

Proposition~\ref{prop:PLReg}, along with the symbolic calculus (Proposition~\ref{prop:SymReg}), allows us to handle nonlinear expressions $F(u)$ \emph{as if} they were linear. We first \emph{paralinearize} the expression $F(u)$, then analyze the operator $T_{F'(u)}$, which behaves algebraically like the differential $F'(u)$. Since any potential loss of regularity can be recovered by the remainder $\R_\PL(u)$, there is no concern about losing regularity. As an immediate consequence of Proposition~\ref{prop:PLReg}, we have the regularity of composition, which can also be justified directly.

\begin{proposition}\label{prop:CompReg}
    Under the same assumptions of Proposition~\ref{prop:PLReg}, we have
    $$
    F(u,\partial_xu,...\partial_x^ku) - F(0) \in \mathcal{T}^\infty(s_0;X^{\bullet},H^{\bullet+\sigma-k}).
    $$
    
    Furthermore, suppose $F = F(y;\xi)$ also depends on $\xi$, and there exists a real number $m\in\xR$ such that for all compact set $K\subset\xR^{k+1}$ and indices $k\in\xN$, $\alpha\in\xN^{k+1}$, there holds
    $$
        \sup_{y\in K}\left| \partial_y^\alpha \partial_\xi^k F(y;\xi) \right| \lesssim |\xi|^{m-k}, \quad \forall\ |\xi|> \frac{1}{2},
    $$
    then the symbol $F(u,\partial_xu,...\partial_x^ku;\xi)$ satisfies
    $$
        F(u,\partial_xu,...\partial_x^ku;\xi) \in \mathcal{T}^\infty(s_0;X^{\bullet},\M^m_{\bullet+\sigma-k-n/2}).
    $$
    Recall that the norm $\M^m_{\bullet+\sigma-k-n/2}$ is defined in~\eqref{defi:norms}.
\end{proposition}

\section{Paralinearization of the Full System}\label{Sec3}

In this section, we present the paralinearization of the full system~\eqref{EQ}. The theory is similar for either $\xG=\xR$ or $\xT$, so we shall not make distinction for these two cases. The main difficulty lies in the paralinearization of the Dirichlet-Neumann operator, which will be treated in Subsection~\ref{subsect:PLDtN}. The paralinearization of other terms in~\eqref{EQ} is relatively straightforward, and will be treated in Subsection~\ref{subsect:PLOther}. We then apply a diagonalization procedure, similar to~\cite{HK2023} and~\cite{Yang2024}, to convert the system into a paradifferential equation of a single complex unknown. This will be done in Subsection~\ref{subsect:Diag0}.

There are two major differences between the result presented here and those in~\cite{HK2023} and~\cite{Yang2024}. The first is a refined analysis of the remainders. From the refined results on paradifferential calculus presented in the previous section, we shall justify the following properties of the remainders in the paralinearization formulas : (1) they are smooth in $(\eta,\psi)$; (2) they verify some tame estimates. These properties are crucial for the construction of hyperbolic invariant manifolds $M_\unst,M_\st$ and invariant set $M_\cen$ in Theorem~\ref{Main1} and~\ref{Main2}.

The other difference is the extraction of hyperbolic modes. As introduced in Section~\ref{Sec1}, the behavior of the system~\eqref{EQ} is totally different along hyperbolic and elliptic frequencies. After the paralinearization of nonlinear terms (see~\eqref{EQ_PL}), we shall split the equation into low and high frequencies, which will be treated separately. For low frequencies (containing hyperbolic regime), we will extract directly the linear parts, since the remaining parts are quadratic and supported on bounded frequencies -- therefore infinitely smoothing. For high frequencies (elliptic regime), we will follow similar strategies as in~\cite{HK2023} and~\cite{Yang2024}. Note that, since we aim at a quadratic remainder, the major contribution of the paralinearization of mean curvature $H[\eta]$ is $\xi^2-\rho^{-2}$ instead of $\xi^2$ (the case in~\cite{HK2023} and~\cite{Yang2024}), which requires only a technical difference in the procedure of diagonalization due to the truncation into high frequencies (see Subsection~\ref{subsect:Diag0}).

\subsection{Paralinearization of the Dirichlet-Neumann Operator}\label{subsect:PLDtN}

In this subsection, we will prove a refined version of the paralinearization formula of the Dirichlet-Neumann operator, namely Theorem~\ref{thm-PLDtN:Main}. We recall from Theorem 3.8 of \cite{HK2023} the paralinearization formula: 
\begin{equation}\label{eq-PLDtN:Basic}
   G[\eta]\psi = T_{\tilde{\lambda}} (\psi - T_B\eta) - T_V \partial_x\eta + \tilde{\R}_{\DN}(\eta,\psi),
\end{equation}
where the symbol $\tilde{\lambda}$ has principal part $|\xi|$, and the velocity components $B,V$ are defined by
\begin{equation}\label{DN_V_B}
B = B(\eta,\psi)=\frac{\partial_x\eta\partial_x\psi+G[\eta]\psi}{1+|\partial_x\eta|^2},
\quad
V = V(\eta,\psi)=\partial_x\psi-B\partial_x\eta,
\end{equation}
where $B$ is the velocity along radial direction while $V$ is the velocity along axial direction. The reminder $\tilde{\R}_\DN$ has higher regularity than other terms on the right-hand side. For our purposes, we need to deduce a refined version such that the reminder, together with its derivatives in $(\eta,\psi)$, is smooth in $(\eta,\psi)$ and verifies some tame quadratic estimates. 

Inspired by \cite{wang2017water}, we will repeat the arguments in \cite{HK2023} and improve them in the following ways: $\textbf{(1)}$ the remainder obtained in each step is smooth in $(\eta,\psi)$; $\textbf{(2)}$ these remainders verify some tame estimates, which can be expected from the paradifferential tools introduced in the previous subsection; \textbf{(3)} by adding low order corrections to the symbol $\tilde{\lambda}$, the remainder $\tilde{\R}_{\DN}$ can attain higher regularity. Time dependence is omitted in this subsection since all the arguments do not depend explicitly on time.

\begin{theorem}\label{thm-PLDtN:Main}
Consider a real number $s_0> 2$ and assume that $\eta$ satisfies $\eta+\rho\ge c$ for some constant $c>0$. Then the Dirichlet-Neumann operator $G[\eta]\psi$ admits the following paralinearization :
\begin{equation}\label{eq-PLDtN:Main}
G[\eta]\psi = T_{\lambda} w - \partial_x(T_V \eta) - T_{(\rho+\eta)^{-1}B}\eta + \R_\DN(\eta,\psi).
\end{equation}
Here $w = \psi - T_{B}\eta$ is the good unknown, $B,V$ are defined by \eqref{DN_V_B} above. The symbol $\lambda$ is of order $1$, taking the form 
\begin{equation}\label{eq-PLDtN:Symb}
\begin{aligned}
    &\lambda(x,\xi) = \lambda ^{(1)}(\xi) + \lambda^{(0)}(x,\xi) + \sum_{1\leq j <s_0-3/2} \lambda^{(-j)}(x,\xi), \\
    &\lambda^{(1)}(\xi) = \frac{I_1(\rho|\xi|)}{I_0(\rho|\xi|)}|\xi|, \quad \lambda^{(0)}(x,\xi) = \frac{\eta}{2\rho(\rho+\eta)} - \frac{\partial_x\eta}{2(\rho+\eta)} i \partial_\xi\lambda^{(1)}, \\
    &\lambda^{(-j)}(x,\xi) = f_j(\eta,\partial_x\eta,\cdots,\partial_x^{j+2}\eta)|\xi|^{-j} + ig_j(\eta,\partial_x\eta,\cdots,\partial_x^{j+2}\eta)|\xi|^{-j}\sgn(\xi), \quad 1\leq j < s_0-3/2,
\end{aligned}
\end{equation}
where $f_j$'s and $g_j$'s are real-valued smooth functions with $f_j(0)=g_j(0)=0$. Moreover, the remainder $\R_\DN(\eta,\psi)$, regarded as a mapping of $(\eta,\psi)$, satisfies
$$
    \R_\DN(\eta,\psi) \in \Ta^\infty(s_0;H^{\bullet+1} \times H^{\bullet+1/2},H^{\bullet+s_0-1}), \quad \R_\DN^{[\leq 1]}(\eta,\psi)=0.
$$
\end{theorem}
\begin{proof}
    The proof of the paralinearization formula \eqref{eq-PLDtN:Main} is classical. One can repeat the same arguments in \cite{HK2023} and improve them by using the properties of paradifferential tools introduced in Section~\ref{Sec:2}. The details can be found in Appendix~\ref{App:PLDtN}.
\end{proof}

As a consequence of this theorem, it is clear that the Dirichlet-Neumann operator $G[\eta]\psi$ is regular in $(\eta,\psi)$. The same property holds for $B$ and $V$, since they can be expressed as smooth functions of $G[\eta]\psi$, $\partial_x\psi$, and $\partial_x\eta$ (see \eqref{DN_V_B}).
\begin{proposition}\label{prop-PLDtN:BdDtN}
    Under the assumptions of Theorem~\ref{thm-PLDtN:Main}, the following mappings of $(\eta,\psi)$ verify
    $$
    G[\eta]\psi, B, V \in \Ta^\infty(s_0;H^{\bullet+1} \times H^{\bullet+1/2},H^{\bullet-1/2}).
    $$
\end{proposition}

In the paralinearization formula \eqref{eq-PLDtN:Main} and the following computations, one can see the significance of the good unknown $w$. Therefore, we will turn to set $(\eta,w)$ as the main unknowns instead of $(\eta,\psi)$. Notice that $-T_B\eta$, regarded as a mapping of $(\eta,\psi)$, satisfies $T_B|_{\psi=0}=0$. This fact, together with the regularity of $B$ (Proposition~\ref{prop-PLDtN:BdDtN}) and paraproduct (Proposition~\ref{prop:PDReg}), allows us to apply inverse function theorem and conclude that the mapping $(\eta,\psi)\mapsto(\eta,w)$ is a local diffeomorphism near zero. More precisely, we have the following lemma.

\begin{lemma}\label{lem:ChgtUnk:psi-w}
For any real number\footnote{Later, we may choose the index $s_1$ different from $s_0$ in Theorem~\ref{thm-PLDtN:Main} above.} $s_1> 2$, there exists a constant $\varepsilon>0$ such that, the map
\begin{equation*}
\begin{array}{cccc}
\mathscr{I}_\GU : &H^{s_1+1} \times H^{s_1+1/2} & \mapsto & H^{s_1+1} \times H^{s_1+1/2} \\
&(\eta,\psi) & \mapsto & (\eta,w)= (\eta,\psi-T_{B}\eta)
	\end{array}
	\end{equation*}
	is invertible from a ball centered at zero with radius $\varepsilon$ to its image. Its inverse, denoted by $\mathscr{A}_\GU$, satisfies
    $$
        \mathscr{A}_\GU(\eta,w) \in \Ta^\infty(s_1;H^{\bullet+1} \times H^{\bullet+1/2}, H^{\bullet+1} \times H^{\bullet+1/2}).
    $$
In particular, the velocity components $B,V$ are regular in $(\eta,w)$.
\end{lemma}
\begin{proof}
    The existence of smooth inverse mapping $\mathscr{A}_\GU \in C^\infty(H^{s_1+1} \times H^{s_1+1/2}; H^{s_1+1} \times H^{s_1+1/2})$ is a consequence of inverse function theorem, since $\mathscr{I}_\GU$ is $C^\infty$ with $\der_{(\eta,\psi)}\mathscr{I}_\GU(0,0)=\Id$. It remains to check the regularity of $\mathscr{A}_\GU$. Indeed, if $(\eta,w)\in H^{s+1} \times H^{s+1/2}$ for some $s \ge s_1$, according to the regularity of paradifferential operators (Proposition~\ref{prop:PDReg}), we have
    \begin{equation*}
        \|\psi\|_{H^{s+1/2}} = \| w + T_{B\circ\mathscr{A}_\GU(\eta,w)}\eta \|_{H^{s+1/2}} \lesssim  \| w \|_{H^{s+1/2}} + \| B\circ\mathscr{A}_\GU(\eta,w) \|_{H^{s_1-1/2}} \|\eta\|_{H^{s+1}},
    \end{equation*}
    since we choose $s_1-1/2>1/2$. Then we can apply Proposition~\ref{prop-PLDtN:BdDtN} above to obtain the estimate of $B\circ\mathscr{A}_\GU(\eta,w)$ in $H^{s_1-1/2}$ norm, requiring only $H^{s_1+1} \times H^{s_1+1/2}$ norm of $(\eta,w)$. This proves the tame estimates of $\mathscr{A}_\GU$. As for its derivatives, we can apply the same argument to the following formulas: for integers $k\in\xN_+$,
    \begin{align*}
        \der_\eta^k \psi = \Op^\PM\big(\der_\eta^k (B\circ\mathscr{A}_\GU)\big) \eta 
        + k \Op^\PM\big(\der_\eta^{k-1} (B\circ\mathscr{A}_\GU)\big), 
        \quad \der_w^k \psi = \der_w^k w + \Op^\PM\big(\der_w^k (B\circ\mathscr{A}_\GU)\big) \eta,
    \end{align*}
    which proves that $\mathscr{A}_\GU$ belongs to the desired class $\Ta^\infty(s_1;H^{\bullet+1} \times H^{\bullet+1/2}, H^{\bullet+1} \times H^{\bullet+1/2})$.
\end{proof}

\subsection{Paralinearization of the Full System}\label{subsect:PLOther}
We now turn to the second equation in the full system (\ref{EQ}). Let us compute
\begin{equation}\label{D_tw}
\begin{aligned}
\partial_tw
&=\partial_t\psi-T_{\partial_tB}\eta-T_B\partial_t\eta\\
&\overset{(\ref{EQ})}{=}
-\frac{1}{2}|\partial_x\psi|^2
+\frac{\left(\partial_x\psi\partial_x\eta+G[\eta]\psi\right)^2}{2(1+|\partial_x\eta|^2)}
-H[\eta]+\frac{1}{\rho}
-T_BG[\eta]\psi-T_{\partial_tB}\eta.
\end{aligned}
\end{equation}
Paralinearization of the sum
$$
-\frac{1}{2}|\partial_x\psi|^2
+\frac{\left(\partial_x\psi\partial_x\eta+G[\eta]\psi\right)^2}{2(1+|\partial_x\eta|^2)}
-T_BG[\eta]\psi
$$ 
in (\ref{D_tw}) can be easily deduced from the properties of paraproduct(see Proposition~\ref{prop:PMReg} and Proposition~\ref{prop:PLReg}).

\begin{proposition}\label{PL_Quad}
Consider a real numbers $s_0>2$ and assume that $\eta$ satisfies $\eta+\rho\ge c$ for some constant $c>0$. There holds
$$
-\frac{1}{2}|\partial_x\psi|^2
+\frac{\left(\partial_x\psi\partial_x\eta+G[\eta]\psi\right)^2}{2(1+|\partial_x\eta|^2)}
-T_BG[\eta]\psi
=-T_V\partial_xw - T_{V\partial_xB}\eta +\R_{\PL}^1(\eta,\psi),
$$
where the remainder $\R_{\PL}^1(\eta,\psi)$, regarded as a mapping of $(\eta,\psi)$, satisfies
$$
    \R_{\PL}^1 \in T^\infty(s_0;H^{\bullet+1}\times H^{\bullet+1/2}, H^{\bullet+s_0-3/2}), \quad (\R_{\PL}^1)^{[\leq 1]} = 0.
$$
\end{proposition}
\begin{proof}
Let us begin with some identities which can be obtained directively from the definition~\eqref{DN_V_B} of $B,V$:
\begin{equation*}
    -\frac{1}{2}|\partial_x\psi|^2 +\frac{\left(\partial_x\psi\partial_x\eta+G[\eta]\psi\right)^2}{2(1+|\partial_x\eta|^2)} = \frac{B^2-V^2}{2} - BV\partial_x\eta, \quad G[\eta]\psi = B - V\partial_x\eta,\quad \partial_x\psi = B\partial_x\eta + V.
\end{equation*}
With these identities, we can write
$$
- \frac{1}{2}|\partial_x\psi|^2 + \frac{\left(\partial_x\psi\partial_x\eta+G[\eta]\psi\right)^2}{2(1+|\partial_x\eta|^2)} - T_BG[\eta]\psi
= \frac{B^2-V^2}{2} - BV\partial_x\eta - T_BB + T_B(V\partial_x\eta).
$$
Through paraproduct decomposition (see Definition~\ref{R_PM}), the right-hand side can be written as
\begin{align*}
    R_\PM(B,B) & - T_VV - R_\PM(V,V) - T_{V\partial_x\eta}B - R_\PM(B,V\partial_x\eta) \\
    & = -T_V\partial_x\psi + T_V(B\partial_x\eta) - T_{V\partial_x\eta}B + R_\PM(B,B) - R_\PM(V,V) - R_\PM(B,V\partial_x\eta) \\
    & = -T_V\partial_x\psi + T_V \big( T_B \partial_x\eta + T_{\partial_x\eta}B + R_\PM(B,\partial_x\eta) \big) - T_{V\partial_x\eta}B \\
    &\quad + R_\PM(B,B) - R_\PM(V,V) - R_\PM(B,V\partial_x\eta) \\
    & = -T_V\partial_x\psi + T_VT_B\partial_x\eta + \R_\PL^1(\eta,\psi),
\end{align*}
where the remainder $\R_\PL^1(\eta,\psi)$ reads
$$
\R_\PL^1(\eta,\psi) = (T_VT_{\partial_x\eta} - T_{V\partial_x\eta})B + T_VR_\PM(B,\partial_x\eta) + R_\PM(B,B) - R_\PM(V,V) - R_\PM(B,V\partial_x\eta).
$$
Note that the main term $-T_V\partial_x\psi + T_VT_B\partial_x\eta$ can be written as $-T_V\partial_xw - T_{V\partial_xB}\eta$, due to the definition $w = \psi - T_B\eta$. To conclude the regularity of $\R_\PL^1(\eta,\psi)$, it suffices to apply the regularity of (para)product (see Proposition~\ref{prop:PMReg}), symbolic calculus (see Proposition~\ref{prop:SymReg}) and $B,V$ (see Proposition~\ref{prop-PLDtN:BdDtN}).
\end{proof}

It remains to paralinearize the mean curvature operator $H[\eta]$. Recall that 
\begin{equation}\label{eq:MeanCurv}
    H[\eta] = -\frac{\partial_x^2\eta}{(1+|\partial_x\eta|^2)^{3/2}}+\frac{1}{(\rho+\eta)\sqrt{1+|\partial_x\eta|^2}},
\end{equation}
which is no more than a smooth function of $(\eta,\partial_x\eta,\partial_x^2\eta)$. Hence, we can apply the paralinearization formula (see Proposition~\ref{prop:PLReg}) to obtain its paralinearization with regular remainder.
\begin{proposition}\label{PL_H}
Consider a real number $s_0>2$ and assume that $\eta$ satisfies $\eta+\rho\ge c$ for some constant $c>0$. Define a classical elliptic differential symbol $h = h^{(2)} + h^{(1)} + h^{(0)}$ with
\begin{equation}\label{eq:PL_H_Symb}
    h^{(2)} = \frac{|\xi|^2}{(1+|\partial_x\eta|^2)^{3/2}}, \quad h^{(1)} = \frac{3(\rho+\eta)\partial_x^2\eta-1-|\partial_x\eta|^2}{(\rho+\eta)(1+|\partial_x\eta|^2)^{5/2}}\partial_x\eta i\xi, \quad h^{(0)} = -\frac{1}{(\rho+\eta)^2\sqrt{1+|\partial_x\eta|^2}} + \frac{1}{\rho^2}.
\end{equation}
Then there holds
\begin{equation}\label{eq:PL_H}
    H[\eta] =\frac{1}{\rho} + T_{h}\eta - \frac{\eta}{\rho^2} - (\partial_x^2 + T_{|\xi|^2})\eta +\R_{\PL}^2(\eta),
\end{equation}
where the remainder $\R_{\PL}^2(\eta)$, regarded as a mapping of $\eta$, satisfies
$$
    \R_{\PL}^2 \in \Ta^\infty(s_0;H^{\bullet+1},H^{\bullet+s_0-3/2}), \quad (\R_{\PL}^2)^{[\leq 1]}=0.
$$
\end{proposition}
\begin{proof}
    Note that if one applies directly Proposition~\ref{prop:PLReg} to $H[\eta]$, the remainder only lies in $H^{\bullet+s_0-5/2}$. To improve the regularity of the remainder, we need to treat the term with $\partial_x^2\eta$ separately. First, we apply paraproduct decomposition to write $H[\eta]$ as
    $$
    \begin{aligned}
    -\Op^\PM\big((1+|\partial_x\eta|^2)^{-3/2}\big) \partial_x^2\eta & - \Op^\PM(\partial_x^2\eta) \big( (1+|\partial_x\eta|^2)^{-3/2} \big) \\
    & - R_\PM\big( (1+|\partial_x\eta|^2)^{-3/2}, \partial_x^2\eta \big) + \frac{1}{(\rho+\eta)\sqrt{1+|\partial_x\eta|^2}}.
    \end{aligned}
    $$
    From the regularity of (para)product (see Proposition~\ref{prop:PMReg}), the third term $R_\PM\big( (1+|\partial_x\eta|^2)^{-3/2}, \partial_x^2\eta \big)$ belongs to the desired class $\Ta^\infty(s_0;H^{\bullet+1},H^{\bullet+s_0-3/2})$. Next, we apply paralinearization formula (see Proposition~\ref{prop:PLReg}) to the second term and the last term to obtain that, up to a remainder in the desired class, $H[\eta]$ equals to
    \begin{align*}
        &-\Op^\PM\big((1+|\partial_x\eta|^2)^{-3/2}\big) \partial_x^2\eta + T_{\partial_x^2\eta} \Op^\PM\big(3(1+|\partial_x\eta|^2)^{-5/2}\partial_x\eta\big) \partial_x\eta \\
        &\hspace{12em}- \Op^\PM\big( (\rho+\eta)^{-1} (1+|\partial_x\eta|^2)^{-3/2} \big)\partial_x\eta + \frac{1}{\rho} - T_{h^{(0)-\rho^{-2}}}\eta.
    \end{align*}
    Through the regularity of the composition of paradifferential operators (see Proposition~\ref{prop:SymReg}), the second term, up to some regular remainders, can be replaced by
    $$
        \Op^\PM\big(3(1+|\partial_x\eta|^2)^{-5/2}\partial_x\eta\partial_x^2\eta\big) \partial_x\eta,
    $$
    which gives
    $$
        H[\eta] - \left( \frac{1}{\rho} + T_{h-\rho^{-2}}\eta \right) \in \Ta^\infty(s_0;H^{\bullet+1},H^{\bullet+s_0-3/2}).
    $$
    Recalling from Example~\ref{Ex:Paradiff} that the differences $T_{\rho^{-2}} - \rho^{-2}$ and $\partial_x^2 + T_{|\xi|^2}$ are smoothing operators, we can thus conclude the equality~\eqref{eq:PL_H} with desired regularity. It remains to check that $\R_\PL^2$ decays quadratically. 

    By taking $\eta=0$, the fact $H[0]=\rho^{-1}$ guarantees that $(\R_{\PL}^2)^{[0]}(\eta)=0$. To prove that the linear part $(\R_\PL^2)^{[1]}(\eta)$ also vanishes, we compare the derivative $\der_\eta H[0]$ obtained through the paralinearization formula~\eqref{eq:PL_H} and the definition~\eqref{eq:MeanCurv} of $H[\eta]$,
    \begin{equation*}
        -\partial_x^2\bm{\delta}\eta - \rho^{-2}\bm{\delta}\eta = \der_\eta H[0] \cdot \bm{\delta}\eta = T_{|\xi|^2}\bm{\delta}\eta -\rho^{-2}\bm{\delta}\eta - (\partial_x^2 + T_{|\xi|^2})\del\eta + \der_\eta \R_\PL^2(0) \cdot \bm{\delta}\eta,
    \end{equation*}
    which yields that $(\R_{\PL}^2)^{[1]}(\eta)=0$.
\end{proof}

To sum up, we have managed to reformulate the full system (\ref{EQ}) into a paralinearized one:
\begin{equation}\label{EQ_PL}
\left\{\begin{aligned}
\partial_t\eta
&=T_{\lambda} w - T_V\partial_x\eta - \Op^\PM(\partial_xV+(\rho+\eta)^{-1}B)\eta + \R_\DN(\eta,\psi), \\
\partial_tw
&=-(T_h - \rho^{-2})\eta - (\partial_x^2 + T_{|\xi|^2})\eta -T_V\partial_xw
-\Op^\PM(\partial_tB+V\partial_xB)\eta+\R_\PL(\eta,\psi),
\end{aligned}\right.
\end{equation}
with regular quadratic remainders
\begin{equation*}
    \begin{gathered}
        \R_\DN(\eta,\psi) \in \Ta^\infty_r(s_0;H^{\bullet+s_0-1}), \quad \R_\PL(\eta,\psi) := \R^1_\PL(\eta,\psi)+\R^2_\PL(\eta) \in \Ta^\infty_r(s_0;H^{\bullet+s_0-3/2}), \\
        \R_\DN^{[\leq 1]} = \R_\PL^{[\leq 1]} = 0.
    \end{gathered}
\end{equation*}
See Notation~\ref{note:Reg} for the definition of $\Ta^\infty_r$.

\begin{remark}
    According to Lemma~\ref{lem:ChgtUnk:psi-w}, for small solution, we can replace $\psi$ by the good unknown $w$ through the diffeomorphism defined in Lemma~\ref{lem:ChgtUnk:psi-w}. Then the system~\eqref{EQ_PL} above can be regarded as a system of $(\eta,w)$, with regular remainders
    $$
    \R_\DN\circ\mathscr{A}_\GU(\eta,w) \in \Ta^\infty_r(s_0;H^{\bullet+s_0-1}), \quad \R_\PL\circ\mathscr{A}_\GU(\eta,w) \in \Ta^\infty_r(s_0;H^{\bullet+s_0-3/2}).
    $$
    For simplicity, we will still be using $\psi$ in the following, but it should be considered as a regular mapping of $(\eta,w)$.
\end{remark}

\begin{remark}\label{Time_implicit}
    As a direct consequence of the paralinearization, one can show that
    $$
    \partial_t \eta \in \Ta_r^\infty(s_0;H^{\bullet-1/2}), \quad \partial_t\psi, \partial_t w \in \Ta_r^\infty(s_0;H^{\bullet-1}),
    $$
    where $\partial_t\eta$, $\partial_t\psi$ should be understood as the right-hand sides of the system~\eqref{EQ} and $\partial_tw$ as the right-hand side of the paralinearized equation~\eqref{EQ_PL}.
    
    In particular, since $B = B(\eta,\psi)$ is smooth in $(\eta,\psi)$ (see Proposition~\ref{prop-PLDtN:BdDtN}), $\partial_tB$ is well-defined as a mapping of $(\eta,\psi)$ (thus of $(\eta,w)$), through the chain rule
    $$
    \partial_t B = \der_\eta B(\eta,\psi)\cdot \partial_t\eta + B(\eta,\partial_t\psi),
    $$
    which belongs to the class $\Ta_r^\infty(s_0;H^{\bullet-2})$ for all $s_0>5$.
\end{remark}

Before diagonalizing~\eqref{EQ_PL}, as mentioned in the beginning of this section, we will split the system into high/low frequencies, corresponding to dispersive/hyperbolic regimes. Let us define 
\begin{equation}\label{Low_High}
\proj_{\low}=\text{Fourier projection to frequencies }|\xi|<\frac{2}{\rho},
\quad
\proj_{\high}=\Id-\proj_{\low},
\end{equation}
and set $\1_\low,\1_\high$ as the set indicator functions in frequency space. Applying them to \eqref{EQ_PL},
\begin{equation}\label{EQ_PL_Spl}
    \left\{\begin{aligned}
        \partial_t \proj_{\high}\eta &= T_{\lambda} \proj_{\high} w - T_V\partial_x \proj_{\high}\eta - \Op^\PM(\partial_xV+(\rho+\eta)^{-1}B) \proj_{\high}\eta + \R_\DN^\high(\eta,w), \\
        \partial_t \proj_{\high}w &= - T_{h - \rho^{-2}} \proj_{\high}\eta - T_V\partial_x \proj_{\high}w - \Op^\PM(\partial_tB+V\partial_xB) \proj_{\high}\eta + \R_\PL^\high(\eta,w), \\
        \partial_t \proj_{\low}\eta &= G[0] \proj_{\low} w + \R_\DN^\low(\eta,w), \\
        \partial_t \proj_{\low}w &= -H'[0] \proj_{\low}\eta + \R_\PL^\low(\eta,w),
    \end{aligned}\right.
\end{equation}
where the remainders are regular and quadratic:
\begin{gather*}
    \R_\DN^\high(\eta,w), \R_\DN^\low(\eta,w) \in \Ta^\infty_r(s_0;H^{\bullet+s_0-1}), \quad \R_\PL^\high(\eta,w), \R_\PL^\low(\eta,w) \in \Ta^\infty_r(s_0;H^{\bullet+s_0-3/2}), \\
    (\R_\DN^\high)^{[\leq 1]} = (\R_\DN^\low)^{[\leq 1]} = (\R_\PL^\high)^{[\leq 1]} = (\R_\PL^\low)^{[\leq 1]} = 0.
\end{gather*}
See Notation~\ref{note:Reg} for the definition of $\Ta_r^\infty$. 

Justification of the size and regularity of the low frequency equations is straightforward, since $G[0],H'[0]$ are both Fourier multipliers, while $\proj_{\low}$ is an infinitely smoothing operator and the remainders $\R^\low_{\DN/\PL}$ can absorb all the nonlinear terms except for their linear parts. As for high frequencies, we observe that the projection $\proj_{\high}$ commutes with Fourier multipliers (thus the operators in linearized equations). Meanwhile, from the definition~\eqref{eq.para} of paradifferential operators, the commutators between $\proj_{\high}=\Id-\proj_{\low}$ and paradifferential operators are infinitely smoothing, since the images of these operators consist only of functions with bounded frequencies. Therefore, we can apply the projection $\proj_{\high}$ directly on $(\eta,w)$ on the right-hand side, without changing the regularity and quadratic nature of the remainders.

\subsection{Diagonalization in High Frequencies}\label{subsect:Diag0}

To end this section, we will diagonalize the paralinearized system~\eqref{EQ_PL_Spl} for high frequencies. Namely, we will make a change of unknowns to transform the principal symbol matrix (see (\ref{LEQ(0,0)}))
$$
    \left(\begin{array}{cc}
        0 & \lambda \\
        -h + \rho^{-2} & 0
    \end{array}\right)
    \simeq L(D_x)=
    \left(\begin{array}{cc}
        0 & G[0] \\
        -H'[0] & 0
    \end{array}\right)
$$
into a skew-symmetric one at principal and subprincipal levels. This procedure is standard in the energy estimate of water waves (see for instance~\cite{ABZ2011,ABZ2014} and also~\cite{HK2023,Yang2024}). The main difference here with these works is that we need to consider an extra low order term $\rho^{-2}$ to maintain the quadratic structure of the remainders. 

\subsubsection{Diagonalizing Symbols}
Let us first introduce some operations that are used only within this subsection. Define $\sharp$ and $*$ as the composition and adjoint operation for principal and subprincipal terms (compare \eqref{eq:SymCal}):
\begin{notation}\label{note:2TermComp}
    For real numbers $m,m'\in\xR$ and symbols 
    $$
        a = a^{(m)} + a^{(m-1)} + ... \in \Gamma_{1}^{m}+\Gamma_{0}^{m-1} + \Gamma^{m-2}_0, \quad b = b^{(m')} + b^{(m'-1)} + ... \in \Gamma_{1}^{m'}+\Gamma_{0}^{m'-1}+\Gamma_{0}^{m'-2},
    $$
    write
    \begin{equation*}
        \begin{aligned}
            a\sharp b &:= a^{(m)}b^{(m')} + a^{(m-1)} b^{(m')} + a^{(m)} b^{(m'-1)} + \frac{1}{i}\partial_\xi a^{(m)} \partial_x b^{(m')}, \\
            a^* &:= \overline{a^{(m)}} + \overline{a^{(m-1)}} + \frac{1}{i}\partial_x\partial_\xi \overline{a^{(m)}}.
        \end{aligned}
    \end{equation*}
\end{notation}

Let $\chi:\xR\to[0,1]$ be a smooth bump function that equals 1 within the interval $[-1.5\rho^{-1},1.5\rho^{-1}]$ and vanishes outside of $[-2\rho^{-1},2\rho^{-1}]$. Imitating Section 5 of \cite{HK2023} (see also Section 4 of \cite{Yang2024}), we define a symbol
\begin{equation}\label{gamma}
\gamma:=\gamma^{(3/2)}+\gamma^{(1/2)}
\in \Gamma^{3/2}_1+\Gamma^{1/2}_0,
\end{equation}
where 
\begin{equation}\label{gamma_detail}
\begin{aligned}
\gamma^{(3/2)}
=\sqrt{ \left(\frac{|\xi|^{2}}{(1+|\partial_x\eta|^2)^{3/2}}-\frac{1}{\rho^2}\right)\frac{I_1(\rho|\xi|)}{I_0(\rho|\xi|)}|\xi|} \cdot\big(1-\chi(\xi)\big),
\quad
\gamma^{(1/2)}
= -\frac{i}{2}\partial_x\partial_\xi\gamma^{(3/2)}.
\end{aligned}
\end{equation}
We further define symbols $\mathfrak{p},\mathfrak{q}$ by\footnote{The order of $\mathfrak{p},\mathfrak{q}$ are $1/2$ higher than \cite{HK2023,Yang2024} due to technical reasons, while this difference has no impact in the diagonalization procedure.}
\begin{equation}\label{eq-diag:Symb}
\begin{aligned}
\mathfrak{q}
&= \sqrt{(\rho+\eta) |\xi|\frac{I_1(\rho|\xi|)}{I_0(\rho|\xi|)}}\big(1-\chi(\xi)\big)\in \Gamma^{1/2}_1,  \\
\mathfrak{p} 
&= \gamma \sharp \mathfrak{q} \sharp \lambda^\textrm{inv}, \quad 
\text{with }\lambda^\textrm{inv} = \frac{\big(1-\chi(\xi)\big)}{\lambda^{(1)}} - \frac{\lambda^{(0)}\big(1-\chi(\xi)\big)}{(\lambda^{(1)})^2}
\in \Gamma^{-1}_1+\Gamma^{-2}_0.
\end{aligned}
\end{equation}

It is easy to check that, when $|\xi|\geq 2\rho^{-1}$, the principal symbol of $\mathfrak{p}$ equals
$$
\mathfrak{p}^{(1)} = \frac{\gamma^{(3/2)}\mathfrak{q}^{(1/2)}}{\lambda^{(1)}} 
= \sqrt{(\rho+\eta) \left(\frac{|\xi|^{2}}{(1+|\partial_x\eta|^2)^{3/2}}-\frac{1}{\rho^2}\right)}.
$$
Straightforward computation shows the following: 

\begin{lemma}\label{lem-diag:Symb}
Fix an index $s_0>5/2$. Suppose $\eta$ is close to 0 in $H^{s_0}$ norm, so that $\eta+\rho\geq c$ for some constant $c>0$, while $|\eta|_{C^2}\ll1$. Then the symbols $\mathfrak{p},\mathfrak{q},\gamma$, which depend only on $\eta$ and its derivatives, have the regularity
    $$
        a^{(m-j)} \in \Ta^\infty(s_0;H^{\bullet+1},\M^{m-j}_{\bullet-j-1/2}), \quad (a,m)\in\{ (\mathfrak{p},1), (\mathfrak{q},1/2), (\gamma,3/2) \},\ j=0,1,
    $$
    Moreover, with $\lambda$ being as in (\ref{eq-PLDtN:Symb}) and the operations $\sharp$, $*$ as in Notation~\ref{note:2TermComp}, there holds
    \begin{gather*}
        \mathfrak{p}\sharp\lambda - \gamma\sharp\mathfrak{q} = 0, \quad \gamma\sharp\mathfrak{p} - \mathfrak{q}\sharp (h-\rho^{-2}) = 0, \quad \gamma = \gamma^*,
        \quad \text{when }|\xi|\geq\frac{2}{\rho}.
    \end{gather*}
\end{lemma}

Applying $\proj_{\high} T_{\mathfrak{p}}$ to the first equation in~\eqref{EQ_PL_Spl} and $\proj_{\high} T_{\mathfrak{q}}$ to the second one, we obtain the evolution of $T_\mathfrak{p}\eta$ and $T_\mathfrak{q} w$ in high frequencies:
\begin{equation}\label{eq-diag:0}
\left\{\begin{aligned}
\partial_t(\proj_{\high} T_\mathfrak{p}\proj_{\high}\eta) 
 - \proj_{\high} T_\gamma T_\mathfrak{q} \proj_{\high}w 
+ \proj_{\high} T_V\partial_x(T_\mathfrak{p}\proj_{\high}\eta) 
& = \mathcal{A}_1(\eta,w),\\
\partial_t(\proj_{\high} T_\mathfrak{q}\proj_{\high}w) 
+ \proj_{\high} T_\gamma T_\mathfrak{p}\proj_{\high}\eta + \proj_{\high} T_V\partial_x(T_\mathfrak{q}\proj_{\high}w) 
& = \mathcal{A}_2(\eta,w),
\end{aligned}\right.
\end{equation}
where
$$
\begin{aligned}
\mathcal{A}_1(\eta,w)
&= \proj_{\high} \Big(-T_\mathfrak{p}\Op^\PM(\partial_xV+(\rho+\eta)^{-1}B)\proj_{\high}\eta  
+ T_\mathfrak{p}\R^\high_\DN(\eta,w) \\
&\quad + T_{\partial_t\mathfrak{p}}\proj_{\high}\eta + [T_V\partial_x,T_\mathfrak{p}]\proj_{\high}\eta + ( T_\mathfrak{p} T_\lambda - T_\gamma T_\mathfrak{q} )\proj_{\high}w\Big),\\
\mathcal{A}_2(\eta,w)
&=\proj_{\high}\Big(T_\mathfrak{q}\Op^\PM(\partial_tB+V\partial_xB)\proj_{\high}\eta 
+ T_\mathfrak{q}\R_\PL^\high(\eta,w) \\
&\quad +T_{\partial_t\mathfrak{q}}\proj_{\high}w + [T_V\partial_x,T_\mathfrak{q}]\proj_{\high}w + ( T_\gamma T_\mathfrak{p} - T_\mathfrak{q} T_{h-\rho^{-2}} )\proj_{\high}\eta\Big).
\end{aligned}
$$
For the right-hand side of the diagonalized system (\ref{eq-diag:0}), we have the following lemma:
\begin{lemma}\label{lem-diag:HiFre-Rem}
Fix a real index $s_0>5$. Suppose $(\eta,w)\in C^0_tH^{s_0+1}_x\times C^0_tH^{s_0+1/2}_x$ solves the paralinearized system (\ref{EQ_PL}). Then the right-hand side of \eqref{eq-diag:0} can be re-organized as
\begin{equation}\label{eq-diag:HiFre-Rem}
\begin{aligned}
\mathcal{A}_1(\eta,w) 
&= \proj_{\high}\big(A_1^1\size[D_x]\eta + A_1^2 \size[D_x]^{1/2}w\big) + \R_\DG^1(\eta,w), \\
\mathcal{A}_2(\eta,w)
&= \proj_{\high}\big(A_2^1\size[D_x]\eta + A_2^2\size[D_x]^{1/2}w\big) + \R_\DG^2(\eta,w).
\end{aligned}
\end{equation}
Here in (\ref{lem-diag:HiFre-Rem}), the bounded linear operators $A_j^k$ depend smoothly on\footnote{We can conclude from Proposition \ref{prop:SymReg} that these operators $A_j^k$ are indeed all paradifferential operators; however, this fact is not needed in this paper.} $(\eta,w)$: for any index $s\in\xR$, there holds
\begin{equation}\label{eq-diag:HiFre-RemSymb} 
A_j^k \in C^\infty\big(H^{s_0+1}\times H^{s_0+1/2};\mathcal{L}(H^s)\big), \quad A_j^k\big|_{(\eta,w)=0} = 0.
\end{equation}
The remainders $\R_\DG^{j}$ are regular in $(\eta,w)$, satisfying
\begin{equation}\label{eq-diag:HiFre-RemReg}
\R_\DG^{j} \in \Ta_r^\infty(s_0;H^{\bullet+s_0-2}), \quad (\R_\DG^{j})^{[\leq 1]} = 0, \quad j=1,2.
\end{equation}
\end{lemma}

\begin{proof}
The result follows easily from the symbolic calculus of paradifferential operators in Proposition \ref{prop:SymReg}. As each terms in the expression of $\mathcal{A}_{1}(\eta,w),\mathcal{A}_{2}(\eta,w)$ are handled quite similarly, we shall only write down how to analyze the typical ones among them. 
    
For example, let us look at the operator
$$
\proj_{\high} T_\mathfrak{p}\Op^{\PM}(\partial_xV+(\rho+\eta)^{-1}B)\proj_{\high}
=\left(\proj_{\high} T_\mathfrak{p}\Op^{\PM}(\partial_xV+(\rho+\eta)^{-1}B)\size[D_x]^{-1}\proj_{\high}\right)\size[D_x]
$$ 
appearing in the expression of $\mathcal{A}_1(\eta,w)$. Proposition~\ref{prop-PLDtN:BdDtN} and Lemma \ref{lem:ChgtUnk:psi-w} show that $B\in H^{s_0+1/2}_x,V\in H^{s_0-1/2}_x$ and depend on $(\eta,w)\in H^{s_0+1}_x\times H^{s_0+1/2}_x$ smoothly, vanishing linearly as $(\eta,w)\to0$. By Proposition \ref{prop:PDReg}, the operator
$$
\proj_{\high} T_\mathfrak{p}\Op^{\PM}(\partial_xV+(\rho+\eta)^{-1}B)\size[D_x]^{-1}\proj_{\high}
$$
is of order zero (in $H^s$ for every $s\in\xR$) , and obviously the coefficients of these paradifferential operators depend on $(\eta,w)\in H^{s_0+1}_x\times H^{s_0+1/2}_x$ smoothly. Therefore it is a smooth map from $(\eta,w)\in H^{s_0+1}_x\times H^{s_0+1/2}_x$ to $\mathcal{L}(H^s)$ for every $s\in\xR$. A similar argument applies to the operators 
$$
\proj_{\high}[T_V\partial_x,T_{\mathfrak{p}}]\size[D_x]^{-1}\proj_{\high},\quad
\proj_{\high}[T_V\partial_x,T_{\mathfrak{q}}]\size[D_x]^{-1/2}\proj_{\high}.
$$

Let us turn to symbols involving time derivatives; the key here is to use the equation (\ref{EQ_PL}) itself to express the time derivatives $\partial_t\eta,\partial_tw$ in terms of $\eta,w$ (and their spatial derivatives). For example, in the term
$$
\proj_{\high} T_{\partial_t\mathfrak{p}}\proj_{\high} \eta
=\proj_{\high} T_{\partial_t\mathfrak{p}}\size[D_x]^{-1}\proj_{\high}\size[D_x]\eta,
$$
the symbol  
$$
\partial_t\mathfrak{p}
=\der_\eta \mathfrak{p} \cdot\partial_t\eta
\overset{\eqref{EQ_PL}}{=}
\der_\eta \mathfrak{p} \cdot G[\eta]\psi.
$$
Recalling the definition (\ref{eq-diag:Symb}), we find that the paradifferential operator $T_{\partial_t\mathfrak{p}}$ is of order 1, with coefficients being smooth functions of $\eta,G[\eta]\psi$ and their spatial derivatives of order $\leq3$. On the other hand, $T_{\partial_t\mathfrak{p}}$ does when $(\eta,w)=0$, since the factor $\eta_t=G[\eta]\psi$ is linear in $\psi$ (thus in $w$). Therefore, we conclude that $(\eta,w)\mapsto T_{\partial_t\mathfrak{p}}\size[D_x]^{-1}\proj_{\high}$ is smooth from $(\eta,w)\in H^{s_0+1}_x\times H^{s_0+1/2}_x$ to $\mathcal{L}(H^s)$ for every $s\in\xR$. A similar argument applies to the term $\proj_{\high} T_\mathfrak{q}\Op^\PM(\partial_tB+V\partial_xB)\proj_{\high}\eta $ and the term $\proj_{\high} T_{\partial_t\mathfrak{q}}\proj_{\high} w$.

For the term $\proj_{\high}(T_\mathfrak{p} T_\lambda - T_\gamma T_\mathfrak{q} )\proj_{\high}w$, we need the composition formula (see Proposition~\ref{prop:SymReg}) for the principal and subprincipal parts separately:
$$
\sum_{j,j'\in\{0,1\}} \big( \mathfrak{p}^{(1-j)} \#_{r_{j,j'}} \lambda^{(1-j')} - \gamma^{(3/2-j)} \#_{r_{j,j'}} \mathfrak{q}^{(1/2-j')} \big), \quad r_{j,j'}=s_0-j-j'.
$$
According to Lemma~\ref{lem-diag:Symb}, the principal and the subprincipal parts of this symbol (namely, the symbols of order $2$ and $1$) vanish. Meanwhile, each term of order $\leq 0$ has at least one of the following factors: (1) $\lambda^{(0)}$; (2) $\gamma^{(1/2)}$; (3) derivative in $x$ of $\lambda$ or $\mathfrak{q}$ (see the symbolic calculus formula in Proposition~\ref{prop:SymReg}). The first two factors has linear decay in $\eta$ due to their definition~\eqref{eq-PLDtN:Symb} and~\eqref{eq-diag:Symb}, while the last one will bring about a factor $\partial_x^j\eta$ for some $j\in\xN_+$. Therefore, the symbol above is of order zero and has at least linear decay in $(\eta,w)$ as $(\eta,w)\to 0$. A similar argument applies to the term $\proj_{\high}(T_\gamma T_\mathfrak{p} - T_\mathfrak{q}T_{h-\rho^{-2}} )\proj_{\high}\eta$, showing that these last two terms are indeed order zero operators acting on $\eta,w$. 

Finally, we simply absorb $\proj_{\high} T_\mathfrak{p}\R^\high_\DN(\eta,w)$, $\proj_{\high} T_\mathfrak{q}\R^\high_\PL(\eta,w)$ into the remainders $\R_\DG(\eta,w)$. They are quadratic in $\eta,w$ and of desired regularity in $x$ by (\ref{EQ_PL_Spl}). This completes the proof.
\end{proof}

In conclusion, the equations for high frequencies read
\begin{equation}\label{eq-diag:high}
\left\{\begin{aligned}
\partial_t(\proj_{\high} T_\mathfrak{p}\proj_{\high}\eta) 
& - \proj_{\high} T_\gamma T_\mathfrak{q} \proj_{\high}w + \proj_{\high} T_V\partial_x(T_\mathfrak{p}\proj_{\high}\eta) \\ 
&= \proj_{\high}\big(A_1^1\size[D_x]\eta + A_1^2 \size[D_x]^{1/2}w\big) + \R_\DG^1(\eta,w),\\
\partial_t(\proj_{\high} T_\mathfrak{q}\proj_{\high}w) 
& + \proj_{\high} T_\gamma T_\mathfrak{p}\proj_{\high}\eta + \proj_{\high} T_V\partial_x(T_\mathfrak{q}\proj_{\high}w) \\
&= \proj_{\high}\big(A_2^1\size[D_x]\eta + A_2^2\size[D_x]^{1/2}w\big) + \R_\DG^2(\eta,w),
\end{aligned}\right.
\end{equation}
where the linear operators $A_j^k$ and the remainders $\R_\DG^j$ verify~\eqref{eq-diag:HiFre-RemSymb} and~\eqref{eq-diag:HiFre-RemReg}, respectively.

\subsubsection{Complex formulation}
We now encode the unknown $(\eta,w)$ into a complex-valued function. Define
\begin{equation}\label{eq-diag:ComplexUnk}
\begin{aligned}
u &= \proj_{\low}(\eta+iw) 
+ \proj_{\high}T_\mathfrak{p}\proj_{\high}\eta
+i\proj_{\high}T_{\mathfrak{q}}\proj_{\high}w.
    \end{aligned}
\end{equation}
We claim that it is equivalent to the unknown $(\eta,w)$ in the following sense:
\begin{lemma}\label{DG_Transform}
For an arbitrary index\footnote{Later, we may choose $s_1$ different from the index $s_0$.} $s_1>5/2$, the mapping 
\begin{equation*}
\begin{array}{cccc}
\mathscr{I}_\DG : &H^{s_1+1}(\xG;\xR) \times H^{s_1+1/2}(\xG;\xR) & \mapsto & H^{s_1}(\xG;\xC) \\
& (\eta,w) & \mapsto & u
\end{array}
\end{equation*}
defined by (\ref{eq-diag:ComplexUnk}) realizes a smooth local diffeomorphism near zero. Its inverse $\mathscr{A}_\DG$ is regular in $u$:
$$
\mathscr{A}_\DG \in \Ta^\infty(s_1;H^{\bullet},H^{\bullet+1} \times H^{\bullet+1/2}).
$$

Moreover, for any index $s_0 \ge s_1$, the map $\mathscr{A}_\DG(u)$ can be paralinearized as follows:
\begin{equation}\label{eq-diag:ParalinTrans}
\begin{aligned}
\eta  = T_{\mathfrak{p}_\inv} \RE u + \R_\inv^1(u),
\quad
w  = T_{\mathfrak{q}_\inv} \IM u + \R_\inv^2(u).
\end{aligned}
\end{equation}
Here the symbols $\mathfrak{q}_\inv$ and $\mathfrak{p}_\inv$ take the form
\begin{equation*}
\begin{aligned}
\mathfrak{q}_\inv &= \sum_{\substack{k\in\xZ+1/2 \\ 1/2 \leq k < s_0}} \left( \mathfrak{q}_{\inv ,1}^{(-k)}(x)|\xi|^{-k} + \mathfrak{q}_{\inv ,2}^{(-k)}(x)i\sgn(\xi)|\xi|^{-k} \right), \\
\mathfrak{p}_\inv &= \sum_{\substack{k\in\xZ \\ 1 \leq k < s_0+1/2}} \left( \mathfrak{p}_{\inv ,1}^{(-k)}(x)|\xi|^{-k} + \mathfrak{p}_{\inv ,2}^{(-k)}(x)i\sgn(\xi)|\xi|^{-k} \right),
\end{aligned}
\end{equation*}
where the functions $\mathfrak{q}^{(-k)}_{\inv, j}$'s and $\mathfrak{p}^{(-k)}_{\inv, j}$'s are regular in $u$:
$$
\mathfrak{q}^{(-k)}_{\inv, j} \in \Ta^\infty(s_0;H^{\bullet},H^{\bullet-k+1/2}), \quad \mathfrak{p}^{(-k)}_{\inv, j} \in \Ta^\infty(s_0;H^{\bullet},H^{\bullet-k+1}).
$$
The remainder $\R_\inv^j(u)$'s are regular in $u$, vanishing linearly as $u\to0$:
$$
\R_\inv^1 \in \Ta^\infty(s_0;H^{\bullet},H^{\bullet+s_0}), \quad \R_\inv^2 \in \Ta^\infty(s_0;H^{\bullet},H^{\bullet+s_0+1/2}).
$$
\end{lemma}

\begin{proof}
By (\ref{eq-diag:Symb}), the symbols $\mathfrak{p},\mathfrak{q}$ involve up to 3 spatial derivatives of $\eta$. By Proposition \ref{prop:SymReg}, the mapping $\mathscr{I}_\DG$ is a smooth map. Furthermore, by (\ref{eq-diag:Symb}), the operators $T_{\mathfrak{p}},T_{\mathfrak{q}}$ shall become simple Fourier multipliers when $\eta=0,w=0$. Thus the linearization of (\ref{eq-diag:ComplexUnk}) at $(\eta,w)=(0,0)$ is the Fourier multiplier
$$
\begin{aligned}
\left(\proj_{\low}+\sqrt{|D_x|^2-\rho^{-2}}\proj_{\high}\right)\RE
+i\left(\proj_{\low}+\sqrt{\lambda^{(1)}(D_x)}\proj_{\high}\right)\IM,
\end{aligned}
$$
which is a linear isomorphism from $H^{s_1+1}(\xG;\xR) \times H^{s_1+1/2}(\xG;\xR)$ to ${H}^{s_1}(\xG;\xC)$. Note that the Fourier multiplier never vanishes in both the real and imaginary part, since the projection $\proj_{\high}$ keeps only frequencies $\geq 2\rho^{-2}$, avoiding the zeros of $\sqrt{\xi^2-\rho^{-2}}$ and $\lambda^{(1)}(\xi)$. Then the existence of inverse mapping $\mathscr{A}_\DG$ at index $s_1$ follows from inverse function theorem. Now, we will first prove the paralinearization~\eqref{eq-diag:ParalinTrans} and then deduce the regularity of $\mathscr{A}_\DG$ iteratively.

The constructions of the symbols $\mathfrak{q}_\inv,\mathfrak{p}_\inv$ indeed follow from a rather standard ``elliptic parametrix argument" -- see for example Theorem 18.1.9 of \cite{Hormander-III}; although the statement was made for pseudodifferential operators in the literature, the argument remains largely unchanged since paradifferential operators enjoy exactly the same symbolic calculus. The key here is to observe that $\mathfrak{p}\geq c|\xi|$ and $\mathfrak{q}\geq c|\xi|^{1/2}$ for $|\xi|\gg1$. We write down the details for the construction of $\mathfrak{q}_\inv$; the argument for $\mathfrak{p}_{\inv}$ follows in the same way.

Our goal is to construct a symbol 
$$
\mathfrak{q}_{\inv} = \sum_{\substack{k\in\xZ+1/2 \\ 1/2 \leq k < s_0-1/2}} \mathfrak{q}_{\inv}^{(-k)}(x,\xi)
$$ 
under the condition $\eta\in H^{s_0+1}$ with $s_0 \ge s_1$ such that 
$$
T_{\mathfrak{q}_{\inv}} T_{\mathfrak{q}} = \Id + \text{operators of order $\leq (-s_0+1/2)$}.
$$
Through symbolic calculus (see Proposition~\ref{prop:SymReg}) \footnote{This also guarantees that the operator norm of remainders depends only on $H^{s_0+1}$ norm of $\eta$.}, the composition of $T_{\mathfrak{q}_{\inv}}$ and $T_{\mathfrak{q}}$ equals
$$
\Op^\PM\sum_{\substack{k\in\xZ+1/2 \\ 1/2 \leq k < s_0}} \mathfrak{q}_{\inv}^{(-k)} \#_{s_0-k} \mathfrak{q} 
,
$$
up to some operators of order $(1-s_0)$. Our purpose is to choose $\mathfrak{q}_\inv$ so that the symbol above equals $1$ (with remainder of order $\leq -s_0+1/2$). Comparing symbols of different orders, the symbol $\mathfrak{q}_{\inv}$ is then uniquely determined by the following iteration procedure: at the highest order,
$$
\mathfrak{q}^{(-1/2)}_\inv = \frac{1}{\mathfrak{q}};
$$
for all integer $0< m < s_0-1/2$, the symbol $\mathfrak{q}_\inv^{(-m-1/2)}$ is explicitly given by 
\begin{equation*}
\begin{aligned}
\mathfrak{q}_\inv^{(-m-1/2)} 
=& - \frac{1}{\mathfrak{q}}\sum_{\substack{k\in\xZ+1/2 \\ 1/2 \leq k \leq m+1/2}} \frac{1}{(m-k+1/2)!} \partial_\xi^{m-k+1/2}\mathfrak{q}_{\inv}^{(-k)}  D_x^{m-k+1/2}\mathfrak{q}.
\end{aligned}
\end{equation*}
Note that $\mathfrak{q} \ge c|\xi|^{1/2}$ for some constant $c>0$. From this formula, it is clear that $\mathfrak{q}_\inv^{(-m-1/2)}$ depends smoothly on $\eta$ and its derivatives up to order $(m+1)$. Finally, one observes that, in the right-hand side of the formula above, the dependence on $\xi$ only comes from modified Bessel functions and its derivatives. Thus, through the asymptotic expansion of these functions as $|\xi| \to \infty$, we can reformulate $\mathfrak{q}_\inv^{(-m-1/2)}$ as some functional of $\eta$ and its derivatives up to order $(m+1)$, multiplies with $|\xi|^{-m-1/2}$ or $i\sgn(\xi)|\xi|^{-m-1/2}$. This is exactly the form of $\mathfrak{q}_\inv$ we are looking for.

With the $\mathfrak{q}_\inv$ constructed as above and $s_0=s_1$, we obtain the following relation between $(\eta,w)$ and $u$:
\begin{equation}\label{eq-diag:qqinvRel}
    \proj_\high w = T_{\mathfrak{q}_\inv} \proj_\high \IM u + \Big( \text{operators of order $\leq (-s_1+1/2)$} \Big) w.
\end{equation}
Regarding the last $w$ as the second component of $\mathscr{A}_\DG(u)$, we prove inductively that for all $s\ge s_1$, $u \in H^{s}$ implies $w\in H^{s+1/2}$ (and $\eta\in H^{s+1}$ in the same way). Moreover, through the regularity of paradifferential operators (see Proposition~\ref{prop:PDReg}), we can also prove the tame estimates of $\mathscr{A}_\DG$, since $T_{\mathfrak{q}_\inv}$ and the remainders have operator norm bounded by $H^{s_1+1}$ norm of $\eta$. As for the derivatives of $\mathscr{A}_\DG$, we can apply derivatives in $u$ to the relation~\eqref{eq-diag:qqinvRel} above and utilize again the iterative argument. This completes the proof of desired regularity $\mathscr{A}_\DG \in \Ta^\infty(s_1;H^{\bullet},H^{\bullet+1} \times H^{\bullet+1/2})$.
\end{proof}

\begin{remark}\label{Complex_u}
Clearly the complex unknown $u$ equals, up to a smooth quadratic remainder, 
$$
\left(\proj_{\low}+\sqrt{|D_x|^2-\rho^{-2}}\proj_{\high}\right)\eta
+i\left(\proj_{\low}+\sqrt{\lambda^{(1)}(D_x)}\proj_{\high}\right)\psi.
$$
Therefore, at the linear level, the low frequency equation in (\ref{EQ_PL_Spl}) is nothing but the right-hand side of (\ref{LEQ(0,0)}) projected to $|\xi|\leq2\rho^{-1}$. From now on, we shall not distinguish the complex unknown $u$ and the $\xR^2$-valued unknown $(\RE u,\IM u)$. We keep the complex unknown only to simplify notation.
\end{remark}

Combing the diagonalized equation \eqref{eq-diag:high}  and the low-frequency equation in \eqref{EQ_PL_Spl}, we arrive at the evolution equation for $u$:
\begin{equation}\label{eq-diag:fin}
\begin{aligned}
\partial_t u& - \Big(G[0]\proj_{\low}\IM u
-iH'[0]\proj_{\low}\RE u\Big) 
+i\proj_{\high} T_{\gamma}\proj_{\high}u
+\proj_{\high} T_V\partial_x \proj_{\high} u\\
&= \proj_{\high}(A_2^1 + iA_1^1)\size[D_x]\eta +  \proj_{\high}(A_2^2 + iA_1^2)\size[D_x]^{1/2}w + \R_\DG^3(\eta,w).
\end{aligned}
\end{equation}
Here the linear operators $A_j^k$ verify (\ref{eq-diag:HiFre-RemSymb}); the remainder $\R_\DG^3(\eta,w)$ collects all the remainders coming from the right-hand side of (\ref{EQ_PL_Spl}) and (\ref{eq-diag:high}):
\begin{equation}\label{R_DG3}
\begin{aligned}
\R_\DG^3(\eta,w)
&=\R_\DN^{\low}(\eta,w)-i\R_\PL^{\low}(\eta,w)
+\R_\DG^1(\eta,w)-i\R_\DG^2(\eta,w)\\
&\quad-i\proj_{\high}T_\gamma\proj_{\low}(\eta+iw)
-\proj_{\high} T_V\partial_x \proj_{\low}(\eta+iw).
\end{aligned}
\end{equation}
From \eqref{eq-diag:HiFre-RemReg} we know that the first four terms in (\ref{R_DG3}) are regular and quadratic in $(\eta,w)$, belonging to the class $\Ta^\infty(s_0;H^{\bullet+1}\times H^{\bullet+1/2},H^{\bullet+s_0-2})$. On the other hand, the last two terms in \eqref{eq-diag:HiFre-RemReg} are infinitely smoothing and quadratic in $(\eta,w)$; in fact, by the definition (\ref{gamma})(\ref{gamma_detail}), the operator $T_\gamma$ differs from the Fourier multiplier $\Lambda_\disp(D_x)(1-\chi(D_x))$ (whose composition with $\proj_{\low}$ vanishes) by an operator of size $O(|\eta|_{C^2}|D_x|^{3/2})$, while the operator $T_V$ is linear in $\psi$; it is also obvious the last two terms are infinitely smoothing. We thus conclude
$$
\R_\DG^3 \in \Ta^\infty(s_0;H^{\bullet+1}\times H^{\bullet+1/2},H^{\bullet+s_0-2}), \quad \big(\R_\DG^3\big)^{[\leq 1]}=0.
$$

It remains to express the right-hand side of~\eqref{eq-diag:fin} in terms of $u$. We apply the paralinearization formula~\eqref{eq-diag:ParalinTrans} to the first two terms:
\begin{equation}\label{Op_h}
\begin{aligned}
\proj_{\high}&(A_2^1 + iA_1^1)\size[D_x]\eta +  \proj_{\high}(A_2^2 + iA_1^2)\size[D_x]^{1/2}w\\
&= \proj_{\high}(A_2^1 + iA_1^1) \size[D_x] T_{\mathfrak{p}_\inv}\RE \proj_{\high}u
+\proj_{\high}(A_2^2 + iA_1^2)\size[D_x]^{1/2}T_{\mathfrak{q}_\inv}\IM \proj_{\high}u
+ \R_\DG^4(u),
\end{aligned}
\end{equation}
where the remainder $\R_\DG^4(u)$ reads
\begin{equation}\label{R_DG4}
\R_\DG^4(u)
=\proj_{\high}(A_2^1 + iA_1^1) \size[D_x] \R^1_\inv(u)
+\proj_{\high}(A_2^2 + iA_1^2)\size[D_x]^{1/2} \R^2_\inv(u).
\end{equation}
Here we substitute in $(\eta,w)=\mathscr{A}_\DG(u)$, so that all the operators $A_j^k$ all depend on $u$ now. By Lemma \ref{lem-diag:HiFre-Rem}, this remainder term is regular and quadratic in $u$, belonging to the class $ \Ta^\infty(s_0;H^{\bullet},H^{\bullet+s_0-1/2})$, and vanishes quadratically in $u$ as $\|u\|_{H^{s_0}} \to 0$, since the operators $A_j^k$ vanish when $u=0$. 

As for the remainder $\R_\DG^3(\eta,w)$, we simply substitute in $(\eta,w)=\mathscr{A}_\DG(u)$. Therefore, 
$$
\R_\DG(u)
:=\R_\DG^3\circ\mathscr{A}_\DG(u)+\R_\DG^4(u)
\in \Ta^\infty(s_0;H^{\bullet},H^{\bullet+s_0-2}).
$$

\subsection{Conclusion: Diagonalization and Extension} 
Summarizing \eqref{eq-diag:fin}-\eqref{R_DG4}, we re-write \eqref{eq-diag:fin} into a terser form: fixing $s_0 > 5$, 
\begin{equation}\label{EQ_DG}
\partial_t u - L(D_x)\proj_{\low}u
+ i\proj_{\high}T_\gamma\proj_{\high} u 
+ \proj_{\high}T_V\partial_x\proj_{\high} u
- \proj_{\high}\A\proj_{\high}u = \R_\DG(u).
\end{equation}
Here in (\ref{EQ_DG}), with $\proj_{\low},\proj_{\high}$ as in (\ref{Low_High}) and $\eta,w,u$ related by \eqref{eq-diag:ComplexUnk},

\begin{enumerate}[label=\textbf{(DG\arabic*)}]
\item\label{Diag1}
The \emph{real linear} operator $L(D_x)$ is the same as in (\ref{LEQ(0,0)}) (see Remark \ref{Complex_u}):
$$
\begin{aligned}
L(D_x)\proj_{\low}u
&=G[0]\proj_{\low}\IM u-iH'[0]\proj_{\low}\RE u \\
&=\frac{I_1(\rho|D_x|)}{I_0(\rho|D_x|)}|D_x|\proj_{\low}\IM u
+i\left(-|D_x|^2+\frac{1}{\rho^2}\right)\proj_{\low}\RE u.
\end{aligned}
$$

\item\label{Diag2}
The symbol $\gamma$ is given by (\ref{gamma})(\ref{gamma_detail}):
$$
\gamma = \gamma^{(3/2)} -\frac{i}{2}\partial_x\partial_\xi\gamma^{(3/2)}, \quad \gamma^{(3/2)}
=\sqrt{ \left(\frac{|\xi|^{2}}{(1+|\partial_x\eta|^2)^{3/2}}-\frac{1}{\rho^2}\right)\frac{I_1(\rho|\xi|)}{I_0(\rho|\xi|)}|\xi|} \cdot\big(1-\chi(\xi)\big).
$$
When $u\simeq0$ varies in $H^4$, the symbol $\gamma$ shall depend smoothly on $u$.

\item\label{Diag3}
The real-valued function $V$ is given in (\ref{DN_V_B}):
$$
V = V(\eta,\psi)=\partial_x\psi-B\partial_x\eta,
$$
where $\psi$ and $w$ are related by Lemma \ref{lem:ChgtUnk:psi-w}. When $u\simeq0$ varies in $H^{s_0}$, $V\in H^{s_0-1/2}$ shall depend smoothly on $u$.

\item\label{Diag4}
The \emph{real linear} operator 
$$
\A=(A_2^1 + iA_1^1) \size[D_x] T_{\mathfrak{p}_\inv}\RE
+(A_2^2 + iA_1^2)\size[D_x]^{1/2}T_{\mathfrak{q}_\inv}\IM. 
$$
For every $s\in\xR$, $u\mapsto\A$ is a $C^\infty$ map from $0\simeq u\in H^{s_0}$ to the algebra of \emph{real linear} operators $\mathcal{L}(H^s)$, and $\A\big|_{u=0}=0$.

\item\label{Diag5}
The remainder $\R_\DG(u)$ is regular and quadratic in $u$ when $\|u\|_{H^{s_0}}\simeq0$:
$$
\R_\DG \in \Ta^\infty(s_0;H^{\bullet},H^{\bullet+s_0-2}), \quad \R_\DG^{[\leq 1]}=0.
$$
\end{enumerate}

Since we aim to study long time behavior of the decaying solutions, it is convenient to introduce a cut-off in the space $H^{s_0}$ and extend (\ref{EQ_DG}) into the whole of $H^{s_0}$. We choose any smooth even function $\kappa:\xR\to[0,1]$, such that $\kappa(x)=1$ for $|x|\leq2$ and $\kappa(x)=0$ for $|x|\geq4$, and fix some $\delta_0>0$ suitably (but not arbitrarily) small. Then we define the extended symbol $\gamma_{\Ext}$ as
\begin{equation}\label{gamma_Ext}
\begin{aligned}
\gamma_{\Ext}
=\left(1-\frac{i}{2}\partial_x\partial_\xi\right)
\sqrt{ \left(\frac{|\xi|^{2}}{\left(1+\kappa\big(\delta_0^{-1}\|u\|_{H^{3}_x}\big)|\partial_x\eta|^2\right)^{3/2}}-\frac{1}{\rho^2}\right)\frac{I_1(\rho|\xi|)}{I_0(\rho|\xi|)}|\xi|} \cdot\big(1-\chi(\xi)\big),
\end{aligned}
\end{equation}
so that $\gamma_{\Ext}$ is now an elliptic symbol well-defined for all $u\in H^{s_0}$ (instead of $u$ near zero). We then define the extended operator
\begin{equation}\label{Gamma_Ext}
\Gamma_\Ext
=\kappa\big(\delta_0^{-1}\|u\|_{H^{s_0}_x}\big)\left(
T_{\gamma_\Ext}-iT_V\partial_x+i\A
\right).
\end{equation}
We finally define the extended system
\begin{equation}\label{EQ_Red_Ext}
\begin{aligned}
\partial_t u - L(D_x)\proj_{\low}u
+ i\proj_{\high} \Gamma_{\Ext} \proj_{\high}u= \R_\Ext(u),
\end{aligned}
\end{equation}
where in (\ref{EQ_Red_Ext}), the extended smoothing remainder
\begin{equation}\label{R_Ext}
\R_\Ext(u)=\kappa\big(\delta_0^{-1}\|u\|_{H^{s_0}_x}\big)\R_\DG(u).
\end{equation}
The advantage of (\ref{EQ_Red_Ext}) is twofold: as we shall see in Theorem \ref{GWP_EQ}, the system is globally well-posed in $H^s$ with $s\geq s_0$, while on the other hand it obviously coincides with the actual water jet system (\ref{EQ_DG}) in the set $\|u\|_{H^{s_0}_x}<2\delta_0$, an open ball in $H^{s_0}$. We summarize the properties that will be used later as follows; they all follow easily from \ref{Diag1}-\ref{Diag5}.

\begin{enumerate}[label=\textbf{(Ext\arabic*)}]

\item\label{Ext1} 
For every $s\in\xR$, the \emph{real linear} operator $\Gamma_\Ext$ defines a smooth mapping from $u\in H^{s_0}$ to the space of operators $\mathcal{L}(H^s,H^{s-3/2})$, vanishing when $\|u\|_{H^{s_0}}\geq 4\delta_0$. Furthermore, when $u\to0$ in $H^{s_0}$, there holds, for all $s\in\xR$,
$$
\left\| \Gamma_\Ext-\kappa\big(\delta_0^{-1}\|u\|_{H^{s_0}_x}\big)\Lambda_\disp(D_x)\big(1-\chi(D_x)\big) ;\mathcal{L}(H^s,H^{s-3/2})\right\|
=O\big(\|u\|_{H^{s_0}_x}\big).
$$

\item\label{Ext2} For $s\geq s_0$, the smoothing remainder $\R_\Ext(u)$ is a smooth map from $H^s$ to $H^{s+s_0-2}$,  vanishing when $\|u\|_{H^{s_0}}\geq 4\delta_0$, satisfying tame estimates: with $\kappa_1$ being some bump function that equals 1 on $\mathrm{supp}\kappa$, 
\begin{equation}\label{Size_RExt}
\begin{aligned}
\|\R_\Ext(u)\|_{H^{s+s_0-2}}
&\lesssim_{s}\kappa\big(\delta_0^{-1}\|u\|_{H^{s_0}_x}\big)\|u\|_{H^{s_0}}\|u\|_{H^s},\\
\big\|\der_u\R_\Ext(u)\cdot\bm{\delta}u\big\|_{H^{s+s_0-2}}
&\lesssim_{s}\kappa_1\big(\delta_0^{-1}\|u\|_{H^{s_0}_x}\big)\big(\|u\|_{H^{s_0}}\|\bm{\delta} u\|_{H^s}
+\|u\|_{H^{s}}\|\bm{\delta} u\|_{H^{s_0}}\big),\\
\big\|\der_u^k\R_\Ext(u)\cdot(\bm{\delta}u)^{\otimes k} \big\|_{H^{s+s_0-2}}
&\lesssim_{s}
\delta_0^{-(k-1)}
\kappa_1\big(\delta_0^{-1}\|u\|_{H^{s_0}_x}\big) 
\big( \|\bm{\delta} u\|_{H^{s_0}}^{k-1}\|\bm{\delta} u\|_{H^s} 
+ \|u\|_{H^{s}} \|\bm{\delta} u\|_{H^{s_0}}^k \big),
\quad k\geq2.
\end{aligned}
\end{equation}
\end{enumerate}

\section{Paradifferential Propagator and Twisted Duhamel Formula}\label{Sec:4}

In this section, we derive from the quasilinear paradifferential equation (\ref{EQ_Red_Ext}) a \emph{twisted Duhamel formula}, which is an integral equation that leads to a fixed point equation. This will be achieved through a detailed analysis of the \emph{propagator} associated to the paradifferential operator $\partial_t+i\proj_{\high}\Gamma_{\Ext}\proj_{\high}$. Due to the quasilinearity of the system, differentiation of the propagator with respect to coefficients of $\gamma_\Ext$ causes a loss of 1.5 derivatives. But this loss can be easily compensated by the regularity gain in the right-hand side of (\ref{EQ_Red_Ext}). From the twisted Duhamel formula, we will deduce the global well-posedness of the truncated system~\eqref{EQ_Red_Ext}. Meanwhile, energy estimates and dependence on initial data of the paradifferential propagator are crucial in the construction of hyperbolic manifolds. The theory is similar for either $\xG=\xR$ or $\xT$, so we shall not make distinction for these two cases. 

\subsection{Paradifferential Propagator}\label{subsect:Fundamental}
We are at the place to study the propagator, or fundamental solution of the paradifferential operator $\partial_t+i\proj_{\high}\Gamma_{\Ext}\proj_{\high}$, which determines the behavior of the solution in high frequencies. The proof will occupy the rest of this subsection.

\begin{theorem}\label{Fundamental}
Fix an index $s_0>5$. Fix some $u\in L^\infty_{t;\loc}H^{s_0}_x\cap W^{1,\infty}_{t;\loc}H^{s_0-3/2}_x$, and let the operator $\Gamma_\Ext$ be as in (\ref{Gamma_Ext}):
$$
\Gamma_\Ext
=\kappa\big(\delta_0^{-1}\|u\|_{H^{s_0}_x}\big)\left(
T_{\gamma_\Ext}-iT_V\partial_x+i\A
\right).
$$
There exists a linear operator $\bm{F}(u;t,t_0)$, called the ``propagator", or ``fundamental solution" of $\partial_t+i\proj_{\high} \Gamma_\Ext\proj_{\high}$, verifying the following properties:
\begin{enumerate}[label=(\textbf{F\arabic*})]
\item\label{F1} There holds the transition property $\bm{F}(u;t,t)\equiv\Id$, $\bm{F}(u;t,t_0)\bm{F}(u;t_0,t')=\bm{F}(u;t,t')$. Fixing $s\in\xR$ and $f\in L^1_{t;\loc}H^s_x$, the function
$$
v(t)=\bm{F}(u;t,t_0)v(t_0)+\int_{t_0}^t\bm{F}(u;t,\tau)f(\tau)\dtau
$$
is of class $C^0_{t;\loc}H^s_x\cap W^{1,\infty}_{t,\loc}H^{s-3/2}_x$ and is the unique solution in that space of the linear evolution equation
$$\partial_t v+i\proj_{\high}\Gamma_{\Ext}\proj_{\high} v=f\in L^1_{t;\loc}H^s_x,\quad v(t_0)\in H^s\text{ given}.
$$

\item\label{F2} For each $s\in\xR$, $\bm{F}(u;t,t_0)$ is a bounded real linear operator on $H^s_x$: given any $h\in H^s_x$,
$$
\|\bm{F}(u;t,t_0)h\|_{H^s_x}
\leq C_s\exp\left(C_s\int_{[t_0,t]}\mathtt{F}(u;
\tau)\dtau\right)\|h\|_{H^s_x},
$$
where with $\kappa_1$ being some bump function that equals 1 on $\mathrm{supp}\kappa$, 
$$
\mathtt{F}(u;t)=\kappa\big(\delta_0^{-1}\|u\|_{H^{s_0}_x}\big) \|u(t)\|_{H^{s_0}_x}+\kappa_1\big(\delta_0^{-1}\|u(t)\|_{H^{3}_x}\big)\|\partial_tu(t)\|_{H^{3}_x}.
$$
Here we write $\int_{[a,b]}f$ for the absolute value of the integral $\int_a^bf$, regardless of the order of $a,b$.

\item\label{F3} Each differentiation of $\bm{F}(u;t,t_0)$ with respect to\footnote{Although $L^\infty_{t;\loc}H^{s_0}_x\cap W^{1,\infty}_{t;\loc}H^{s_0-3/2}_x$ is not a normed space, the propagator $\bm{F}(u;t,t_0)$ clearly depends only on the restriction of $u$ between time $t_0$ and $t$, so differentiation in $u$ does make sense.} $u\in L^\infty_{t;\loc}H^{s_0}_x\cap W^{1,\infty}_{t;\loc}H^{s_0-3/2}_x$ loses 1.5 derivatives: fixing $k\geq1$, $s\in\xR$, given any $h\in H^s$, there holds
$$
\begin{aligned}
\|(\der_u&\bm{F}\cdot(\del u)^{\otimes k})(u;t,t_0)h\|_{H^{s-3k/2}_x}\\
&\leq C_{s,k}\delta_0^{-k}
\exp\left(C_{s,k}\int_{[t_0,t]}\mathtt{F}(u;\tau)\dtau\right)\cdot
\left(\int_{[t_0,t]}\del u(\tau)\|_{H^{s_0}_x}\dtau\right)^k
\|h\|_{H^s_x}.
\end{aligned}
$$
\end{enumerate}
\end{theorem}

The proof follows from a rather standard energy estimate and duality argument. We divide it into several steps. Throughout this proof, we shall only specify the dependence of functions, symbols, and operators on the time variable $t$ and omit the dependence on $x$. For example, $\Gamma_\Ext(t)v(t)$ denotes the action of the operator $\Gamma_\Ext$, evaluated at time $t$, on the function $v(t)\in H^s_x$. When considered at a fixed time $t$, even the dependence on $t$ is omitted. It is always assumed that $u$ and $(\eta,w)$ are related by Lemma~\ref{DG_Transform} when $\|u\|_{H^{s_0}_x}\leq 4\delta_0$ (due to the truncation $\kappa\big(\delta_0^{-1}\|u\|_{H^{s_0}_x}\big)$, the operators are all trivial otherwise).

\noindent
\textbf{Step 1: Almost Self-adjointness.} In this first step, we shall fix the time $t$ and omit its dependence in equations. All the constants appearing in this step are independent of $t$. 

Consider the \emph{real-valued} $L^2_x$-inner product for complex-valued functions:
$$
\langle u_1,u_2\rangle_{L^2_x}
:=\langle \RE u_1,\RE u_2\rangle_{L^2_x}+\langle \IM u_1,\IM u_2\rangle_{L^2_x}.
$$
We claim that the operator $\Gamma_\Ext$ is almost self-adjoint in the following sense: setting $\Gamma_\Ext^*$ to be its $L^2_x$-adjoint\footnote{Note that $\Gamma_\Ext$ is \emph{real linear} instead of \emph{complex linear}; but this does not harm the argument at all, see Remark \ref{Complex_u}.}, there holds, for any $v\in H^s_x$,
\begin{equation}\label{Self_adjoint}
\big\|(\Gamma_\Ext-\Gamma_\Ext^*)v\big\|_{L^2_x}
\leq C\kappa\big(\delta_0^{-1}\|u\|_{H^{s_0}_x}\big)\|u\|_{H^{s_0}_x}\|v\|_{L^2_x},
\end{equation}
where $C$ is a constant independent of all the arguments. This follows easily from the definition (\ref{gamma_Ext}) of $\gamma_\Ext$ and the symbolic calculus formula in Proposition~\ref{prop:SymReg} for adjoints of paradifferential operators. In fact, we compute, using the definition of the ${\times;r}$ operation (noting $\|u\|_{H^{s_0}_x}$ does not depend on $x,\xi$), for $r\in(1,2)$,
$$
\begin{aligned}
(\gamma_\Ext+V\xi)^{\times;r}
&=\left(1-i\partial_\xi\partial_x\right)\kappa\big(\delta_0^{-1}\|u\|_{H^{s_0}_x}\big)
\overline{\left(\gamma^{(3/2)}_\Ext-\frac{i}{2}\partial_\xi\partial_x\gamma^{(3/2)}_\Ext+V\xi\right)}\\
&=\kappa\big(\delta_0^{-1}\|u\|_{H^{s_0}_x}\big)
\big(\gamma_\Ext+V\xi\big)
+\kappa\big(\delta_0^{-1}\|u\|_{H^{s_0}_x}\big)
\left(\frac{1}{2}\partial_x^2\partial_\xi^2\gamma^{(3/2)}_\Ext
-i\partial_xV\right).
\end{aligned}
$$
the last term vanishes linearly as $\|u\|_{H^{s_0}_x}\to0$ by \ref{Diag3} and \ref{Ext1}. On the other hand, the adjoint of the bounded real linear operator $\kappa\big(\delta_0^{-1}\|u\|_{H^{s_0}_x}\big)\A$ is still bounded. This proves (\ref{Self_adjoint}).

\noindent
\textbf{Step 2: Adapted Energy Functional.} Time dependence will still be omitted in this step. Introduce
\begin{equation}\label{b_Ext}
1+b_\Ext:=\left(1+\kappa\big(\delta_0^{-1}\|u\|_{H^{3}_x}\big)|\partial_x\eta|^2\right)^{-3/4}.
\end{equation}
Comparing with the definition of $\gamma_\Ext$ in (\ref{gamma_Ext}), recalling the asymptotic formula $\Lambda_\disp(\xi)=|\xi|^{3/2}+O(|\xi|^{1/2})$, it is clear that 
\begin{equation}\label{T_gamma_Ext_Asym}
T_{\gamma_\Ext}=T_{1+b_\Ext}\size[D_x]^{3/2}
+\text{paradifferential operator of order }\leq\frac{1}{2}.
\end{equation}
We claim the following auxiliary result:
\begin{lemma}\label{Ellip_s}
Fix a function $u\in H^{s_0}_x$ and $s\in\xR$. Fix a smooth bump function $\chi\in C^\infty_0(\xR)$ that equals to 1 near 0. Consider the paradifferential operator 
\begin{equation}\label{Gamma_s}
\Gamma_s:=\Op^\PM\left((1+b_\Ext)^{2s/3}\right)\size[D_x]^{2s/3};
\end{equation}
recall from (\ref{DefOpPM}) that $\Op^\PM(a)$ is just another way of writing $T_a$. There is a constant $K_s\geq1$ depending only on $s$, such that the following ellipticity inequality holds:
\begin{equation}\label{Elliptic_s}
\begin{aligned}
\frac{\|v\|_{H^s_x}^2}{K_s}
&\leq
\|\Gamma_sv\|_{L^2_x}^2
+K_s\cdot
\big\|\chi\big(D_x/K_s\big)v\big\|_{L^2_x}^2
\leq K_s\|v\|_{H^s_x}^2.
\end{aligned}
\end{equation}
\end{lemma}
\begin{proof}
In the following, $K_s\geq1$ will refer to a constant that might vary from line to line. The proof of the right half of the inequality is easy since $\Gamma_s$ is a paradifferential operator of order $s$: recalling Lemma \ref{DG_Transform}, $\partial_x\eta$ has the same regularity as $u$, so $|b_\Ext|_{C^1_x}\leq C\kappa\big(\delta_0^{-1}\|u\|_{H^{3}_x}\big)\|u\|_{H^{3}_x}^2\leq C\delta_0^2$, which \emph{does not depend on $u$} due to the cut-off $\kappa\big(\delta_0^{-1}\|u\|_{H^{3}_x}\big)$. The proof of the left half follows from a standard elliptic parametrix argument: recalling the symbolic calculus formula in Proposition \ref{prop:SymReg}, we find 
$$
\begin{aligned}
\Gamma_{-s}\Gamma_s
&=\Op^\PM((1+b_\Ext)^{-2s/3})\size[D_x]^{-s}\Op^\PM((1+b_\Ext)^{2s/3})\size[D_x]^{s}\\
&=\Id+R_s,
\end{aligned}
$$
where the operator $R_s$ satisfies 
$$
\|R_sv\|_{H^{s}_x}\leq K_s|b_\Ext|_{C^1_x}\|v\|_{H^{s-1}_x}\leq K_s\|v\|_{H^{s-1}_x}.
$$
Therefore,
$$
\begin{aligned}
\|v\|_{H^s_x}
&\leq \|\Gamma_{-s}\Gamma_s v\|_{H^s_x}+\|R_sv\|_{H^s_x}\\
&\leq K_s\|\Gamma_sv\|_{L^2_x}
+K_s\|v\|_{H^{s-1}_x}\\
&\leq K_s\|\Gamma_sv\|_{L^2_x}+\frac{1}{2}\|v\|_{H^s_x}
+K_s\big\|\chi\big(D_x/K_s\big)v\big\|_{L^2_x}^2.
\end{aligned}
$$
Absorbing the term $\|v\|_{H^s_x}/2$ to the left, we obtain the desired inequality. Note that the second step is valid since $\Gamma_{-s}$ is a paradifferential operator of order $-s$ (see Proposition~\ref{prop:PDReg}), while the last step follows from interpolation: for $\varepsilon<1$, there holds 
$$
\|v\|_{H^{s-1}_x}\leq \varepsilon\|v\|_{H^s_x}+C_{s}\max\left(1,\varepsilon^{-(s-1)}\right)\big\|\chi(\varepsilon D_x)v\big\|_{L^2_x}.
$$
This finishes the proof.
\end{proof}

\noindent
\textbf{Step 3: A Priori Energy Estimate in $L^2_x$.} Starting from this step, time dependence becomes important. Given $u\in  L^\infty_{t;\loc}H^{s_0}_x\cap W^{1,\infty}_{t;\loc}H^{s_0-3/2}_x$, the symbol $\gamma_\Ext$ is then time-dependent (though not explicitly); we write $\Gamma_\Ext(t)$ to denote the paradifferential operator corresponding to $\gamma_\Ext$ evaluated at time $t$.

Suppose \emph{in a priori} that $v\in L^\infty_{t;\loc}L^2_x$ solves the linear evolution equation
\begin{equation}\label{Lin_Ev_L2}
\partial_t v+i\proj_{\high}\Gamma_{\Ext}\proj_{\high} v=f
\in L^1_{t;\loc}L^2_x,\quad v(t_0)\in L^2\text{ given}.
\end{equation}
We claim the following \emph{a priori} differential inequality:
\begin{equation}\label{Energy_L2_Diff}
\frac{\diff}{\dt}\|v(t)\|_{L^2_x}^2
\leq C\kappa\big(\delta_0^{-1}\|u(t)\|_{H^{s_0}_x}\big)\|u(t)\|_{H^{s_0}_x}\|v\|_{L^2_x}^2+\|f(t)\|_{L^2_x}\|v(t)\|_{L^2_x}.
\end{equation}
where $C>0$ is a universal constant. Indeed, differentiating $\|v(t)\|_{L^2_x}^2$ with respect to $t$, we find
$$
\begin{aligned}
\frac{\diff}{\dt}\|v(t)\|_{L^2_x}^2
&=2\langle i\proj_{\high}\Gamma_{\Ext}(t)\proj_{\high}v(t),v(t)\rangle_{L^2_x}
+\langle f(t),v(t)\rangle_{L^2_x}.
\end{aligned}
$$
We compute the first term as 
$$
\begin{aligned}
2\langle i\proj_{\high}\Gamma_{\Ext}(t)\proj_{\high}v(t),v(t)\rangle_{L^2_x}
&=\left\langle i\big(\proj_{\high}\big(\Gamma_{\Ext}(t)-\Gamma_\Ext(t)^*\big)\proj_{\high} \big)
v(t),v(t)\right\rangle_{L^2_x}
\end{aligned}
$$
By (\ref{Self_adjoint}), the magnitude of this is 
$$
\leq C\kappa\big(\delta_0^{-1}\|u(t)\|_{H^{s_0}_x}\big)\|u(t)\|_{H^{s_0}_x}\|v(t)\|_{L^2_x}^2.
$$
The inner product $\langle f(t),v(t)\rangle_{L^2_x}$ is quite easily controlled. We thus obtain (\ref{Energy_L2_Diff}).

\noindent
\textbf{Step 4: A Priori Energy Estimate in $H^s_x$.} Given $u\in L^\infty_{t;\loc}H^{s_0}_x\cap W^{1,\infty}_{t;\loc}H^{s_0-3/2}_x$, We write $\Gamma_s(t)$ to denote the evaluation of the paradifferential operator (\ref{Gamma_s}) at time $t$. Based on Lemma \ref{Ellip_s}, we define the following \emph{energy functional}:
\begin{equation}\label{Energy_s}
E_s[v(t)]:=\|\Gamma_s(t)v(t)\|_{L^2_x}^2
+K_s\cdot
\big\|\chi\big(D_x/K_s\big)v(t)\big\|_{L^2_x}^2.
\end{equation}

Fix some $s\in\xR$. Suppose \emph{in a priori} that $v\in L^\infty_{t;\loc}H^s_x$ solves the linear evolution equation
\begin{equation}\label{Lin_Ev_Hs}
\partial_t v+i\proj_{\high}\Gamma_{\Ext}\proj_{\high} v=f
\in L^1_{t;\loc}H^s_x,\quad v(t_0)\in H^s\text{ given}.
\end{equation}
We apply the operator $\Gamma_s(t)$ to (\ref{Lin_Ev_Hs}), and find the equation satisfied by $v_s(t):=\Gamma_s(t)v(t)$:
\begin{equation}\label{Lin_Ev_Gv}
\begin{aligned}
\partial_t v_s(t)+i\proj_{\high}\Gamma_{\Ext}(t)\proj_{\high} v_s(t)
&=B_s(t)v(t)+\Gamma_s(t)f(t),
\quad\text{with }v_s(t_0)=\Gamma_s(t_0)v(t_0)\in L^2,\\
\text{where }
B_s(t)&=\left(\partial_t\Gamma_s(t)+i\big[\Gamma_s(t),\proj_{\high} \Gamma_\Ext(t) \proj_{\high}\big]\right),
\end{aligned}
\end{equation}

We claim that the right-hand side of (\ref{Lin_Ev_Gv}) is in $L^1_{t;\loc}L^2_x$, even though it involves the commutator $\big[\Gamma_s(t),\proj_{\high} \Gamma_\Ext(t) \proj_{\high}\big]$ that is supposed to be of order $s+1/2$. We first compute 
$$
\partial_t\Gamma_s(t)
=\frac{2s}{3}
\Op^{\PM}\left(\partial_tb_\Ext(1+b_\Ext)^{2s/3}\right)\size[D_x]^{s},
$$
which is still a paradifferential operator of order $s$. Recalling from (\ref{b_Ext}) the definition of $b_\Ext$, 
$$
\partial_tb_\Ext
=-\frac{3\kappa\big(\delta_0^{-1}\|u(t)\|_{H^{3}_x}\big)(\partial_t\partial_x\eta\cdot\partial_x\eta)}{2\left(1+\kappa\big(\delta_0^{-1}\|u\|_{H^{3}_x}\big)|\partial_x\eta|^2\right)^{7/4}}
-\frac{3\kappa'\big(\delta_0^{-1}\|u(t)\|_{H^{3}_x}\big)\langle u(t),\partial_tu(t)\rangle_{H^{3}_x}|\partial_x\eta|^2}{4\delta_0\|u(t)\|_{H^{3}_x}\left(1+\kappa\big(\delta_0^{-1}\|u\|_{H^{3}_x}\big)|\partial_x\eta|^2\right)^{7/4}}.
$$
By Lemma \ref{DG_Transform}, $u\mapsto \eta$ is smooth from $H^3_x$ to $H^4_x$. So we obtain the estimate 
$$
|\partial_tb_\Ext|_{L^\infty_x}
\leq C\kappa_1\big(\delta_0^{-1}\|u(t)\|_{H^{3}_x}\big)
\|\partial_tu\|_{H^{3}_x},
$$
where $\kappa_1\geq0$ is any bump function that dominates $|\kappa'|$. Therefore, recalling the boundedness of paradifferential operators (Proposition \ref{prop:PDReg}), we have
\begin{equation}\label{RHS1}
\left\|\big(\partial_t\Gamma_s(t)\big) v(t)\right\|_{L^2_x}
\leq C_s\kappa_1\big(\delta_0^{-1}\|u(t)\|_{H^{3}_x}\big)\|\partial_tu(t)\|_{H^{3}_x}\|v(t)\|_{H^s_x}.
\end{equation}

We then make a key observation: the commutator $[\Gamma_s(t),\Gamma_\Ext(t)]$ is an operator \emph{of order $s-1/2$ instead of $s+1/2$}. Indeed, from (\ref{T_gamma_Ext_Asym}), the symbol $\gamma_\Ext(t)=(1+b_\Ext(t))\size[\xi]^{3/2}+O(|\xi|^{1/2})$. By the composition formula in Proposition \ref{prop:SymReg}, the commutator $[\Gamma_s(t),\Gamma_\Ext(t)]$ has principal symbol given by the Poisson bracket
$$
    \kappa\big(\delta_0^{-1}\|u(t)\|_{H^{s_0}_x}\big) \Big\{(1+b_\Ext(t))^{2s/3}\size[\xi]^s,(1+b_\Ext(t))\size[\xi]^{3/2}\Big\}.
$$
Elementary calculus immediately shows that this Poisson bracket vanishes identically. On the other hand, the commutator of $\Gamma_s(t)$ with any paradifferential operator of order $\leq1$ is again of order $\leq s$ -- including $T_{V(t)}\partial_x$ and the subprincipal part of $T_{\gamma_\Ext(t)}$. The commutator $[\Gamma_s(t),\A(t)]$ is still of order $\leq s$, since $\A(t)$ is bounded in all Sobolev spaces (see \ref{Diag4}). Therefore, we find 
\begin{equation}\label{Gamma_s_Comm_T_gamma}
\big\|[\Gamma_s(t), \Gamma_\Ext(t)] v(t)\big\|_{L^2_x}
\leq C_s\kappa\big(\delta_0^{-1}\|u(t)\|_{H^{s_0}_x}\big)\|u(t)\|_{H^{s_0}_x}\|v(t)\|_{H^s_x}.
\end{equation}
This estimate is valid since the differentiated symbols 
$$
\partial_\xi^{\alpha}\gamma_\Ext(t),\,
\partial_x^{\alpha}\gamma_\Ext(t),\,
\partial_\xi^{\alpha}\big((1+b(t))^{2s/3}\size[\xi]^s\big), \,\partial_x^{\alpha}\big((1+b(t))^{2s/3}\size[\xi]^s\big),
\quad |\alpha|\geq2,
$$
all have coefficients vanishing linearly as $\|u(t)\|_{H^{s_0}_x}\to0$ (recall the symbolic calculus formula in Proposition \ref{prop:SymReg}). The factor $\kappa\big(\delta_0^{-1}\|u(t)\|_{H^{s_0}_x}\big)$ enters due to the presence of it in the definition of $\Gamma_\Ext(t)$.

On the other hand, $[\Gamma_s(t),\proj_{\high}]=-\big[\Gamma_s(t)-\kappa\big(\delta_0^{-1}\|u\|_{H^{s_0}_x}\big)\Lambda_\disp(D_x)^{2s/3}(1-\chi(D_x))^{2s/3},\proj_{\low}\big]$ is infinitely smoothing, since $\proj_{\low}$ has finite rank and commutes with any Fourier multiplier. By \ref{Ext1}, it vanishes linearly as $\|u\|_{H^{s_0}_x}\to0$,  Therefore, 
\begin{equation}\label{RHS2}
\begin{aligned}
\big\|i[\Gamma_s(t)&,\proj_{\high}\Gamma_{\Ext}(t)\proj_{\high}]v(t)+\Gamma_s(t)f(t)\big\|_{H^s_x}\\
&\overset{(\ref{Elliptic_s})}{\leq} 
\big\|[\Gamma_s(t),\proj_{\high}]\Gamma_\Ext(t)\proj_{\high}v(t)\big\|_{H^s_x}
+\big\|\proj_{\high}\Gamma_\Ext(t)[\Gamma_s(t),\proj_{\high}]v(t)\big\|_{H^s_x}\\
&\quad+
\big\|[\Gamma_s(t), \Gamma_\Ext(t)]\proj_{\high} v(t)\big\|_{H^s_x}
+C_s\|f(t)\|_{H^s_x}\\
&\overset{(\ref{Gamma_s_Comm_T_gamma})}{\leq} C_s\kappa\big(\delta_0^{-1}\|u(t)\|_{H^{s_0}_x}\big)\|u(t)\|_{H^{s_0}_x}\|v(t)\|_{H^s_x}+C_s\|f(t)\|_{H^s_x}.
\end{aligned}
\end{equation}

Since $\chi(D_x/K_s)$ is infinitely smoothing, we find that
$$
\partial_t\chi(D_x/K_s)v(t)=-i\chi(D_x/K_s)\proj_{\high}\Gamma_{\Ext}(t)\proj_{\high}v(t)+\chi(D_x/K_s)f(t)
\in H^{+\infty}_x.
$$
Using again the almost self-adjointness (\ref{Self_adjoint}), arguing exactly as in Step 3, we obtain
\begin{equation}\label{RHS3}
\frac{\diff}{\dt}\big\|\chi(D_x/K_s)v(t)\big\|_{L^2_x}^2
\leq C_s\kappa\big(\delta_0^{-1}\|u\|_{H^{s_0}_x}\big)
\|u\|_{H^{s_0}_x}\|v(t)\|_{H^s_x}^2+\|f(t)\|_{H^{s}_x}\big\|\chi(D_x/K_s)v(t)\big\|_{L^2_x}.
\end{equation}
Putting (\ref{RHS1})(\ref{RHS2})(\ref{RHS3}) together, using (\ref{Elliptic_s}) and the differential inequality (\ref{Energy_L2_Diff}), we finally obtain
$$
\begin{aligned}
\frac{\diff}{\dt}E_s[v(t)]
&\leq C_s\mathtt{F}(u;t)E_s[v(t)]
+C_s\|f(t)\|_{H^s_x}\sqrt{E_s[v(t)]},\\
\text{where}\quad\mathtt{F}(u;t)&=\kappa\big(\delta_0^{-1}\|u\|_{H^{s_0}_x}\big)\|u(t)\|_{H^{s_0}_x}+\kappa_1\big(\delta_0^{-1}\|u(t)\|_{H^{3}_x}\big)\|\partial_tu(t)\|_{H^{3}_x}.
\end{aligned}
$$
A standard Gr\"{o}nwall inequality argument, together with (\ref{Elliptic_s}), yields
\begin{equation}\label{Energy(s)}
\begin{aligned}
\|v(t)\|_{H^s_x}
&\leq C_s\exp\left(C_s\int_{[t_0,t]}\mathtt{F}(u;\tau)\dtau\right)\left(\|v(t_0)\|_{H^s_x}
+\int_{[t_0,t]}\|f(\tau)\|_{H^s_x}\dtau\right).
\end{aligned}
\end{equation}

\noindent
\textbf{Step 5: Duality Argument.} Having obtained the adjoint formula (\ref{Self_adjoint}) and the energy estimate (\ref{Energy(s)}) for any $s\in\xR$, the well-posedness of the linear evolution equation 
$$
\partial_t v+i\proj_{\high}\Gamma_{\Ext}\proj_{\high} v=f
\in L^1_{t;\loc}H^s_x,\quad v(t_0)\in H^s\text{ given}
$$
in the space $L^\infty_{t;\loc}H^s_x\cap W^{1,\infty}_{t;\loc}H^{s-3/2}_x$ then follows from a standard duality argument, restricted to any finite time interval. See for example, Subsection 6.3 in Chapter VI of \cite{Hormander1997}, or Section 2C of \cite{AG}; we omit the details here. To justify $v\in C^0_{t;\loc}H^s_x$, consider any sequence $\{f^n\}\subset C_{0;t,x}^{\infty}$ that converges to $f$ in the topology of $L^1_{t;\loc}H^s_x$ and initial value sequence $\{v^n(t_0)\}\subset C^\infty_0$ that converges to $v(t_0)$ in $H^s$. Then $v^n$ is indeed $L^\infty_{t;\loc}H^{s'}_x$ for any $s'\geq s$, so by the equation itself we conclude $v\in W^{1,\infty}_{t;\loc}H^{s'-3/2}_x$; in particular $v\in C^0_{t;\loc}H^s_x$. The energy estimate (\ref{Energy(s)}) then implies, for any time interval $I\Subset\xR$ containing $t_0$,
$$
\begin{aligned}
\sup_{t\in I}\|v^n(t)-v(t)\|_{H^s_x}
\leq C_s(u,I)\left(\|v^n(t_0)-v(t_0)\|_{H^s}+\int_{I}\|f^n(\tau)-f(\tau)\|_{H^s_x}\dtau\right)
\to 0,
\end{aligned}
$$
forcing the limit $v$ to be continuous. In conclusion, the well-posedness immediately gives the existence of the propagator $\bm{F}(u;t,t_0)$, satisfying \ref{F1}\ref{F2}. 

\noindent
\textbf{Step 6: Differentiability.} The differentiability of $\bm{F}(u;t,t_0)$ with respect to $u$ follows easily by a direct computation. For simplicity, we write down only the estimate for first order differentiation. Indeed, given $u\in L^\infty_{t;\loc}H^{s_0}_x\cap W^{1,\infty}_{t;\loc}H^{s_0-3/2}_x$ and $h\in H^s_x$, the function $v(t):=\bm{F}(u;t,t_0)h\in L^\infty_{t;\loc}H^s_x$ is the unique solution of the homogeneous Cauchy problem
$$
\partial_tv+i\proj_{\high}\Gamma_{\Ext}\proj_{\high}v=0,\quad w(t_0)=h,
$$
and verifies the energy estimate
$$
\|v(t)\|_{H^s_x}\leq
C_s\exp\left(C_s\int_{[t_0,t]}\mathtt{F}(u;
\tau)\dtau\right)\|h\|_{H^s_x}
$$
The linearized equation with respect to $u$ reads
$$
\partial_t(\der_uv\cdot\del u)+i\proj_{\high}\Gamma_{\Ext}\proj_{\high}(\der_uv\cdot\del u)=i\proj_{\high} (\der_u\Gamma_{\Ext}\cdot\del u)\proj_{\high}v,
\quad (\der_uv\cdot\del u)(t_0)=0.
$$
The right-hand side is of class $L^\infty_{t;\loc}H^{s-3/2}_x$: for any $z\in L^\infty_{t;\loc}H^s_x$, by \ref{Ext1},
\begin{equation}\label{Tilde_Gamma}
\big\|\proj_{\high} (\der_u\Gamma_{\Ext}(t)\cdot\del u(t))\proj_{\high}z(t)\big\|_{H^{s-3/2}_x}
\leq C_s\delta_0^{-1}
\|\del u(t)\|_{H^{s_0}_x}\cdot\|z(t)\|_{H^s_x},
\end{equation}
Differentiation in $u$ of order $k\geq2$ shall yield a factor $\delta_0^{-k}$ by a similar computation. 

Using the inhomogeneous energy estimate (\ref{Energy(s)}) with solution $v$ and source term $f$ replaced by $\der_uv\cdot\del u$ and $i\proj_{\high} \der_u\Gamma_{\Ext}\proj_{\high}v$, respectively, we immediately obtain the desired estimate \ref{F3} for $k=1$ in the statement of Theorem \ref{Fundamental}. Indeed, with the explicit formula in Item \ref{F1}, we can even write down the explicit formula
\begin{equation}\label{DuF(u)}
\begin{aligned}
\big((\der_u\bm{F}\cdot\del u)(u;t,t_0)\big)h
&=i\int_{t_0}^t\bm{F}(u;t,\tau)\cdot\tilde{\Gamma}(\tau)
\cdot\bm{F}(u;\tau,t_0)h\dtau,\\
\tilde{\Gamma}(t)&:=\proj_{\high}(\der_u\Gamma_{\Ext}(t)\cdot\del u(t))\proj_{\high}.
\end{aligned}
\end{equation}
This finishes the proof for Theorem \ref{Fundamental}.

\subsection{Twisted Duhamel Formula and Global Existence for the Extended System}\label{Subsec_St_1}
Having obtained the extended equation (\ref{EQ_Red_Ext}) and Theorem \ref{Fundamental}, we are now able to deduce a \emph{twisted Duhamel formula} that brings (\ref{EQ_Red_Ext}) into an integral equation. This is the starting point of the construction of any invariant set. Suppose $s\geq s_0$ and $I$ is a closed time interval containing $t_0$. Suppose $u\in C^0_IH^s_x\cap W^{1,\infty}_IH^{s-3/2}_x$ solves (\ref{EQ_Red_Ext}). Projecting to the low frequency regime $0\leq|\xi|<2\rho^{-1}$, the usual Duhamel formula yields
$$
\begin{aligned}
\proj_{\low} u(t)
=e^{(t-t_0)L(D_x)}u(t_0)
+e^{tL(D_x)}\int_{t_0}^t e^{-\tau L(D_x)}\proj_{\low}\R_{\Ext}(u(\tau))\dtau;
\end{aligned}
$$
we emphasize here that the Fourier multiplier $L(D_x)$, given by \ref{Diag1}, is \emph{real linear} (equivalently by (\ref{LEQ(0,0)}); see Remark \ref{Complex_u}). On the other hand, the projection to high frequencies becomes
$$
\partial_t \proj_{\high}u+i\proj_{\high}\Gamma_{\Ext}\proj_{\high}u
=\proj_{\high} \R_\Ext(u),
$$
which is completely governed by the propagator $\bm{F}(u;t,t_0)$ constructed in Theorem \ref{Fundamental}. Therefore, we obtain the \emph{twisted Duhamel formula}:

\begin{proposition}\label{Twist_Duhamel}
Fix indices $s\geq s_0>5$, an initial time $t_0$ and a time interval $I$ containing $t_0$. The function $u\in C^0_IH^s_x\cap W^{1,\infty}_IH^{s-3/2}_x$ solves the extended equation (\ref{EQ_Red_Ext}) with a given value $u(t_0)\in H^s$ at $t=t_0$, if and only if it solves the following integral equation for $t\in I$:
\begin{equation}\label{Duhamel}
\begin{aligned}
u(t)&=e^{(t-t_0)L(D_x)}\proj_{\low}u(t_0)
+\bm{F}(u;t,t_0)\proj_{\high} u(t_0)\\
&\quad+e^{tL(D_x)}\int_{t_0}^t e^{-\tau L(D_x)}\proj_{\low}\R_{\Ext}(u(\tau))\dtau
+\int_{t_0}^t \bm{F}(u;t,\tau)\proj_{\high}\R_{\Ext}(u(\tau))\dtau.
\end{aligned}
\end{equation}
\end{proposition}

As a consequence of the twisted Duhamel formula (\ref{Duhamel}), we find that the extended equation (\ref{EQ_Red_Ext}) is \emph{globally well-posed} in $H^s$ for any initial value $u(t_0)\in H^s$, with $s\geq s_0+3/2$. The solution $u(t)$ therefore exists globally in time and solves the converted water jet equation (\ref{EQ_DG}) within the ball $B^{s_0}(0,2\delta_0)$.

\begin{theorem}\label{GWP_EQ}
Fix indices $s_0>5$, $s\geq s_0+3/2$ and initial data $u(t_0)\in H^{s}$. The paradifferential equation \eqref{EQ_Red_Ext} has a unique global-in-time solution $u(t)$ in the space $C^0_{t;\loc}H^{s}_x\cap W^{1,\infty}_{t;\loc}H^{s-3/2}_x$. For any given $\delta>0$ and given time interval $I\Subset\xR$, the solution map $H^{s}_x \mapsto C^0_t\big(I;H^{s-\delta}_x\big), \,  u(t_0) \mapsto u$ is continuous \footnote{It is possible to prove the continuity in $H^{s}_x$ norm instead of $H^{s-\delta}_x$. But this fact is useless in the construction of hyperbolic invariant manifolds.}. 
\end{theorem}

The local well-posedness for the water jet system is already established in \cite{HK2023} and \cite{Yang2024}, based on which we can expect this global well-posedness result, since the solution is governed by linear equation when it exits the ball $\|u\|_{H^{s_0}_x}\leq 4\delta_0$. Here we apply a more delicate argument using Theorem \ref{Fundamental} to obtain the global-in-time result. The proof will occupy the rest of this subsection. 

\noindent
\textbf{Step 1: Set up.} 
Let $I\Subset\xR$ be a small interval containing $t_0$ to be determined later. Define a mapping $\mathscr{G}(u,f;t)$ for $u\in W^{1,\infty}_IH^{s-3/2}_x$ and $f\in H^{s}$ by
\begin{equation}\label{LWP}
\begin{aligned}
\mathscr{G}(u,f;t)
&:=e^{(t-t_0)L(D_x)}\proj_{\low}f
+\bm{F}(u;t,t_0)\proj_{\high} f\\
&\quad+e^{tL(D_x)}\int_{t_0}^t e^{-\tau L(D_x)}\proj_{\low}\R_{\Ext}(u(\tau))\dtau
+\int_{t_0}^t \bm{F}(u;t,\tau)\proj_{\high}\R_{\Ext}(u(\tau))\dtau.
\end{aligned}
\end{equation}
Note that by Theorem \ref{Fundamental}, the propagator $\bm{F}(u;t,t_0)$ is well-defined since $s\geq s_0+3/2$. We aim to show that $u\mapsto\mathscr{G}(u,f;t)$ is indeed $C^1$ mapping from $W^{1,\infty}_IH^{s-3/2}_x$ to itself, and defines a contraction when $I$ is sufficiently short (depending on $\|f\|_{H^s}$).

Using the explicit expression (\ref{Sol_Lin}), we obtain the easy linear bound
\begin{equation}\label{exp(tL)_easy}
\|e^{tL(D_x)}f\|_{H^s_x}
\leq C_s e^{\lambda|t|}\|f\|_{H^s}.
\end{equation}

\noindent
\textbf{Step 2: Self-map on $W^{1,\infty}_IH^{s-3/2}_x$.} 
We first show that $\mathscr{G}$ does map $W^{1,\infty}_IH^{s-3/2}_x$ to itself. Suppose $\big\|u;W^{1,\infty}_IH^{s-3/2}_x\big\|\leq R$. We have for all $s_1\in\xR$,
\begin{equation}\label{Factor}
\begin{aligned}
\sup_{t,\tau\in I}&\|\bm{F}(u;t,\tau);\mathcal{L}(H^{s_1}_x)\|\\
&\overset{\ref{F2}}{\leq} C_{s_1}\exp\left(C_{s_1}\int_{I}\kappa\big(\delta_0^{-1}\|u(\tau)\|_{H^{s_0}_x}\big)\|u(\tau)\|_{H^{s_0}_x}+\kappa_1\big(\delta_0^{-1}\|u(\tau)\|_{H^{3}}\big)\|\partial_\tau u(\tau)\|_{H^{3}_x}\dtau\right)\\
&\leq C_{s_1}\exp\big[C_{s_1,s}|I|(\delta_0+R)\big].
\end{aligned}
\end{equation}
Therefore, we obtain the (even better) $H^{s}_x$-estimate for $\mathscr{G}$:
\begin{equation}\label{G(u,f)_Hs}
\begin{aligned}
\big\|\mathscr{G}(u,f;t);L^\infty_IH^{s}_x\big\|
&\overset{(\ref{exp(tL)_easy})(\ref{Size_RExt})}{\leq} 
C_se^{\lambda|I|}\|f\|_{H^s}
+C_s\exp\big[C_{s}|I|(\delta_0+R)\big]\|f\|_{H^s}\\
&\quad+C_se^{\lambda|I|}|I|\delta_0R
+C_s\exp\big[C_{s}|I|(\delta_0+R)\big]|I|\delta_0R\\
&=:M_s(|I|,R,\|f\|_{H^s}).
\end{aligned}
\end{equation}
We then use \ref{F1} to write down the paradifferential equation satisfied by $\mathscr{G}(u,f;t)$:
\begin{equation}\label{D_tG(u,f)}
\partial_t\mathscr{G}(u,f;t)
=L(D_x)\proj_{\low}{\mathscr{G}}(u,f;t)
-i\proj_{\high}\Gamma_{\Ext}(t)\proj_{\high}\mathscr{G}(u,f;t)
+\R_\Ext(u(t))
\in L^\infty_IH^{s-3/2}_x.
\end{equation} 
Using \ref{Ext1}\ref{Ext2} and inserting (\ref{G(u,f)_Hs}), we find 
\begin{equation}\label{D_tG(u,f)_Hs}
\begin{aligned}
\big\|\partial_t\mathscr{G}(u,f;t);L^\infty_IH^{s-3/2}_x\big\|
&\leq C_sM_s(|I|,R,\|f\|_{H^s})+C_s\delta_0R.
\end{aligned}
\end{equation}
Therefore $\mathscr{G}(u,f;t)$ is indeed still in $W^{1,\infty}_IH^{s-3/2}_x$.

\noindent
\textbf{Step 3: Differentiation in $u$.} 
We then estimate the differential of $\mathscr{G}(u,f;t)$. We claim that with $\tilde{\Gamma}(t):=\proj_{\high} (\der_u\Gamma_{\Ext}(t)\cdot\del u(t))\proj_{\high}$, the following explicit expression is valid:
\begin{equation}\label{D_uG(u,f)}
\begin{aligned}
(\der_u\mathscr{G}&\cdot\del u)(u,f;t)\\
&=e^{tL(D_x)}\int_{t_0}^t e^{-\tau L(D_x)}\proj_{\low}\Big( \der_u\R_{\Ext}(u(\tau))\cdot\bm{\delta} u(\tau) \Big)\dtau
\dtau\\
&\quad+i\int_{t_0}^t \bm{F}(u;t,\tau)\tilde{\Gamma}(\tau)\mathscr{G}(u,f;\tau)\dtau
+\int_{t_0}^t \bm{F}(u;t,\tau)\proj_{\high}\Big( \der_u\R_{\Ext}(u(\tau))\cdot\bm{\delta} u(\tau) \Big)\dtau.
\end{aligned}
\end{equation}
This can be deduced using the explicit formula (\ref{DuF(u)}) through a somewhat tideous computation. However, the simplest method is to directly differentiate (\ref{D_tG(u,f)}) with respect to $u\in W^{1,\infty}_IH^{s-3/2}_x$. Since $s-3/2\geq s_0$, this is legitimate, leading to the equation satisfied by $z(t):=(\der_u\mathscr{G}\cdot\del u)(u,f;t)$:
$$
\partial_tz(t)-L(D_x)z(t)
+i\proj_{\high}\Gamma_{\Ext}(t)\proj_{\high}z(t)
=i\tilde{\Gamma}(t)\mathscr{G}(u,f;t)
+\der_u\R_\Ext(u(t))\cdot\del u(t),
\quad z(t_0)=0.
$$
We then use the explicit integral formula in Item \ref{F1} of Theorem \ref{Fundamental} to obtain (\ref{D_uG(u,f)}).

Therefore, we compute
$$
\begin{aligned}
\sup_{t\in I}\big\|(\der_u&\mathscr{G}\cdot\del u)(u,f;t);H^{s-3/2}_x\big\|\\
&\leq
C_s|I|e^{\lambda|I|}\|\der_u\R_\Ext(u(\tau))\cdot\del u(\tau);H^{s}_x\|\\
&\quad+C_s|I|\sup_{\tau\in I}\left(\big\|\bm{F}(u;t,\tau);\mathcal{L}(H^{s}_x)\big\|
\cdot\|\tilde{\Gamma}(\tau);\mathcal{L}(H^{s}_x,H^{s-3/2}_x)\|
\cdot\|\mathscr{G}(u,f;\tau);H^{s}_x\|\right)\\
&\quad+C_s|I|\sup_{\tau\in I}\left(\big\|\bm{F}(u;t,\tau);\mathcal{L}(H^{s}_x)\big\|\cdot\|\der_u\R_\Ext(u(\tau))\cdot\del u(\tau);H^{s}_x\|\right).
\end{aligned}
$$
If $\big\|u;W^{1,\infty}_IH^{s-3/2}_x\big\|\leq R$, the right-hand side is controlled by
\begin{equation}\label{D_uG(u,f)_bound}
 C_s|I|\cdot \left(e^{\lambda|I|}R+
\exp\big[C_{s}|I|(\delta_0+R)\big]\left(\delta_0^{-1}M(|I|,R,\|f\|_{H^s})+R\right)\right)
\|\del u;C^0_IH^{s-3/2}_x\|.
\end{equation}
Here $M_s(|I|,R,\|f\|_{H^s})$ is the right-hand-side of (\ref{G(u,f)_Hs}). The computation is straightforward: since $s\geq s_0+3/2$, 
$$
\tilde{\Gamma}(t)=\proj_{\high} (\der_u\Gamma_{\Ext}(t)\cdot\del u(t))\proj_{\high}:
H^{s}_x\to H^{s-3/2}_x
$$
is bounded by (\ref{Tilde_Gamma}). The bound for $\big\|\bm{F}(u;t,\tau);\mathcal{L}(H^{s}_x)\big\|$ is (\ref{Factor}). On the other hand, for $u\in H^{s-3/2}_x$, the mapping $\R_\Ext(u)$ takes value in $H^{s+s_0-5/2}_x\subset H^{s+3/2}_x$ and is smooth by (\ref{Size_RExt}). This shows that $\mathscr{G}(u,f;t)$ is indeed $C^1$ on $W^{1,\infty}_IH^{s-3/2}_x$. 

We then observe that in (\ref{D_uG(u,f)_bound}), the coefficient is bounded when $|I|,R,\|f\|_{H^s}$ are bounded. Recalling the expression of $M_s$ in (\ref{G(u,f)_Hs}), we can take $R\gg_s\|f\|_{H^s}$ and $|I|\ll_s R^{-1}$ to make $M_s(|I|,R,\|f\|_{H^s})\leq R$, and guarantee that (\ref{D_uG(u,f)_bound}) is further bounded by $\leq 0.5\|\del u;C^0_IH^{s-3/2}_x\|$. With such choice of $R,|I|$, the mapping $u\mapsto \mathscr{G}(u,f;t)$ is a $C^1$ contraction on the ball $\big\|u;W^{1,\infty}_IH^{s-3/2}_x\big\|\leq R$. By the Banach fixed point theorem, it has a unique fixed point depending in a Lipschitz manner on $f\in H^s$. 

To summarize, given any $f\in H^s$ with $s\geq s_0+3/2$, there is a small time interval $I$ on which the fixed point equation 
$$
u(t)=\mathscr{G}(u,f;t),\quad t\in I
$$
has a unique solution in $W^{1,\infty}_IH^{s-3/2}_x$, depending on $f\in H^s$ in a Lipschitz manner. By (\ref{G(u,f)_Hs}) itself (and \ref{F1}), the solution is automatically of class $C^0_IH^s_x$. Therefore, for any $\delta>0$, the solution map is continuous as a map $H^s\mapsto C^0_IH^{s-\delta}_x$.

\noindent
\textbf{Step 4: A Priori Bound and Global Well-posedness.} 
In Step 1-3, we proved that given any initial value $u(0)\in H^s$, with $s\geq s_0+3/2$, the integral equation (\ref{Duhamel}) and therefore, the extended equation (\ref{EQ_Red_Ext}), has a unique local-in-time solution. Let us set $I$ to be the maximal existence interval, and the solution $u\in C^0_{I;\loc}H^{s}_x\cap W^{1,\infty}_{I;\loc}H^{s-3/2}_x$. 

We now deduce an \emph{a priori estimate} for the solution: for an absolute constant $C>1$ independent from $I$, and $\lambda=\max_\xi \Lambda_\grow(\xi)$ as in \eqref{Max_Lambda}, 
\begin{equation}\label{APriori_s0}
\|u(t)\|_{H^{s_0}_x}
\leq C\delta_0e^{\lambda|t|},\quad t\in I.
\end{equation}
To justify this, we split
$$
I=\{t\in I:\|u(t)\|_{H^{s_0}_x}\leq 4\delta_0\}\cup \{t\in I:\|u(t)\|_{H^{s_0}_x}> 4\delta_0\}
=:L\cup J.
$$
By continuity, $U$ is open in $I$ and is thus a countable disjoint union of intervals; write $J=\cup_{k=1}^\infty J_k$. We notice that due to the truncation $\kappa\big(\delta_0^{-1}\|u\|_{H^{s_0}_x}\big)$, the equation degenerates to $\partial_tu=L(D_x)\proj_{\low}u$ on $J$, therefore on each $J_k$. Thus for $t\in J_k$,
$$
\partial_tu(t)=L(D_x)\proj_{\low}u(t),\quad \big\|u\big|_{\partial J_k}\big\|_{H^{s_0}_x}=4\delta_0.
$$
By (\ref{exp(tL)_easy}), this gives the exponential estimate $\sup_{t\in J_k}e^{-\lambda|t|}\|u(t)\|_{H^{s_0}_x}\leq C\delta_0$. For $t\in L$ the bound $\|u(t)\|_{H^{s_0}_x}\leq 4\delta_0$ obviously is true. We thus complete the proof of \eqref{APriori_s0}. 

Using the equation (\ref{EQ_Red_Ext}) itself, we find
$$
\begin{aligned}
\|\partial_tu(t)\|_{H^{3}_x}
&\leq C\big\|L(D_x)\proj_{\low}u(t)-i\proj_{\high}\Gamma_{\Ext}(t)\proj_{\high}u(t)+\R_\Ext(u(t))\big\|_{H^{s_0-3/2}_x}\\
&\overset{\ref{Ext1}\ref{Ext2}}{\leq} C\|u(t)\|_{H^{s_0}_x}
\overset{(\ref{APriori_s0})}{\leq}
Ce^{\lambda|t|}\delta_0.
\end{aligned}
$$
Therefore, the factor $\mathtt{F}(u;t)$ in the operator bound of $\bm{F}(u;t,t_0)$ (Item \ref{F2} of Theorem \ref{Fundamental}) is $\leq Ce^{\lambda|t|}\delta_0$. Recalling the integral equation (\ref{Duhamel}), this gives the bound
$$
\begin{aligned}
\|u(t)\|_{H^s_x}
&\overset{\ref{F2}}{\leq}
C_s\exp\left(C_se^{\lambda|t|}\delta_0\right)\|u(0)\|_{H^s}
+C_s\exp\left(C_se^{\lambda|t|}\delta_0\right)\int_{[0,t]}\|\R_\Ext(u(\tau))\|_{H^s_x}\dtau\\
&\overset{\ref{Ext2}}{\leq}
C_s\exp\left(C_se^{\lambda|t|}\delta_0\right)\|u(0)\|_{H^s}
+C_s\exp\left(C_se^{\lambda|t|}\delta_0\right)\int_{[0,t]}\|u(\tau)\|_{H^s_x}\dtau,
\end{aligned}
$$
for any $t\in I$. Grönwall's inequality then ensures that $\|u(t)\|_{H^s_x}\leq C_s\exp\left(C_se^{\lambda|t|}\delta_0\right)\|u(0)\|_{H^s}$ for $t\in I$. This forces $I=\xR$ by local well-posedness that we just proved. The proof for Theorem \ref{GWP_EQ} is now complete.

\section{Hyperbolic Invariant Submanifods}\label{Sec5}
In this section, we shall prove Theorem \ref{Main1} by a variant of the Lyapunov-Perron method, which is enabled by the twisted Duhamel formula (\ref{Duhamel}). The issue of losing regularity for quasilinear systems is manifested as the propagator $\bm{F}(u;t,\tau)$ -- linearizing with respect to $u$ loses derivatives, but this loss is compensated by the regularity gain due to the remainder $\R_\DG$. We shall only prove the statement for stable submanifold, since the unstable submanifold is obtained simply by reversing the time direction.

Throughout this section, we shall fix some
\begin{equation}\label{mu_fix}
\mu\in\big(0,\max_{\xi\in(0,\rho^{-1})} \Lambda_\grow(\xi)\big),
\end{equation}
By Remark \ref{RT_Case}, in the case of $\xG=\xT$, when $\mu\leq\min_{\xi\in\xZ:|\xi|<\rho^{-1}}\Lambda_\grow(\xi)$, the subspace $E_\st^\mu$ coincides with $E_\st$ itself. Apart from this, the theory is similar for either $\xG=\xR$ or $\xT$, so we shall not make distinction for these two cases throughout the proof.

\subsection{Decaying Solutions}\label{Subsec_St_2}
Given the global existence result, namely Proposition~\ref{GWP_EQ}, for the extended system (\ref{EQ_Red_Ext}), our construction of stable manifold then resembles the proof of the classical stable manifold theorem, as sketched in Theorem~\ref{Stab_Mfd_Finite}.

We first state two auxiliary results that will be used repeatedly. The first one concerns $e^{tL(D_x)}\proj_{\low}$, true for either $\xR$ or $\xT$ but the proof is non-trivial only for $\xR$. 

\begin{lemma}\label{exp(tL)}
Let $L(D_x)$ be the real linear Fourier multiplier in \ref{Diag1} (equivalently, the operator in \eqref{LEQ(0,0)}). Let $\Lambda_\grow(\xi),\Lambda_\disp(\xi)$ be as in (\ref{Disp_lh}), let $\lambda=\max_\xi \Lambda_\grow(\xi)$ as in \eqref{Max_Lambda}, and fix some $0<\mu<\lambda$. Let $0<\xi_\mu<\xi_\mu'<\rho^{-1}$ and $E_\st^\mu,E_\unst^\mu$ be as in (\ref{Stable_Space})-(\ref{Unstable_Space}).
\begin{itemize}
\item There are two bounded, commuting real linear projections 
$$
\proj_\st^\mu:\mathrm{Ran}\proj_{\low}\to E_\st^\mu,
\quad
\proj_\unst^\mu:\mathrm{Ran}\proj_{\low}\to E_\unst^\mu.
$$
They both commute with $L(D_x)$. There sum is the Fourier projection $\1_{\xi_\mu\leq|\xi|\leq\xi_\mu'}(D_x)$. With all spaces equipped with $L^2$ norm, the operatorial bounds depend only on $\mu$. As operator-valued functions, $\mu\mapsto\proj_\st^\mu,\proj_\unst^\mu$ are left continuous in the strong operator topology.

\item There holds
$$
\begin{aligned}
\|e^{tL(D_x)}\proj_\st^\mu f\|_{H^s}
&\leq C_{s,\mu} e^{-\mu t}\|f\|_{L^2},\quad t\geq0;\\
\big\|e^{tL(D_x)}(\proj_{\low}-\proj_\st^\mu) f\big\|_{H^s}
&\leq C_{s,\mu} e^{-\mu t}\|f\|_{L^2},\quad t\leq0.
\end{aligned}
$$
Similar inequalities hold for $e^{tL(D_x)}\proj_\unst^\mu$.
\end{itemize}
\end{lemma}
\begin{proof}
The proof is an easy computation involving matrix exponentials, avoiding the issue of defining spectral projection for non-normal operators with continuous spectrum. 

The Fourier transform is an isomorphism from $\mathrm{Ran}\proj_{\low}$ to the Lebesgue space
$$
L^2(-2\rho^{-1},2\rho^{-1}),\quad \text{when }\xG=\xR;\quad
L^2(|\xi|<2\rho^{-1},\xi\in\xZ),\quad \text{when }\xG=\xT.
$$
The operator $L(D_x)$ is converted to the pointwise multiplier $L(\xi)$ as in (\ref{LEQ_Coeff}) (see Remark \ref{Complex_u}). Since $\Lambda_\grow(\xi)\geq\mu$ when $\xi_\mu\leq|\xi|\leq\xi_\mu'$, we can diagonalize 
$$
L(\xi)=S(\xi)\begin{pmatrix}
-\Lambda_\grow(\xi) & 0 \\ 0 & \Lambda_\grow(\xi) 
\end{pmatrix}S(\xi)^{-1},
\quad 
S(\xi)\text{ is continuous in }\xi.
$$
We may simply define $\proj_\st^\mu,\proj_\unst^\mu$ as the matrix Fourier multipliers on $L^2$ corresponding to 
$$
S(\xi)\begin{pmatrix}
1 & 0 \\ 0 & 0 
\end{pmatrix}S(\xi)^{-1}\cdot\1_{\xi_\mu\leq|\xi|\leq\xi_\mu'}(\xi),\quad
S(\xi)\begin{pmatrix}
0 & 0 \\ 0 & 1
\end{pmatrix}S(\xi)^{-1}\cdot\1_{\xi_\mu\leq|\xi|\leq\xi_\mu'}(\xi),
$$
respectively. These projections obviously commute with $L(D_x)$, and their sum is obviously the Fourier projection to $\xi_\mu\leq|\xi|\leq\xi_\mu'$. Direct computation shows that the operatorial bounds depend only on $\mu^{-1}$. By the explicit formula of $e^{tL(D_x)}$ in (\ref{Sol_Lin}), the estimate for $e^{tL(D_x)}\proj_{\low}$ and $e^{tL(D_x)}(\proj_{\low}-\proj_\st^\mu)$ follow immediately.

Furthermore, if $\mu\uparrow \mu_0$, then $\xi_{\mu}\uparrow\xi_{\mu_0},\xi_{\mu}'\downarrow\xi_{\mu_0}'$, implying $\1_{\xi_\mu\leq|\xi|\leq\xi_\mu'}(\xi)\downarrow \1_{\xi_{\mu_0}\leq|\xi|\leq\xi_{\mu_0}'}(\xi)$ pointwise. This proves the left continuity in strong operator topology.
\end{proof}

\begin{lemma}\label{Size_Fundamental}
Suppose $u\in L^\infty_{t;\loc}H^{s_0}_x\cap W^{1,\infty}_{t;\loc}H^{s_0-3/2}_x$ verifies the decay estimate 
$$
\|u(t)\|_{H^{s_0}_x}+\|\partial_tu(t)\|_{H^{s_0-3/2}_x}=O(e^{-\mu t}),
\quad t\to+\infty.
$$
Then the propagator $\bm{F}(u;t,t_0)$ satisfies
$$
\sup_{t_0:t_0 \ge t} \big\| \bm{F}(u;t,t_0);\mathcal{L}(H^{s}_x) \big\|
\leq C_{s}\exp\left(C_{s,\mu}\sup_{\tau:\tau\geq t}e^{\mu \tau}\left(\|u(\tau)\|_{H^{s_0}_x}+\|\partial_\tau u(\tau)\|_{H^{s_0-3/2}_x}\right)\right).
$$
A similar estimate holds for differentials of $\bm{F}(u;t,t_0)$. 
\end{lemma}
\begin{proof}
Recall from Theorem \ref{Fundamental} that the propagator $\bm{F}(u;t,t_0)$ satisfies
$$
\|\bm{F}(u;t,t_0);\mathcal{L}(H^s_x)\|
\leq C_s\exp\left(C_s\int_{[t_0,t]}\mathtt{F}(u;
\tau)\dtau\right),
$$
where for some smooth bump function $\kappa_1:\xR\to[0,1]$ that equals 1 on $\mathrm{supp}\kappa$,
$$
\mathtt{F}(u;\tau)
=\kappa\big(\delta_0^{-1}\|u\|_{H^{s_0}_x}\big)\|u(t)\|_{H^{s_0}_x}+\kappa_1\big(\delta_0^{-1}\|u(t)\|_{H^{3}_x}\big)\|\partial_tu(t)\|_{H^{3}_x}.
$$
If $u$ has the decay property, then we find
$$
\int_{[t_0,t]}\mathtt{F}(u;
\tau)\dtau
\leq C\sup_{\tau:\tau\geq t}e^{\mu \tau}\left(\|u(\tau)\|_{H^{s_0}_x}+\|\partial_\tau u(\tau)\|_{H^{s_0-3/2}_x}\right)
\cdot\int_t^{+\infty} e^{-\mu \tau}\dtau
<+\infty.
$$
Recalling Item \ref{F3} of Theorem \ref{Fundamental}, we estimate the differential of $\bm{F}(u;t,t_0)$ in a similar manner. This finishes the proof.
\end{proof}

We state a Lyapunov-Perron type lemma, comparable to the classical integral equation in \eqref{Classical}:
\begin{lemma}\label{Int_Eq_Stab}
Fix $s_0>5$. A solution $u\in C^0_{t;\loc}H^{s_0}_x\cap W^{1,\infty}_{t;\loc}H^{s_0-3/2}_x$ of (\ref{EQ_Red_Ext}) satisfying $\|u(t)\|_{H^{s_0}_x}=O(e^{-\mu t})$ as $t\to+\infty$ must solve the integral equation
\begin{equation}\label{Decay_Int}
\begin{aligned}
u(t)=
e^{tL(D_x)}\proj_\st^\mu u(0)
+\mathscr{A}^\mu(u;t),
\end{aligned}
\end{equation}
where we define 
$$
\begin{aligned}
\mathscr{A}^\mu(u;t)
&:=\mathscr{A}_\st^\mu(u;t)+\mathscr{A}_\unst^\mu(u;t)
+\mathscr{A}_{\high}(u;t),\\
\mathscr{A}_\st^\mu(u;t)
&:=e^{tL(D_x)}\int_{0}^t e^{-\tau L(D_x)}\proj_\st^\mu\R_{\Ext}(u(\tau))\dtau \\
\mathscr{A}_\unst^\mu(u;t)
&=-e^{tL(D_x)}\int_{t}^{+\infty} e^{-\tau L(D_x)}(\proj_{\low}-\proj_\st^\mu)\R_{\Ext}(u(\tau))\dtau\\
\mathscr{A}_{\high}(u;t)
&=-\int_t^{+\infty}\bm{F}(u;t,\tau)\proj_{\high}\R_{\Ext}(u(\tau))\dtau.
\end{aligned}
$$
The integrals $\mathscr{A}_\st^\mu(u;t),\mathscr{A}_\unst^\mu(u;t)
,\mathscr{A}_{\high}(u;t)$ all converge in the topology of $H^{s_0}_x$.
\end{lemma}
\begin{proof}
By the equation (\ref{EQ_Red_Ext}) itself, we conclude that $\partial_tu(t)$ also decays like $O(e^{-\mu t})$ but in the $H^{s_0-3/2}_x$ norm. Therefore, 
$$
\|u(t)\|_{H^{s_0}_x}+\|\partial_tu(t)\|_{H^{s_0-3/2}_x}=O(e^{-\mu t}),
\quad t\to+\infty.
$$
By Lemma \ref{Size_Fundamental}, we have $\big\|\bm{F}(u;t,t_0);\mathcal{L}(H^s_x)\big\|\leq C_{s,\mu}$ for $t_0\geq t\geq0$. We now substitute the decay assumption into the twisted Duhamel formula (\ref{Duhamel}) and project it using $\Id-\proj_\st^{\mu-\delta}$ for any $0<\delta<\mu/2$ (note that $\proj_\st^{\mu-\delta}$ vanishes on frequencies $>\rho^{-1}$): 
\begin{equation}\label{Duhamel_Temp}
\begin{aligned}
(\Id&-\proj_\st^{\mu-\delta}) u(t)\\
&=e^{(t-t_0)L(D_x)}(\proj_{\low}-\proj_\st^{\mu-\delta}) u(t_0)
+\bm{F}(u;t,t_0)\proj_{\high} u(t_0)\\
&\quad+e^{tL(D_x)}\int_{t_0}^t e^{-\tau L(D_x)}(\proj_{\low}-\proj_\st^{\mu-\delta}) \R_{\Ext}(u(\tau))\dtau
+\int_{t_0}^t \bm{F}(u;t,\tau)\proj_{\high}\R_{\Ext}(u(\tau))\dtau.
\end{aligned}
\end{equation}
Fixing $t$ and letting $t_0\to+\infty$, by Lemma \ref{exp(tL)}, we find that the first and second terms in (\ref{Duhamel_Temp}) have $H^{s_0}_x$ norm of size $O(e^{-\delta t_0})$, while the last integral converges in $H^{s_0}_x$ since the integrand have size $O(e^{-2\mu\tau})$ in $H^{s_0}_x$, thanks to the quadratic and smoothing property of $\R_\Ext$ (see the estimate~\eqref{Size_RExt}). As for the first integral in (\ref{Duhamel_Temp}), we re-write it as
$$
(\proj_{\low}-\proj_\st^{\mu-\delta})e^{tL(D_x)}\int_{t_0}^t e^{-\tau L(D_x)}(\proj_{\low}-\proj_\st^{\mu}) \R_{\Ext}(u(\tau))\dtau.
$$
The integrand has $H^{s_0}_x$ norm $O(e^{-\mu \tau})$ by Lemma \eqref{exp(tL)} and \eqref{Size_RExt}, so it converges as $t_0\to+\infty$. Therefore, 
$$
\begin{aligned}
(\Id&-\proj_\st^{\mu-\delta}) u(t)\\
&=-(\proj_{\low}-\proj_\st^{\mu-\delta})e^{tL(D_x)}\int_t^{+\infty} e^{-\tau L(D_x)}(\proj_{\low}-\proj_\st^{\mu}) \R_{\Ext}(u(\tau))\dtau \\
&\quad-\int_t^{+\infty} \bm{F}(u;t,\tau)\proj_{\high}\R_{\Ext}(u(\tau))\dtau.
\end{aligned}
$$
Letting $\delta\downarrow0$, using the left continuity of $\proj_\st^\mu$ in $\mu$ guaranteed by Lemma \ref{exp(tL)}, we finish the proof. 
\end{proof}

Note that in Lemma \ref{Int_Eq_Stab}, there is no assumption on the decay of $\partial_tu$, since it follows directly from the equation satisfied by $u$ itself, forcing $\partial_tu(t)$ to decay in $H^{s_0-3/2}_x$. As a consequence of Lemma~\ref{Int_Eq_Stab}, the integral equation~\eqref{Decay_Int} is what a solution $u(t)$ of (\ref{EQ_Red_Ext}) decaying like $O(e^{-\mu t})$ as $t\to+\infty$ must satisfy. Obviously, a solution $u\in C^0_{t;\loc}H^{s_0}_x\cap W^{1,\infty}_tH^{s_0-3/2}_x$ to (\ref{Decay_Int}) must also solve the paradifferential equation (\ref{EQ_Red_Ext}) by the twisted Duhamel formula (\ref{Duhamel}).

\subsection{The Solution Map}
The aim of this subsection is to prove the converse of Lemma \ref{Int_Eq_Stab}: corresponding to each $f\in E_\st^\mu$ close to zero, there is a \textit{unique} solution $u(t)$ of the integral equation
\begin{equation}\label{Stab_Int}
u(t)=e^{tL(D_x)}\proj_{\low}f+\mathscr{A}^\mu(u;t)
\end{equation}
that decays like $O(e^{-\mu t})$ as $t\to+\infty$. From the twisted Duhamel formula, we find that a function $u\in C^0_{t;\loc}\big([0,+\infty);H^{s_0}_x\big)\cap W^{1,\infty}_{t;\loc}\big([0,+\infty);H^{s_0-3/2}_x\big)$ solves (\ref{Stab_Int}) if and only if it solves the Cauchy problem (\ref{EQ_Red_Ext}) with initial value
$$
u(0)=f-\int_0^{+\infty} e^{-\tau L(D_x)}(\proj_{\low}-\proj_\st^\mu)\R_{\Ext}(u(\tau))\dtau
-\int_0^{+\infty} 
\bm{F}(u;0,\tau)\proj_{\high}\R_{\Ext}(u(\tau))\dtau.
$$
The initial data involves information of $u$ throughout its lifespan; therefore, it is the integral equation (\ref{Stab_Int}) that should be manipulated with.

We now turn to the integral equation (\ref{Stab_Int}). Let us consider the Banach space
\begin{equation}\label{E_gamma_mu}
 \begin{aligned}
&\mathfrak{E}_{\mu}:=\Big\{u \in C^0_{t}([0,+\infty);H^{s_0}_x)\cap W^{1,\infty}_{t}([0,+\infty);H^{s_0-3/2}_x): \\
&\hspace{8em}\E[u]_{\mu} := \sup_{t\ge 0} e^{\mu t}\big(\|u(t)\|_{H^{s_0}_x}+\|\partial_tu(t)\|_{H^{s_0-3/2}_x}\big)<+\infty \Big\}.
\end{aligned}
\end{equation} 
We also introduce the following notations:
\begin{notation}\label{note:OpenBall}
For $f\in H^s$ with $s\in\xR$, we denote the open ball centered at $f$ with radius $r>0$ by
$$
B_s(f,r) := \{ g\in H^s: \|f-g\|_{H^s}<r \}.
$$
Similarly, for all $u\in\mathfrak{E}_{\mu}$, we denote the open ball centered at $u$ with radius $r>0$ by
$$
\mathfrak{B}_{\mu}(u,r) =\left\{v\in\mathfrak{E}_{\mu}: \E[v-u]_{\mu}<r \right\}.
$$
\end{notation}
In the space $\mathfrak{E}_{\mu}$, we are able to solve the integral equation~\eqref{Stab_Int} near zero.

\begin{proposition}\label{Exp_Contraction}
Fix $s_0>5$. Fix an index $\mu$ as in (\ref{mu_fix}). There exist small real numbers $\varepsilon,\varepsilon'>0$ depending on $\mu$ and a unique $C^1$ solution map
\begin{equation}\label{eq-smfd:SolMap}
    \xS_\st^\mu : B_{s_{0}}(0,\varepsilon)\cap E_\st^\mu \to \mathfrak{B}_{\mu}(0,\varepsilon'),\quad f \mapsto u,
\end{equation}
such that, given any function $f\in E_\st^\mu\cap B_{s_{0}}(0,\varepsilon)$ and $u\in\mathfrak{B}_{\mu}(0,\varepsilon')$, $u$ solves the integral equation~\eqref{Stab_Int} if and only if $u = \xS_\st^\mu(f)$. The value $\xS_\st^\mu(f)$ indeed belongs to $C^0_t([0,+\infty);H^s_x)$ for any $s\geq s_0$, and indeed $\|\xS_\st^\mu(f)(t)\|_{H^s_x}=O(e^{-\mu t})$ for any $s\geq s_0$. Moreover, if $s_0-2>3k/2$ for some integer $k\in\xN$, then the solution map $\xS_\st^\mu$ is of class $C^k(B_{s_{0}}(0,\varepsilon)\cap E_\st^\mu;\mathfrak{E}_{\mu})$.
\end{proposition}

\begin{proof}
We emphasize that, since $E_\st^\mu$ consists only of $L^2$ functions with bounded frequency support (in case of $\xT$, it is finite dimensional), any Sobolev norm on it induces the same topology. 

To solve the integral equation~\eqref{Stab_Int}, we will apply the classical implicit function theorem. By Lemma \ref{exp(tL)}, there holds
$$
\E[e^{tL(D_x)}\proj_{\low}f]_{\mu} \lesssim_\mu \|f\|_{L^2}, \quad \forall f\in E_\st^\mu.
$$
So it suffices to check that the map $u \mapsto \mathscr{A}^\mu(u;\cdot)$ belongs to $C^1(\mathfrak{E}_{\mu};\mathfrak{E}_{\mu})$, and its first order differentiation at zero vanishes. 

An initial observation is that $\mathscr{A}^\mu(u;t)$ maps $\mathfrak{E}_{\mu}$ to itself. This follows from estimates for convergent integrals involving exponentially decaying integrands. The computation is similar to but even simpler than the estimate for $\der_u\mathscr{A}^\mu\cdot\del u$ as in (\ref{Ineq_duA_su})(\ref{Ineq_DtduA_su})(\ref{Ineq_d_uA_d(u)})(\ref{Ineq_Dtd_uA_d(u)}) below, so we shall save the computations later.

\noindent
\textbf{Part 1: Differentiability of Hyperbolic Part.} 
We start with the differentiability of $\mathscr{A}_{\st}^\mu(u;t),\mathscr{A}_{\unst}^\mu(u;t)$ with respect to $u\in\mathfrak{E}_{\mu}$. From \ref{Ext2}, we know that $\R_\Ext(u)$ is a smooth tame mapping from $H^{s}_x$ to $H^{s+s_0-2}_x$ when $s\geq s_0$, and vanishes quadratically as $\|u\|_{H^{s_0}_x}\to0$. Clearly, 
\begin{equation}\label{d_uA_Hyp}
\begin{aligned}
(\der_u\mathscr{A}_\st^\mu\cdot \bm{\delta} u)(u;t)
&=e^{tL(D_x)}\int_{0}^t e^{-\tau L(D_x)}\proj_\st^\mu \Big( \der_u\R_{\Ext}(u(\tau))\cdot\bm{\delta} u(\tau) \Big) \dtau \\
(\der_u\mathscr{A}_\unst^\mu\cdot \bm{\delta} u)(u;t)
&=-e^{tL(D_x)}\int_{t}^{+\infty} e^{-\tau L(D_x)}(\proj_{\low}-\proj_\st^\mu)\Big( \der_u\R_{\Ext}(u(\tau))\cdot\bm{\delta} u(\tau) \Big) \dtau.
\end{aligned}
\end{equation}
At any given time $t\geq0$, by Lemma \ref{exp(tL)}, we have, for any $s\in\xR$,
\begin{equation}\label{exp(tL)_mu}
\begin{aligned}
\big\|e^{(t-\tau)L(D_x)}\proj_\st^\mu;\mathcal{L}(H^s_x)\big\|
&\lesssim_{s,\mu} e^{\mu(\tau-t)},
\quad \tau\leq t;\\
\big\|e^{(t-\tau)L(D_x)}(\proj_{\low}-\proj_\st^\mu);\mathcal{L}(H^s_x)\big\|
&\lesssim_{s,\mu} e^{\mu(\tau-t)},
\quad \tau\geq t.
\end{aligned}
\end{equation}
We thus compute
\begin{equation}\label{Ineq_duA_su}
\begin{aligned}
\Big\| (\der_u\mathscr{A}_\st^\mu\cdot \bm{\delta} u)(u;t)\Big\|_{H^{2s_0-2}_x} 
&+ \Big\| (\der_u\mathscr{A}_\unst^\mu\cdot \bm{\delta} u)(u;t)\Big\|_{H^{2s_0-2}_x} \\
&\overset{(\ref{exp(tL)_mu}),\ref{Ext2}}{\leq} 
C_{s_0,\mu}  \int_0^{+\infty}e^{\mu(\tau-t)}
\|u(\tau)\|_{H^{s_0}_x}\|\bm{\delta} u(\tau)\|_{H^{s_0}_x}\dtau\\
&\leq C_{s_0,\mu} e^{-\mu t}\int_0^{+\infty} \E[u]_{\mu}\E[\bm{\delta} u]_{\mu}e^{-\mu\tau} \dtau 
\leq C_{s_0,\mu} e^{-\mu t}\E[u]_{\mu}\E[\bm{\delta} u]_{\mu}.
\end{aligned}
\end{equation}

The time derivatives of these differentials are estimated similarly; for example, a direct computation yields
$$
\partial_t\big((\der_u\mathscr{A}_\st^\mu\cdot \bm{\delta} u)(u;t)\big)
=L(D_x)(\der_u\mathscr{A}_\st^\mu\cdot \bm{\delta} u)(u;t)
+\proj_\st^\mu \Big( \der_u\R_{\Ext}(u(t))\cdot\bm{\delta} u(t) \Big),
$$
and we can simply substitute in the previous inequality to obtain
\begin{equation}\label{Ineq_DtduA_su}
\big\|\partial_t\big((\der_u\mathscr{A}_\st\cdot \bm{\delta} u)(u;t)\big)\big\|_{H^{2s_0-2}_x} 
\leq  C_{s_0,\mu} e^{-\mu t}\E[u]_{\mu}\E[\bm{\delta} u]_{\mu}.
\end{equation}
To summarize, for $u\in \mathfrak{E}_{\mu}$, there holds
\begin{equation}\label{Est_d_uA_Hyp}
\E[(\der_u\mathscr{A}_\st^\mu\cdot \bm{\delta} u)(u;\cdot)]_{\mu}
+\E[(\der_u\mathscr{A}_\unst^\mu\cdot \bm{\delta} u)(u;\cdot)]_{\mu}
\leq C_{s_0,\mu}\big( \E[u]_{\mu} \big) \E[u]_{\mu}\E[\bm{\delta} u]_{\mu}.
\end{equation}
Clearly, $\der\mathscr{A}_{\st}$ and $\der\mathscr{A}_{\unst}$ vanish at $u=0$ since $\R_\Ext(u)$ vanishes quadratically near 0. Their high order derivatives can be deduced in a similar way from the third inequality of~\eqref{R_Ext}, with details omitted. Note that $\mathscr{A}_{\st}^\mu(u;\cdot),\mathscr{A}_{\unst}^\mu(u;\cdot)$ are actually $C^\infty$ in $u$ without loss of derivative, which is not the case for $\mathscr{A}_{\high}(u;\cdot)$ to be studied below. 

\noindent
\textbf{Part 2: Differentiability of Dispersive Part.} 
The major technicality is caused by the dispersive term $\mathscr{A}_{\high}(u;t)$. However, similar to the derivation of \eqref{D_tG(u,f)}, the function $\mathscr{A}_{\high}(u;t)$ solves the linear equation
$$
\begin{aligned}
\partial_t\mathscr{A}_{\high}(u;t)
&=-i\proj_{\high}\Gamma_\Ext(t)\proj_{\high}\mathscr{A}(u;t)
+\proj_{\high}\R_\Ext(u(\tau))\\
\mathscr{A}_{\high}(u;0)&=-\int_0^{+\infty}\bm{F}(u;0,\tau)\proj_{\high}\R_\Ext(u(\tau))\dtau.
\end{aligned}
$$
We can directly linearize this equation with respect to $u$; with 
$$
\tilde{\Gamma}(t)=\proj_{\high} (\der_u\Gamma_{\Ext}(t)\cdot\del u(t))\proj_{\high},
$$ 
the differential $z(t)=(\der_u\mathscr{A}_{\high}\cdot \bm{\delta} u)(u;t)$ then solves the linearized equation 
\begin{equation}\label{D_uA_d_Eq}
\begin{aligned}
\partial_tz(t)
&=-i\proj_{\high}\Gamma_{\Ext}(t)\proj_{\high}z(t)
-i\tilde{\Gamma}(t)\mathscr{A}_{\high}(u;t)
+\der_u\R_\Ext(u(t))\cdot\del u(t),\\
z(0)&=-\int_0^{+\infty}\left[(\der_u\bm{F}\cdot\del u)(u;0,\tau)\proj_{\high}\R_\Ext(u(\tau))
+\bm{F}(u;0,\tau)\proj_{\high}\der_u\R_\Ext(u(\tau))\cdot\del u(\tau)\right]\dtau.
\end{aligned}
\end{equation}
Therefore, using the explicit formula in Item \ref{F1} of Theorem \ref{Fundamental}, we compute
\begin{equation}\label{d_uA_d}
\begin{aligned}
(\der_u\mathscr{A}_{\high}\cdot \bm{\delta} u)(u;t)
&=-i\int_t^{+\infty}\left[\int_\tau^t \bm{F}(u;t,\tau_1)\tilde{\Gamma}(\tau_1)\bm{F}(u;\tau_1,\tau)\dtau_1\right]
\proj_{\high} \R_{\Ext}(u(\tau)) \dtau\\
&\quad-\int_t^{+\infty}\bm{F}(u;t,\tau)
\proj_{\high}\Big( \der_u\R_{\Ext}(u(\tau))\cdot\bm{\delta} u(\tau) \Big) \dtau.
\end{aligned}
\end{equation}
The time derivative of $\der_u\mathscr{A}_{\high}\cdot \bm{\delta} u$ is then computed directly using (\ref{D_uA_d_Eq}). 

We then make use of the norm estimate for the propagator $\bm{F}(u;t,\tau)$ in Lemma \ref{Size_Fundamental}: for any $s\in\xR$ and any time $t\geq0$,
$$
\sup_{t_0:t_0 \ge t} \big\| \bm{F}(u;t,t_0);\mathcal{L}(H^{s}_x) \big\|
\leq C_{s}\exp\left(C_{s,a}\E[u]_{\mu}\right).
$$
On the other hand, using \ref{Ext1}, recalling the notation $\tilde{\Gamma}_{\Ext}(t)=\proj_{\high} (\der_u\Gamma_{\Ext}(t)\cdot\del u(t))\proj_{\high},$, we have, for any $s\in\xR$,
$$
\big\|\tilde{\Gamma}_{\Ext}(t);\mathcal{L}(H^{s}_x,H^{s-3/2}_x)\big\|
\leq C_{s_0,s}\delta_0^{-1}\|\del u(t)\|_{H^{s_0}_x}
\leq C_{s_0,s}\delta_0^{-1}\E[\del u]_{\mu}e^{-\mu t}.
$$
With these in mind, we obtain the following estimate:
\begin{equation}\label{Ineq_d_uA_d(u)}
\begin{aligned}
\big\|(\der_u\mathscr{A}_{\high}&\cdot\del u)(u;t)\big\|_{H^{2s_0-7/2}_x} \\
&\overset{(\ref{Size_RExt})}{\leq} C_{s_0,\mu}\delta_0^{-1}\exp\left(C_{s_0,\mu}\E[u]_{\mu}\right)
\int_{t}^{+\infty}\left(\int_t^\tau\E[\bm{\delta} u]_{\mu}e^{-\mu\tau_1}\dtau_1\right)
\E[u]_{\mu}e^{-2\mu\tau}\dtau\\
&\quad
+C_{s_0,\mu}\exp\left(C_{s_0,\mu}\E[u]_{\mu}\right)
\int_{t}^{+\infty}\E[u]_{\mu}\E[\del u]_{\mu}e^{-2\mu\tau}\dtau\\
&\leq
C_{s_0,\mu}\delta_0^{-1}\exp\left(C_{s_0,\mu}\E[u]_{\mu}\right) \E[u]_{\mu}\E[\del u]_{\mu}e^{-2\mu t}.
\end{aligned}
\end{equation}
As commented in the beginning of this section, a key in the proof of (\ref{Ineq_d_uA_d(u)}) is that the differential $\der_u\mathscr{A}_{\high}\cdot\del u$ is still of $H^{s_0}_x$ regularity in $x$: even though the expression (\ref{d_uA_d}) involves 3/2 more derivative in $x$, this loss of regularity can be compensated by the $s_0-2$ gain of regularity due to $\R_\Ext(u)$.

Substituting (\ref{Ineq_d_uA_d(u)}) into (\ref{D_uA_d_Eq}), we obtain the estimate for the time derivative:
\begin{equation}\label{Ineq_Dtd_uA_d(u)}
\begin{aligned}
\Big\|\partial_t\big((\der_u\mathscr{A}_{\high}\cdot\del u)(u;t)\big)\Big\|_{H^{2s_0-5}_x}
\leq C_{s_0,\mu}\delta_0^{-1}\exp\left(C_{s_0,\mu}\E[u]_{\mu}\right) \E[u]_{\mu}\E[\bm{\delta} u]_{\mu}e^{-2\mu t}.
\end{aligned}
\end{equation}
Summarizing~\eqref{Ineq_d_uA_d(u)},~\eqref{Ineq_Dtd_uA_d(u)}, and using the condition $s_0>5$, we have
\begin{equation}\label{Est_d_uA_d(u)}
\E[(\der_u\mathscr{A}_{\high}\cdot\del u)(u;t)]_{\mu}
\leq C_{s_0,\mu}\delta_0^{-1}\exp\big(C_{s_0,\mu}\E[u]_{\mu} \big) \E[u]_{\mu} \E[\bm{\delta} u]_{\mu}.
\end{equation}
Clearly $(\der_u\mathscr{A}_{\high}\cdot\del u)(u;t)$ vanishes when $u=0$. For higher order derivatives, since the derivative of the propagator $\bm{F}(u;t,\tau)$ in $u$ leads to $3/2$-loss of derivative, we need the condition $s_0-2>3k/2$ to ensure that its $k$-order derivative converges in $\mathfrak{E}_{\mu}$.

\noindent
\textbf{Conclusion.} 
The nonlinear map $\mathscr{A}^\mu(u;\cdot)$ is $C^k$ in $u\in\mathfrak{E}_{\mu}$ with $\der_u\mathscr{A}^\mu(u;\cdot)$ vanishing as $u=0$. By the classical implicit function theorem, the integral equation~\eqref{Stab_Int} admits a unique solution, represented by $u = \xS_\st^\mu(f)$ for some $C^k$ map $\xS_\st^\mu$, defined for $f\in B_{s_{0}}(0,\varepsilon)\cap E_\st^\mu$, with $\varepsilon$ suitably small (depending on $\mu$). The inequalities (\ref{Ineq_duA_su})(\ref{Ineq_d_uA_d(u)}) implies better spatial regularity of the solution $\xS_\st^\mu(f)\in C^0_{t\geq0}H^{2s_0-7/2}_x$, forcing it to decay like $O(e^{-\mu t})$ in $H^{2s_0-7/2}_x$. This is due to the regularity gain of $\R_\Ext$. We can thus argue inductively to show that $\xS_\st^\mu(f)\in C^0_{t\geq0}H^{s}_x$ and $\|u(t)\|_{H^s_x}=O(e^{-\mu t})$ for every $s\geq s_0$. Moreover, the solution map $\xS_\st^\mu$ is unique if the solution is assumed to be small (namely, belonging to some small ball $\mathfrak{B}_\mu(0,\varepsilon')$). 
\end{proof}

\subsection{The Stable and Unstable Manifold}\label{Subsec_St_3}
We have all the ingredients needed to prove Theorem~\ref{Main1}. In this subsection, we will construct the local stable manifold, while the unstable manifold is constructed simply by reversing time. Note that by Remark \ref{RT_Case}, in case of periodic boundary condition, we can simply take $\mu$ as in (\ref{mu_fix}) to ensure $E_\st^\mu=E_\st$. Therefore, we shall prove Theorem~\ref{Main1} by a unified argument.

\noindent
\textbf{Construction and tangent space.} 
With the solution map $\xS_\st^\mu$ constructed in Proposition \ref{eq-smfd:SolMap}, the local stable manifold $M_\st^\mu$ is then defined as the set
    \begin{equation}\label{eq:StabMfd}
        M_\st^\mu:=\left\{\xS_\st(f)\big|_{t=0}:f\in E_\st^\mu,\|f\|_{H^{s_0}}< \varepsilon\right\}.
    \end{equation}
Note that $E_\st^\mu$ is spectrally bounded, so any Sobolev norm on it induces the same topology. Equivalently, $M_\st^\mu$ can be regarded as the graph of the mapping
$$
\begin{aligned}
E_\st^\mu\cap B_{s_0}(0,\varepsilon)
&\mapsto 
&& \mathrm{Ran}\big((\proj_{\low}-\proj_\st^\mu)+\proj_{\high}\big)\cap H^{s_0}, \\
f&\mapsto && \big((\proj_{\low}-\proj_\st^\mu)+\proj_{\high}\big)\mathscr{A}^\mu(\xS_\st(f);0).
\end{aligned}
$$
In the proof of Proposition~\ref{Exp_Contraction}, we have also shown $\der_u \mathscr{A}^\mu(0)=0$, yielding that
$\der_f\big(\mathscr{A}^\mu(\xS_\st(f);0)\big)$ vanishes at $f=0$. In other words, the tangent space of $M_\st^\mu$ at the point 0 is just the subspace $E_\st^\mu$. 
\vspace{1em}

\noindent
\textbf{Invariance.} To show that $M_\st^\mu$ is an \emph{invariant set} for (\ref{EQ_Red_Ext}), it suffices to show that for any time $t_1\in\xR$ and any solution $u=\xS_\st^\mu(f)$ with $f\in E_\st^\mu\cap B_{s_0}(0,\varepsilon)$, the value $u(t_1)$ is still on $M_\st^\mu$ (unless $\|\proj_\st u(t_1)\|_{H^{s_0}_x} \ge \varepsilon$). Indeed, the solution $u_1(t)$ of (\ref{EQ_Red_Ext}) with $u_1(0)=u(t_1)$ is explicitly determined by $u_1(t)=u(t_1+t)$, therefore still satisfy the decay condition $\|u(t)\|_{H^{s_0}_x}=O(e^{-\mu t})$. By Lemma \ref{Int_Eq_Stab}, $u_1(t)$ must still solve the integral equation (\ref{Decay_Int}), just with $u(0)$ replaced by $u(t_1)$. By Proposition \ref{Exp_Contraction}, this forces $u_1=\xS_\st^\mu(\proj_\st^\mu u(t_1))$, so $u(t_1)$ must still be on the stable manifold $M_\st^\mu$.

\vspace{1em}

\noindent
\textbf{Regularity and topology.} It remains to prove that $M_\st^\mu$ is a $C^\infty$ submanifold with respect to $H^s$ topology for any $s\geq s_0$. As a consequence, $M_\st^\mu$ is contained in the space of smooth functions $H^\infty := \cap_{s\geq s_0}H^s$ and the exponential decay occurs with respect to all the $H^s$ norms. 

To begin with, we fix some large $s_0 \gg 1$. Then Proposition~\ref{Exp_Contraction} guarantees that $f\mapsto \mathscr{A}^\mu(\xS_\st^\mu(f);0)$ is at least a $C^1$ mapping, and therefore $M_\st^\mu$ is a $C^1$ submanifold of $H^{s_0}$. To improve the level of regularity and Sobolev index, let us replace $s_0$ by a larger index $\tilde{s}_0 := s_0 + 5k/2$ where 
$k\in\xN_+$ is an arbitrary integer. The truncation parameter $\delta_0$ in~\eqref{EQ_Red_Ext} and the radii $\varepsilon,\varepsilon'$ in~\eqref{eq-smfd:SolMap} have to be replaced by smaller ones. Repeating the arguments above, we can find a new solution map $\tilde{\xS}_\st^\mu$, which coincides with the previous one above at least for small $f\in E_\st^\mu$ due to the uniqueness part in Proposition~\ref{Exp_Contraction}. In what follows, we will extend the regularity of $\tilde{\xS}_\st^\mu$ to $\xS_\st^\mu$ through the flow map of equation~\eqref{EQ_Red_Ext}. The key observation is that, even through the equation is quasilinear, the loss of derivatives has no impact on the topology of $M_\st^\mu$ since it is a graph above $E_\st^\mu$, a space of functions with bounded frequency support.

Consider any open subset $I$ of $\{ f\in E_\st^\mu: \|f\|_{H^{s_0}}<\varepsilon \}$ and its image $N\subset M_\st^\mu$, which is a $C^1$ submanifold of $M_\st^\mu$ with respect to $H^{s_0}$ topology. Then we apply the flow map of equation~\eqref{EQ_Red_Ext} and obtain another submanifold $\tilde{N}$ whose projection to $E_\st^\mu$, denoted by $\tilde{I}$, lies in the domain of the new solution map $\tilde{\xS}_\st^\mu$. From Proposition~\ref{Exp_Contraction}, $\tilde{N}$ itself is a $C^{k}$ submanifold of $H^{\tilde{s}_0}$. The discussion is summed up in the commutative diagram below.

\begin{figure}[h]
\begin{center}

\begin{tikzpicture}[x=0.75pt,y=0.75pt,yscale=-1,xscale=1]

\draw    (39.6,210.8) -- (327.6,210.8) ;
\draw [shift={(329.6,210.8)}, rotate = 180] [color={rgb, 255:red, 0; green, 0; blue, 0 }  ][line width=0.75]    (10.93,-3.29) .. controls (6.95,-1.4) and (3.31,-0.3) .. (0,0) .. controls (3.31,0.3) and (6.95,1.4) .. (10.93,3.29)   ;
\draw    (59.6,230.8) -- (59.6,72.8) ;
\draw [shift={(59.6,70.8)}, rotate = 90] [color={rgb, 255:red, 0; green, 0; blue, 0 }  ][line width=0.75]    (10.93,-3.29) .. controls (6.95,-1.4) and (3.31,-0.3) .. (0,0) .. controls (3.31,0.3) and (6.95,1.4) .. (10.93,3.29)   ;
\draw [color={rgb, 255:red, 74; green, 144; blue, 226 }  ,draw opacity=1 ][line width=1.5]    (32.27,206.63) .. controls (33.27,206.91) and (38.38,210.91) .. (60.6,210.47)(59.6,210.8) .. controls (80.82,210.69) and (269.27,184.47) .. (312.6,122.47) ;
\draw  [dash pattern={on 0.84pt off 2.51pt}]  (289.6,144.97) -- (289.6,210.8) ;
\draw  [dash pattern={on 0.84pt off 2.51pt}]  (259.6,162.3) -- (259.6,210.8) ;
\draw  [dash pattern={on 0.84pt off 2.51pt}]  (169.6,193.63) -- (169.6,210.8) ;
\draw  [dash pattern={on 0.84pt off 2.51pt}]  (129.6,200.63) -- (129.6,210.8) ;
\draw    (440,91.1) -- (558,91.1) ;
\draw [shift={(560,101.1)}, rotate = 180] [color={rgb, 255:red, 0; green, 0; blue, 0 }  ][line width=0.75]    (10.93,-3.29) .. controls (6.95,-1.4) and (3.31,-0.3) .. (0,0) .. controls (3.31,0.3) and (6.95,1.4) .. (10.93,3.29)   ;
\draw  [dash pattern={on 4.5pt off 4.5pt}]  (440,223.6) -- (558,223.6) ;
\draw [shift={(560,223.6)}, rotate = 180] [color={rgb, 255:red, 0; green, 0; blue, 0 }  ][line width=0.75]    (10.93,-3.29) .. controls (6.95,-1.4) and (3.31,-0.3) .. (0,0) .. controls (3.31,0.3) and (6.95,1.4) .. (10.93,3.29)   ;
\draw    (415,114.6) -- (415,190.6) ;
\draw [shift={(415,192.6)}, rotate = 270] [color={rgb, 255:red, 0; green, 0; blue, 0 }  ][line width=0.75]    (10.93,-3.29) .. controls (6.95,-1.4) and (3.31,-0.3) .. (0,0) .. controls (3.31,0.3) and (6.95,1.4) .. (10.93,3.29)   ;
\draw    (425,192.6) -- (425,116.6) ;
\draw [shift={(425,114.6)}, rotate = 90] [color={rgb, 255:red, 0; green, 0; blue, 0 }  ][line width=0.75]    (10.93,-3.29) .. controls (6.95,-1.4) and (3.31,-0.3) .. (0,0) .. controls (3.31,0.3) and (6.95,1.4) .. (10.93,3.29)   ;
\draw    (572.33,114.6) -- (572.33,190.6) ;
\draw [shift={(572.33,192.6)}, rotate = 270] [color={rgb, 255:red, 0; green, 0; blue, 0 }  ][line width=0.75]    (10.93,-3.29) .. controls (6.95,-1.4) and (3.31,-0.3) .. (0,0) .. controls (3.31,0.3) and (6.95,1.4) .. (10.93,3.29)   ;
\draw    (582.33,192.93) -- (582.33,116.93) ;
\draw [shift={(582.33,114.93)}, rotate = 90] [color={rgb, 255:red, 0; green, 0; blue, 0 }  ][line width=0.75]    (10.93,-3.29) .. controls (6.95,-1.4) and (3.31,-0.3) .. (0,0) .. controls (3.31,0.3) and (6.95,1.4) .. (10.93,3.29)   ;
\draw    (560,101.1) -- (442,101.1) ;
\draw [shift={(440,91.1)}, rotate = 360] [color={rgb, 255:red, 0; green, 0; blue, 0 }  ][line width=0.75]    (10.93,-3.29) .. controls (6.95,-1.4) and (3.31,-0.3) .. (0,0) .. controls (3.31,0.3) and (6.95,1.4) .. (10.93,3.29)   ;
\draw  [dash pattern={on 4.5pt off 4.5pt}]  (560,213.6) -- (442,213.6) ;
\draw [shift={(440,213.6)}, rotate = 360] [color={rgb, 255:red, 0; green, 0; blue, 0 }  ][line width=0.75]    (10.93,-3.29) .. controls (6.95,-1.4) and (3.31,-0.3) .. (0,0) .. controls (3.31,0.3) and (6.95,1.4) .. (10.93,3.29)   ;

\draw (330,219.53) node [anchor=north west][inner sep=0.75pt]    {$E_{\st}^\mu$};
\draw (70,70) node [anchor=north west][inner sep=0.75pt]    {$(E_\st^\mu)^\perp$};
\draw (264.27,130) node [anchor=north west][inner sep=0.75pt]    {$N$};
\draw (270.6,220) node [anchor=north west][inner sep=0.75pt]    {$I$};
\draw (143.27,170) node [anchor=north west][inner sep=0.75pt]    {$\tilde{N}$};
\draw (142.27,217) node [anchor=north west][inner sep=0.75pt]    {$\tilde{I}$};
\draw (316.6,102.53) node [anchor=north west][inner sep=0.75pt]  [color={rgb, 255:red, 74; green, 144; blue, 226 }  ,opacity=1 ]  {$M_{\st}^\mu$};
\draw (412,87) node [anchor=north west][inner sep=0.75pt]    {$\tilde{N}$};
\draw (572,90) node [anchor=north west][inner sep=0.75pt]    {$N$};
\draw (412,207) node [anchor=north west][inner sep=0.75pt]    {$\tilde{I}$};
\draw (572,210) node [anchor=north west][inner sep=0.75pt]    {$I$};
\draw (383.47,144) node [anchor=north west][inner sep=0.75pt]    {$\proj _{\st}^\mu$};
\draw (438.93,142) node [anchor=north west][inner sep=0.75pt]    {$\tilde{\xS}_\st^\mu$};
\draw (540,144) node [anchor=north west][inner sep=0.75pt]    {$\proj _{\st}^\mu$};
\draw (594.27,144) node [anchor=north west][inner sep=0.75pt]    {$\xS_\st^\mu$};
\draw (465.67,68) node [anchor=north west][inner sep=0.75pt]   [align=left] {Flow map};

\end{tikzpicture}
\end{center}
\end{figure}

Note that the flow map of equation~\eqref{EQ_Red_Ext}, even being bijective, cannot serve as a $C^{k}$ diffeomorphism in the topology $H^{\tilde{s}_0}$ due to quasilinearity. Indeed, as a $C^{k}$ mapping, it loses $3k/2$ derivatives. However, we still read from the commutative diagram that the bijections between $I$ and $\tilde{I}$ are $C^{k}$. Indeed, these bijections can be decomposed as following:
\begin{gather*}
    I \xrightarrow[C^1]{\xS_\st^\mu} N \subset H^{s_0} \xrightarrow[C^1]{\text{flow map}} \tilde{N} \subset H^{s_0-3/2} \xrightarrow[C^\infty]{\proj_\st^\mu} \tilde{I}; \\
    \tilde{I} \xrightarrow[C^{k}]{\tilde{\xS}_\st} \tilde{N} \subset H^{\tilde{s}_0} \xrightarrow[C^{k}]{\text{flow map}} N \subset H^{s_0+k} \xrightarrow[C^\infty]{\proj_\st^\mu} I.
\end{gather*}
Note that these ensure the mapping $I \to \tilde{I}$ to be $C^k$ (not only $C^1$). Then we decompose the solution map $\xS_\st:I \to N$ by the following diagram:
\begin{equation*}
    I \xrightarrow[C^k]{\text{as above}} \tilde{I} \xrightarrow[C^k]{\tilde{\xS}_\st^\mu}\tilde{N} \subset H^{\tilde{s}_0} \xrightarrow[C^k]{\text{flow map}} N \subset H^{\tilde{s}_0-3k/2} = H^{s_0+k}.
\end{equation*}
In conclusion, for arbitrary open set $N\subset M_\st$ and $k\in\xN_+$, $N$ is a $C^k$ submanifold of $H^{s_0+k}$, which completes the proof of Theorem~\ref{Main1}.

\section{Mild Growth and Center Invariant Set}\label{Sec6}
This final section is devoted to the proof of Theorem \ref{Main2}. We focus on the periodic boundary condition case, namely the analysis for (\ref{EQ_Red_Ext}) (and its equivalent form (\ref{Twist_Duhamel})) takes place on $\xT$. We shall construct the center invariant set $M_\cen$ in (\ref{M_Center}), therefore completing the proof of Theorem~\ref{Main2}. Unfortunately, due to the quasilinear nature of the system, we are not able to show that the center invariant set is a submanifold. However, the set is, in a proper sense, \emph{tangent to $E_\disp$} at the origin (see Corollary~\ref{cor-center:Tgt} for the rigorous statement). Therefore, solutions on the set $M_\cen$ can be considered as corresponding to short wave perturbations, exhibiting stability within a sufficiently long time, as observed experimentally (see Subsection \ref{Facts}). The trivial lifespan estimate in Proposition \ref{prop-center:TrivLSp} is far from optimal: a normal form argument should indicate much better results, as discussed in Subsection \ref{Perspect}.

\subsection{Solutions of Mild Growth}
In this section, we shall fix 
$$
\mu=\min_{\xi\in\xZ,0<|\xi|<\rho^{-1}}\Lambda_\grow(\xi).
$$
As in Remark \ref{RT_Case}, we can drop the superscript $\mu$ in the following. We shall write $\proj_\st,\proj_\unst,\proj_\disp$ for the spectral projections to $E_\st,E_\unst,E_\disp$ corresponding to $L(D_x)$ respectively; they are all bounded operators, with $\proj_\st,\proj_\unst$ having finite rank, while $\proj_\disp$ is simply a Fourier projection. We therefore have
$$
L^2(\xT;\xC)=E_\st\oplus E_\unst\oplus E_\disp.
$$
For the propagator $\bm{F}(u;t,t_0)$, we shall make a little abuse of notation: we require it to also encode the evolution on the frequencies $\xi=0$ and $|\xi|\geq\rho^{-1}$, i.e. it shall instead be the fundamental solution corresponding to
$$
\partial_t-L(D_x)\big(\proj_{\xi=0}+\proj_{\rho^{-1}\leq|\xi|\leq2\rho^{-1}}\big)
+i\proj_{\high}\Gamma_\Ext\proj_{\high}.
$$
Note that by (\ref{Sol_Lin}), the solution operator $e^{tL(D_x)}\big(\proj_{\xi=0}+\proj_{|\xi|=\rho^{-1}}\big)$ is indeed linear in $t$. 

We start by relaxing the condition on the asymptotic behavior of the solution $u(t)$ to (\ref{EQ_Red_Ext}) when $t\to\pm\infty$. Under the new convention in this section, we re-write (\ref{EQ_Red_Ext}) as
\begin{equation}\label{EQ_Red_Ext_New}
\partial_t u - L(D_x)\big(\proj_\st+\proj_\unst\big)u
-L(D_x)\big(\proj_{\xi=0}+\proj_{\rho^{-1}\leq|\xi|\leq2\rho^{-1}}\big)u
+ i\proj_{\high}\Gamma_{\Ext}\proj_{\high}u= \R_\Ext(u).
\end{equation}
As shown in Theorem \ref{GWP_EQ}, this system is globally well-posed in $H^s$ for $s\geq s_0+3/2$, since it involves a truncation and degenerates to a simple system when $\|u\|_{H^{s_0}}\geq4\delta_0$. However, instead of decaying, we assume only the mild growth condition $\|u(t)\|_{H^{s_0}_x}=o(e^{\mu|t|})$. We prove the following Lyapunov-Perron type lemma:

\begin{lemma}\label{Lem_Center}
Fix some $s_1\geq s_0+3/2$. Consider a function $u\in C^0_{t;\loc}H^{s_1}_x\cap W^{1,\infty}_{t;\loc}H^{s_1-3/2}_x$. If $\delta_0>0$ is chosen sufficiently small\footnote{Within this section, it is of no harm allowing $\delta_0$ to vary, since no differentiation in $u$ is involved; the bounds of $\bm{F}$ itself does not involve $\delta_0^{-1}$ as seen from Theorem \ref{Fundamental}.} (independent from $u$), then the following assertions are equivalent -- note that only $H^{s_0}_x$ norms are involved instead of $H^{s_1}_x$:
\begin{enumerate}
\item $u$ solves the equation (\ref{EQ_Red_Ext}) (equivalently \eqref{EQ_Red_Ext_New}) with parameter $\delta_0>0$ with mild growth in time:
\begin{equation}\label{eq-center:MildGrow}
\| u(t) \|_{H^{s_0}_x} = o(e^{\mu|t|}), \quad \text{as }t\to\pm\infty.
\end{equation}

\item $u$ solves the integral equation
\begin{equation}\label{Int_Eq_Center}
u(t)
=\bm{F}(u;t,0)\proj_\disp u(0)
+\mathscr{B}_\disp(u;t)+\mathscr{B}_\st(u;t)+\mathscr{B}_\unst(u;t),
\end{equation}
where
$$
\begin{aligned}
\mathscr{B}_\disp(u;t)
&:=\int_0^t \bm{F}(u;t,\tau) \proj_\disp\R_\Ext(u(\tau))\dtau\\
\mathscr{B}_\st(u;t)
&:=e^{tL(D_x)}\int_{-\infty}^t e^{-\tau L(D_x)}\proj_\st\R_\Ext(u(\tau))\dtau\\
\mathscr{B}_\unst(u;t)
&:=-e^{t L(D_x)}\int_t^{+\infty} e^{-\tau L(D_x)}
\proj_\unst\R_\Ext(u(\tau))\dtau.
\end{aligned}
$$

\item $u$ solves the equation \eqref{EQ_Red_Ext} (equivalently \eqref{EQ_Red_Ext_New}) with parameter $\delta_0>0$ and the initial value of $u$ satisfies
\begin{equation}\label{Eq_Initial_Center}
(\proj_{\st}+\proj_{\unst})u(0)
=\int_{-\infty}^0 e^{-\tau L(D_x)}\proj_\st\R_\Ext(u(\tau))\dtau
-\int_0^{+\infty} e^{-\tau L(D_x)}
\proj_\unst\R_\Ext(u(\tau))\dtau.
\end{equation}
\end{enumerate}
\end{lemma}

\begin{remark}
As sketched in Theorem \ref{Stab_Mfd_Finite}, different truncations lead to possibly different center manifolds. Therefore, even in the finite dimensional case, we usually only refer to \emph{a local center manifold} since it is not necessarily unique. 
\end{remark}

\begin{proof}
We first show that (1) implies (2). If $u$ solves the equation \eqref{EQ_Red_Ext_New}, then it can be expressed by the twisted Duhamel formula~\eqref{Duhamel}. Then, we apply on both sides the projection $\bm{\Pi}_\st$ (resp. $\bm{\Pi}_\unst$) to stable (resp. unstable) directions and take the limit $t_0 \to -\infty$ (resp. $t_0 \to +\infty$). By Lemma \ref{exp(tL)} and the assumption $\|u(t)\|_{H^{s_0}_x}=o(e^{\mu|t|})$, the term $e^{t_0L(D_x)}\proj_\unst u(t_0)\to0$ in $H^{s_0}_x$ when $t_0 \to +\infty$, while $e^{t_0L(D_x)}\proj_\st u(t_0)\to0$ in $H^{s_0}_x$ when $t_0 \to -\infty$. Therefore,
\begin{equation*}
\begin{aligned}
\bm{\Pi}_\st u(t) &= e^{tL(D_x)}\int_{-\infty}^t e^{-\tau L(D_x)}\bm{\Pi}_\st\R_{\Ext}(u(\tau))\dtau, \\
\bm{\Pi}_\unst u(t) &= -e^{tL(D_x)}\int_t^{+\infty} e^{-\tau L(D_x)}\bm{\Pi}_\unst\R_{\Ext}(u(\tau))\dtau,
\end{aligned}
\end{equation*}
where the integrals converge in $H^{s_0}_x$ for all time $t$. Then we evaluate these formulas at $t=0$ and substitute them into the twisted Duhamel formula~\eqref{Duhamel} to conclude the integral equation~\eqref{Int_Eq_Center}. 

Taking $t=0$ in (2), we immediately see that (2) implies (3). Using the twisted Duhamel formula (\ref{Twist_Duhamel}), we immediately see that (3) implies (2). 

We now show that $(3)$ yields $(1)$. Suppose $u\in C^0_{t;\loc}H^{s_0}_x\cap W^{1,\infty}_{t;\loc}H^{s_0-3/2}_x$ solves the integral equation (\ref{Int_Eq_Center}). We first note that regardless of how $u(t)$ behaves for large $t$, the two infinite integrals $\mathscr{B}_\st,\mathscr{B}_\unst$ always satisfy
\begin{equation}\label{Hyp_Quad}
\begin{aligned}
\|\mathscr{B}_\st(u;t)+\mathscr{B}_\unst(u;t)\|_{H^{s_0}_x}
&\overset{\text{Lem.}\ref{exp(tL)}}{\leq}
\int_{-\infty}^{+\infty} e^{-|t-\tau|\mu}\big\|(\proj_\st+\proj_\unst)\R(u(\tau))\big\|_{H^{s_0}_x}\dtau\\
&\overset{\ref{Ext2}}{\leq}
C_{s_0}\int_{-\infty}^{+\infty} e^{-|t-\tau|\mu}\delta_0^2\dtau
\leq C_{s_0}\delta_0^2.
\end{aligned}
\end{equation}
It is therefore legitimate to differentiate them in $t$. Furthermore, by Theorem \ref{Fundamental} (and the discussion for \eqref{EQ_Red_Ext_New} above), since $u\in C^0_{t\loc}H^{s_0}_x$ and $\partial_tu\in L^\infty_{t;\loc}H^{s_0-3/2}_x$, it is legitimate to differentiate $\bm{F}(u;t,\tau)$ in $t$ according to \ref{F1}. Therefore, $u$ must solve the paradifferential equation (\ref{EQ_Red_Ext_New}).

It remains to show that $u(t)$ has mild growth in time if $\delta_0>0$ is sufficiently small (independent from $u$). Indeed, from the equation $\partial_t u - L(D_x)\proj_{\low}u +i\proj_{\high}\Gamma_{\Ext}\proj_{\high}u= \R_\Ext(u)$ itself, we see that 
$$
\big\|(\proj_\st+\proj_\unst)\partial_tu(t)\big\|_{H^{3}_x}
=\|L(D_x)(\mathscr{G}_\st(u;t)+\mathscr{G}_\unst(u;t))\|_{H^{3}_x}
\leq C_{s_0}\delta_0^2,
$$
while 
$$
\begin{aligned}
\|\proj_\disp\partial_tu(t)\|_{H^{3}}
&\leq \|-i\proj_{\high}\Gamma_{\Ext}(t)\proj_{\high}u(t)+\R_\Ext(u(t))\|_{H^{s_0-3/2}}\\
&\overset{\ref{Ext1}\ref{Ext2}}{\leq}
C_{s_0}\kappa\big(\delta_0^{-1}\|u(t)\|_{H^{s_0}_x}\big)\|u(t)\|_{H^{s_0}_x}
\leq C_{s_0}\delta_0^2.
\end{aligned}
$$
Therefore, the factor $\mathtt{F}(u;t)$ in \ref{F2} satisfies $\mathtt{F}(u;t)\leq C_{s_0}\delta_0$. If we choose $\delta_0$ so small that $C_{s_0}\delta_0<\mu$, then according to \ref{F2} (and also the linear growth on frequencies $\xi=0$ or $\rho^{-1}\leq|\xi|\leq2\rho^{-1}$), there holds
$$
\|\bm{F}(u;t,t_0);\mathcal{L}(H^{s_0}_x)\|
\leq C_{s_0}\size[t-t_0]
+C_{s_0}\exp\left(C_{s_0}\int_{[t_0,t]}\mathtt{F}(u;
\tau)\dtau\right)
\leq C_{s_0,\mu}e^{(\mu-\delta)|t-t_0|}.
$$
Substituting (\ref{Hyp_Quad}) and this estimate of $\bm{F}$ into the integral equation (\ref{Int_Eq_Center}), we estimate
\begin{equation}\label{u_Cen_growth}
\begin{aligned}
\|u(t)\|_{H^{s_0}_x}
&\leq C_{s_0}\exp\big(C_{s_0}\delta_0|t|\big)\|u(0)\|_{H^{s_0}_x}
+C_{s_0}\int_{[0,t]}\exp\big(C_{s_0}\delta_0|t-\tau|\big)
\|\R(u(\tau))\|_{H^{s_0}_x}\dtau+\delta_0^2\\
&\overset{\ref{Ext2}}{\leq}
C_{s_0}\exp\big(C_{s_0}\delta_0|t|\big)(\|u(0)\|_{H^{s_0}_x}+\delta_0^2)
\leq C_{s_0,\mu}e^{(\mu-\delta)|t|}(\|u(0)\|_{H^{s_0}_x}+\delta_0^2).
\end{aligned}
\end{equation}
This finishes the proof.
\end{proof}

\subsection{The Center Invariant Set}
To construct the center invariant set, we are led by (3) of Lemma \ref{Lem_Center} to the fixed-point type equation
\begin{equation}\label{Eq_Initial_Center'}
\begin{aligned}
g&=\mathcal{I}(g,f)\\
&:=\int_{-\infty}^0 e^{\tau L(D_x)}\proj_\st\R_\Ext\big(\xS(\tau;g+f)\big)\dtau
-\int_0^{+\infty} e^{-\tau L(D_x)}
\proj_\unst\R_\Ext\big(\xS(\tau;g+f)\big)\dtau.
\end{aligned}
\end{equation}
Here $f\in E_\disp\cap H^{s_1}$ is any given element, and 
$$
\xS(t;\cdot):H^{s_1}\mapsto C^0_{t;\loc}H^{s_1}_x\cap W^{1,\infty}_{t;\loc}H^{s_1-3/2}_x
$$
is the solution map of the evolution equation (\ref{EQ_Red_Ext}) (equivalently \eqref{EQ_Red_Ext_New}), while $g\in E_\st\oplus E_\unst$ is the unknown to be solved. By global well-posedness shown in Theorem \ref{GWP_EQ}, we know that $\xS(t;\cdot)$ is a continuous map in the weaker norm topology
$$
H^{s_1}\mapsto C^0_{t;\loc}H^{s_1-\delta}_x.
$$
Since $\R_\Ext: H^s_x\to H^{s+s_0-2}_x, s\geq s_0$ is smooth by \ref{Ext2}, we find that given any $f\in E_\disp$, the mapping $\mathcal{I}(g,f)$ in (\ref{Eq_Initial_Center'}) is a continuous map defined for $g\in E_\st\oplus E_\unst$, satisfying
\begin{equation}\label{I(g,f)}
\|\mathcal{I}(g,f)\|_{L^2}
\overset{\ref{Ext2}}{\leq}
C_{s_1}\int_{\xR}e^{-\tau L(D_x)}\delta_0^2\dtau
\leq C_{s_1}\delta_0^2.
\end{equation}
If we take $\delta_0\ll1$, we can make sure that
$$
g\mapsto \mathcal{I}(g,f)
$$
maps the closed ball $\{g\in E_\st\oplus E_\unst:\|g\|_{L^2}\leq \delta_0\}$ to itself continuously. 

Recall that in the periodic boundary case, the hyperbolic subspace $E_\st\oplus E_\unst$ is finite-dimensional. By the Brouwer fixed point theorem, the equation (\ref{Eq_Initial_Center'}) then has a non-void close set of solutions. The estimate (\ref{I(g,f)}) shows that this solution set must be contained in the even smaller ball
$$
\{g\in E_\st\oplus E_\unst:\|g\|_{L^2}\leq C_{s_1}\delta_0^2\}.
$$
Using Lemma~\ref{Lem_Center}, we conclude from these arguments the following:
\begin{proposition}\label{Brouwer_Set}
Fix $s_1\geq s_0+3/2$ and $\delta_0$ small enough, depending only on $s_1$. Given any $f\in E_\disp\cap H^{s_1}$, there is a non-void close subset 
$$
\mathfrak{F}(f)\subset\{g\in E_\st\oplus E_\unst:\|g\|_{L^2}\leq C_{s_1}\delta_0^2\},
$$
such that if $g\in \mathfrak{F}(f)$, then $u(t)=\xS(t;g+f)\in C^0_{t;\loc}H^{s_1}_x$ is a solution of the integral equation (\ref{Int_Eq_Center}) with $\proj_\disp u(0)=f$.
\end{proposition}

With the \emph{close set-valued mapping} $\mathfrak{F}$ given by Proposition~\ref{Brouwer_Set}, let us then consider the set 
\begin{equation}\label{M_Center}
M_{\cen}:=\big\{(f,g):f\in E_\disp\cap H^{s_1},\,g\in\mathfrak{F}(f)\big\}.
\end{equation}
We claim that $M_{\cen}$ is an invariant set for the evolution equation (\ref{EQ_Red_Ext_New}). Indeed, we have just shown that the solution $u(t)$ with $u(0)\in M_{\cen}$ must satisfy the integral equation (\ref{Int_Eq_Center}). For any time $t_1$, the solution $u_1(t)$ of (\ref{EQ_Red_Ext_New}) with initial value $u_1(0)=u(t_1)$ is still of class $C^0_{t;\loc}H^{s_1}_x$ by the global well-posedness result. Obviously $u_1(t)$ still satisfies the mild growth requirement $\|u_1(t)\|_{H^{s_0}_x}=o(e^{\mu|t|})$ since it is simply a time shift of $u(t)$. By Lemma~\ref{Lem_Center}, $u_1(0)=u(t_1)$ must still satisfy (\ref{Eq_Initial_Center}), hence it must still be in the set $M_{\cen}$. 

The set $M_{\cen}$ given by (\ref{M_Center}) will then be our \emph{center invariant set}. Restricted back to the ball $\|u\|_{H^{s_0}}\leq 2\delta_0$, we see that $M_{\cen}$ is an invariant set for the water jet system (\ref{EQ_DG}) with respect to this ball. Unlike the construction of $M_\st$ and $M_\unst$, we are \emph{unable} to prove that $M_{\cen}$ is a submanifold (namely a graph), since there is \emph{no} guarantee that the integral equation (\ref{Int_Eq_Center}) has exactly one solution for a given $f=\proj_\disp u(0)$. The major reason is the quasilinear nature of the system -- if $u$ does not have any decay in time, then the differential with respect to $u$ of the right-hand side of (\ref{Int_Eq_Center}) shall result in extra growth in $t$ due to terms of the form $\der_u\bm{F}\cdot\del u$, by the explicit formula (\ref{DuF(u)}). This does not seem to be a technical issue; we identify it as an essential feature for problems involving propagators whose ``phase function" does not admit good estimate. A similar difficulty for semilinear Schr\"{o}dinger (NLS) equation has already been recognized by Krieger and Schlag, and they were forced to resort on the non-constructive Schauder-Tychnoff fixed point theorem to construct ``non-generic blow up set" for the NLS\footnote{The authors would like to thank Zexing Li for bringing this to their attention.}. See \emph{Stage D: Locating a fixed point} in \cite{KS2009}.

\subsection{Lifespan Estimate}
We prove a trivial lifespan estimate based on energy inequalities. A consequence of this is that $M_\cen$ is ``tangent to $E_\disp\cap H^{s_1}$ at 0'', completing the proof of Theorem~\ref{Main2}.

\begin{proposition}\label{prop-center:TrivLSp}
Fix $s_1\geq s_0+3/2$. Let $u \in C^0_{t;\loc}H^{s_1}_x$ be a solution to the integral equation~\eqref{Int_Eq_Center}. If $\delta_0>0$ is chosen small enough, then there exist constants $c_1',A>0$ depending only on $s_0$, such that, when the initial data $u(0)$ satisfies
    $$
        \| u(0) \|_{H^{s_0}} = \varepsilon \leq\delta_0 ,
    $$
    there holds
    \begin{equation*}\label{eq-center:TrivLSp}
        \sup_{|t| \leq c_1'\varepsilon^{-1}} \| u(t) \|_{H^{s_0}_x} \leq A\varepsilon.
    \end{equation*}
\end{proposition}
\begin{proof}
Throughout the proof, let us fix $\bar{\varepsilon} \in [\delta_0/2,\delta_0]$. Suppose for the moment $\|u(0)\|_{H^{s_0}}\leq\bar{\varepsilon}$. Using the intermediate inequality in (\ref{u_Cen_growth}), we find that 
$$
\|u(t)\|_{H^{s_0}_x}\leq
C_{s_0}\exp\big(C_{s_0}\delta_0|t|\big)(\|u(0)\|_{H^{s_0}_x}+\delta_0^2)
$$
In particular, when $|t| \leq c_1'\bar{\varepsilon}^{-1} \leq 2c_1'\delta_0^{-1}$, the right-hand side can be bounded by $A\bar{\varepsilon}$ if $\delta_0$ is chosen small enough and $A$ is chosen large enough.

Now, let us pass to the general case. We define
$$
\begin{aligned}
T_n(u) &:= \inf\Big\{ T > 0:u\text{ is a solution}, \sup_{|t|\leq T}\|u(t)\|_{H^{s_0}_x} \leq 2^{-n}A\bar{\varepsilon} \Big\}, \\
\quad T_n &:= \inf\Big\{T_n(u):u\text{ is a solution, } \|u(0)\|_{H^{s_0}}\leq 2^{-n}\bar{\varepsilon}\Big\}.
\end{aligned}
$$
The goal is to justify that for all $n\in\xN$ and any initial value $\|u(0)\|_{H^{s_0}}\leq 2^{-n}\bar{\varepsilon}$, there holds
\begin{equation}\label{eq-center:TrivLSp-iter}
T_n(u) \ge c_1' 2^{n}\bar{\varepsilon}^{-1}.
\end{equation}
Here $c_1',A>0$ are constants depending only on $s_0$. The argument in the previous paragraph already proves the initial case $n=0$. We are going to prove it for general $n$ by induction. Assume that \eqref{eq-center:TrivLSp-iter} holds for some $n=k$ with $k\in\xN$. 

Consider any solution $u$ with $\|u(0)\|_{H^{s_0}}\leq 2^{-k-1}\bar{\varepsilon}$. If $T_{k+1}(u) = + \infty$, then~\eqref{eq-center:TrivLSp-iter} is obvious. Otherwise, at least one of the following holds:
$$
\lim_{t\to T_{k+1}(u)} \|u(t)\|_{H^{s_0}_x} = 2^{-k}\bar{\varepsilon}, \quad \lim_{t\to -T_{k+1}(u)} \|u(t)\|_{H^{s_0}_x} = 2^{-k}\bar{\varepsilon}.
$$
Without loss of generality, we assume that the first one holds. Then we estimate the integral equation~\eqref{Int_Eq_Center} as $t \to T_{k+1}(u)$. According to the definition of $T_{k+1}(u)$, we have
$$
\sup_{t\in [0,T_{k+1}(u)]} \| u(t) \|_{H^{s_0}_x} \leq 2^{-k-1}A\bar{\varepsilon}.
$$
Based on this, we obtain from the equation itself
$$
\|\partial_tu(t)\|_{H^{s_0-3/2}_x}
\leq C_{s_0}2^{-k-1}A\bar{\varepsilon},
$$
so by \ref{F2}, 
$$
\|\bm{F}(u;t,\tau);\mathcal{L}(H^s_x)\|
\leq C_s\exp\left(C_{s}2^{-k-1}A\bar{\varepsilon}|t-\tau|\right),\quad
t,\tau\in[0,T_{k+1}(u)].
$$
Similar to (\ref{u_Cen_growth}), we thus find that for all $t\in [0,T_{k+1}(u)]$,
\begin{equation}\label{Induction(k+1)}
\begin{aligned}
\| u(t) \|_{H^{s_0}_x} 
&\leq C_{s_0}\exp\left(C_{s_0}2^{-k-1}A\bar{\varepsilon}t\right)\|u(0)\|_{H^{s_0}_x}\\
&\quad+C_{s_0}\int_{[0,t]}\exp\left(C_{s_0}2^{-k-1}A\bar{\varepsilon}|t-\tau|\right)\|\R(u(\tau))\|_{H^{s_0}_x}\dtau\\
&\quad+\int_{-\infty}^{+\infty} e^{-|t-\tau|\mu} \|\R_\Ext(u(\tau))\|_{H^{s_0}_x} \dtau.
\end{aligned}
\end{equation}
Using \ref{Ext2}, the sum of the first two terms in the right-hand side of (\ref{Induction(k+1)}) has upper bound
$$
C_{s_0}\exp(C_{s_0} 2^{-k-1}A\bar{\varepsilon} t)\big( 2^{-k-1} \bar{\varepsilon} + 2^{-2k-2}A^2\bar{\varepsilon}^2 t\big).
$$
For the last integral in the right-hand side of (\ref{Induction(k+1)}), we split it as follows:
$$
\begin{aligned}
\int_{-\infty}^{+\infty} e^{-|t-\tau|\mu} \|\R_\Ext(u(\tau))\|_{H^{s_0}_x} \dtau 
&\leq \int_{|\tau| \ge T_{k+1}(u)+T_{k}} + \int_{|\tau| \leq T_{k+1}(u)+T_{k}} \\
&\leq e^{-T_k\mu} C_{s_0}\delta_0^2 + C_{s_0} 2^{-2k}A^2\bar{\varepsilon}^2.
\end{aligned}
$$
Note that the second inequality follows from
$$
\sup_{t\in [0,T_{k+1}(u)+T_k]} \| u(t) \|_{H^{s_0}_x} \leq 2^{-k}A\bar{\varepsilon},
$$
due to the definition of $T_k$. 
    
To sum up, we have shown that, for all time $t\in [0,T_{k+1}(u)]$,
\begin{equation}\label{Induction(k+1)_1}
\begin{aligned}
\frac{\| u(t) \|_{H^{s_0}_x}}{2^{-k-1}A\bar{\varepsilon}} 
&\leq C_{s_0}\exp(C_{s_0} 2^{-k-1}A\bar{\varepsilon} t) \big( A^{-1} + 2^{-k-1}A\bar{\varepsilon} t \big) \\
&\quad + e^{-T_k\mu} C_{s_0}\delta_0^2(2^{-k-1}A\bar{\varepsilon})^{-1} + C_{s_0} 2^{-k+1}A\bar{\varepsilon} \\
&\overset{\eqref{eq-center:TrivLSp-iter}}{\leq} 
C_{s_0}\exp(C_{s_0} 2^{-k-1}A\bar{\varepsilon} t) \big( A^{-1} 
+ 2^{-k-1}A\bar{\varepsilon} t \big) \\
&\quad + \exp(-C_{s_0}c_1'\mu2^k\bar{\varepsilon}^{-1})
C_{s_0}A^{-1}2^{k+2}\delta_0
+ 2C_{s_0} A\delta_0.
\end{aligned}
\end{equation}
The right-hand side satisfies
$$
\begin{aligned}
\exp(-C_{s_0}c_1'\mu2^k\bar{\varepsilon}^{-1})
C_{s_0}A^{-1}2^{k+2}\delta_0 
\leq C_{s_0}' \frac{2^{k+2}\delta_0}{c_1'A\mu2^k\bar{\varepsilon}^{-1}} 
\leq \frac{4C_{s_0}'}{c_1'A\mu} \delta_0^2.
\end{aligned}
$$
Taking the limit $t\to T_{k+1}(u)$ in \eqref{Induction(k+1)_1}, we conclude
$$
1 \leq C_{s_0}\exp(C_{s_0} 2^{-k-1}A\bar{\varepsilon}T_{k+1}(u)) \big( A^{-1} 
+ 2^{-k-1}A\bar{\varepsilon}T_{k+1}(u) \big) 
+ \frac{4C_{s_0}'}{c_1'A\mu} \delta_0^2
+ 2C_{s_0} A\delta_0.
$$
We choose $A\gg 1$, then $\delta_0\ll1$ and $c_1'A \ll 1$, so that 
$$
\frac{4C_{s_0}'}{c_1'\mu} A^{-1}\delta_0^2
+ 2C_{s_0} A\delta_0<\frac{1}{4}.
$$
This should imply the desired inequality \eqref{eq-center:TrivLSp-iter} for $n=k+1$. 
\end{proof}

Thanks to Proposition~\ref{prop-center:TrivLSp}, we can prove that the center invariant set $M_\cen$ constructed in~\eqref{M_Center} is ``tangent'' to the dispersive subspace $E_\disp$ in the linear analysis~\ref{Lin3}. To be precise, for a small enough solution on the center invariant set, its projection along hyperbolic directions is much smaller than that along $E_\cen$. Restricted to the small ball $\|u\|_{H^{s_0}}\leq 2\delta_0$, recalling that the relation between $(\eta,\psi)$ and $u$ is diffeomorphic within that neighborhood (Lemma \ref{lem:ChgtUnk:psi-w} and \ref{DG_Transform}), while the system satisfied by $u$ is equivalent to the water jet system within that neighborhood, this finishes the proof of Theorem \ref{Main2}:

\begin{corollary}\label{cor-center:Tgt}
Fix $s_1\geq s_0+3/2$ and $\delta_0>0$ small enough as above. Let $z\in M_\cen$, and suppose $\| z \|_{H^{s_0}} = \varepsilon \leq \delta_0$. Then the stable and unstable components verify $\| \proj_\st z\|_{L^2} + \| \proj_\unst z\|_{L^2} \lesssim_{s_0} \varepsilon^2$. 
\end{corollary}
\begin{proof}
We denote by $u(t)$ the unique solution to~\eqref{EQ_Red_Ext} with initial data $u(0)=z$. From Proposition~\ref{prop-center:TrivLSp} above, the $H^{s_0}_x$ norm of $u(t)$ is smaller than $A\varepsilon$ within time $|t| \leq c_1'\varepsilon^{-1}$. Recall that its stable and unstable components at $t=0$ are given by the formula~\eqref{Eq_Initial_Center}. We split the integral as
\begin{equation*}
\begin{aligned}
(\proj_{\st}+\proj_{\unst})u(0) 
&= \int_{-c_1'\varepsilon^{-1}}^0 e^{-\tau L(D_x)}\proj_\st\R_\Ext(u(\tau))\dtau -\int_0^{c_1'\varepsilon^{-1}} e^{-\tau L(D_x)} \proj_\unst\R_\Ext(u(\tau))\dtau \\
&\quad+ \int_{-\infty}^{-c_1'\varepsilon^{-1}} e^{-\tau L(D_x)}\proj_\st\R_\Ext(u(\tau))\dtau -\int_{c_1'\varepsilon^{-1}}^{+\infty} e^{-\tau L(D_x)} \proj_\unst\R_\Ext(u(\tau))\dtau
\end{aligned}
\end{equation*}
From estimates~\eqref{Size_RExt} of $\R_\Ext$, the first two integrals are bounded by $\varepsilon^2$, while the remaining two integrals are bounded by
$$
\exp(-c_1'\varepsilon^{-1}\mu) C_{s_0}\delta_0^2 \lesssim_{s_0} (c_1'\mu)^{-2} \delta_0^2 \varepsilon^2,
$$
which completes the proof.
\end{proof}

\appendix

\section{Paralinearization of Dirichlet-Neumann operator}\label{App:PLDtN}

The aim of this appendix is to prove Theorem~\ref{thm-PLDtN:Main}. As mentioned in Section~\ref{subsect:PLDtN}, this is actually a refined version of the known paralinearization formula \eqref{eq-PLDtN:Basic}, while we will check in each step that the remainder is regular in $(\eta,\psi)$ (recall the definition~\ref{def:Reg} of regularity) with quadratic decay as $(\eta,\psi)\to0$. Before starting the proof, we recall the strategy to show \eqref{eq-PLDtN:Basic}, which serves as the guideline for the proof of Theorem~\ref{thm-PLDtN:Main}.
\begin{enumerate}[label=\textbf{Step \arabic* }]
\item\label{step:Paralin1} \textbf{(Regularity of scalar potential)} From the definition of Dirichlet-Neumann operator, we first need to investigate the regularity of the scalar potential $\Psi$, which is a harmonic function in the fluid region $\Omega_t$ with $\Psi(x,\rho+\eta(x))=\psi(x)$. Based on cylindrical coordinate $(r,\theta,x)$ (where $r,\theta,x$ refer to radial, angular, and axial variable, respectively), we apply the change of variable $(r,x)=(y\eta(x),x)$ and the axisymmetric scalar potential can be regarded as a function on $\{(x,y): x\in\xT, 0<y<1\}$:
$$
\Phi(x,y) = \Psi(y\eta(x),x).
$$
Furthermore, in this coordinate, the scalar potential $\Phi$ is the unique solution to the following boundary problem:
\begin{equation}\label{eq-PLDtN:Ellip}
    \left\{\begin{aligned}
        & \mathscr{L}[\eta]\Phi := \big( \alpha\partial_y^2 + \beta\partial_x\partial_y + \partial_x^2 - \gamma \partial_y \big) \Phi = 0, &&0<y<1 \\
        & \Phi|_{y=1}=\psi, && \\
        & \partial_y\Phi|_{y=0} = 0, &&
    \end{aligned}\right.
\end{equation}
where the coefficients are smooth functions of $\eta$ and its derivatives $\partial_x\eta, \partial_x^2\eta$,
\begin{equation}\label{eq-PLDtN:EllipCoeff}
    \alpha = \frac{1+y^2|\partial_x\eta|^2}{(\rho+\eta)^2}, \quad \beta = -\frac{2y\partial_x\eta}{\rho+\eta}, \quad \gamma = \frac{1}{y(\rho+\eta)^2} - y(\rho+\eta)\partial_x\left(\frac{\partial_x\eta}{(\rho+\eta)^2}\right).
\end{equation}
This elliptic problem is solved in Section 2.2.3 of \cite{HK2023} through a similar argument to that in \cite{alvarez2008water}. The authors managed to show that, given $(\eta,\psi)\in H^{s+1}\times H^{s+1/2}$, the scalar potential $\Phi$ lies in $L^2_y(0.5,1;H^{s+1}) \cap C^0_y(0.5,1;H^{s+1/2})$ and its radial derivatives $\partial_y^k\Phi$ belongs to $L^2_y(0.5,1;H^{s-k+1})$ with $k\in\{0,1,2\}$. The differentiability in $(\eta,\psi)$ and tame estimates can be seen in the in Section 2.2.3 of \cite{HK2023}, or from Lemma~\ref{lem:SolReg} below.

\item\label{step:Paralin2} \textbf{(Paralinearization)} With the regularity of $\Phi$, we are able to justify the paralinearization of \eqref{eq-PLDtN:Ellip}. More precisely, we will introduce \textit{Alinhac's good unknown}
$$
W := \Phi - T_{(\rho+\eta)^{-1}\partial_y\Phi}(y\eta),
$$
and check that the action of $\mathscr{P}[\eta]$ ( the paralinear version of $\mathscr{L}[\eta]$ ) on $W$,
$$
\mathscr{P}[\eta]W := \big( T_\alpha\partial_y^2 + T_{\beta i\xi}\partial_y - T_{\xi^2} - T_\gamma \partial_y \big) W
$$
has regularity $L^2_y(0.5,1;H^{2s-N})$ for some real number $N$ independent of $s$. Note that since we only apply paradifferential calculus in this step, all the estimates can be shown to be tame. Moreover, the coefficients $\alpha,\beta,\gamma$ and the scalar potential $\Phi$ are all smooth in $\eta$, ensuring that the remainders are also $C^\infty$ in $\eta$.

\item\label{step:Paralin3} \textbf{(Factorization)} Thanks to the elliptic nature of $\mathscr{P}[\eta]$, we can expect the following factorization 
$$
\mathscr{P}[\eta] = T_{\alpha}(\partial_y + T_a)(\partial_y - T_A) + \text{smooth remainders},
$$
where the symbols $a,A$ are elliptic and of order $1$. In fact, we will assume that $a,A$ are composed by homogeneous symbols of order $0,-1,-2$, etc, based on which the symbolic calculus allows us to uniquely determine the desired symbols $a,A$. Note that, the symbolic calculus does not change the smooth dependence on $\eta$, making expect that the remainders obtained in this step is still $C^\infty$ in $\eta$.

\item\label{step:Paralin4} \textbf{(Boundary regularity)} From the previous step, the difference between radial and axial derivatives of $W$, namely
$$
(\partial_y - T_A) W
$$
solves a parabolic equation. Thus the parabolic regularity yields that the trace of this difference at the free boundary has regularity $H^{2s-N}$ for some fixed real number $N$. Note that during the resolution of this parabolic equation, the smooth dependence in $\eta$ can be maintained (see Lemma~\ref{lem:SolReg} below).

\item\label{step:Paralin5} \textbf{(Conclusion)} In terms of $\Phi$ (the scalar potential $\Psi$ in cylindrical coordinate), the Dirichlet-Neumann operator can be written as
\begin{equation}\label{eq:Formu-DtN}
	G[\eta]\psi = \left.\left( \frac{1+|\partial_x\eta|^2}{\rho+\eta} \partial_y\Phi - \partial_x\eta\partial_x\Phi \right)\right|_{y=1}.
\end{equation}
According to the previous step, we can replace the radial derivative $\partial_y\Phi$ as axial ones via good unknown $W$ and the desired formula \eqref{eq-PLDtN:Basic} follows with extra information - the remainder is linear in $\psi$ (thus in $w:=W|_{y=1}$), smooth in $\eta$, and satisfies some tame estimates.
\end{enumerate}

\begin{notation}
    Let $p\in[1,\infty]$ and $s\in\xR$. For functions defined on $\{(x,y):x\in\xT,0<y<1\}$, we use the following simplified notation of mixed-type spaces:
    $$
    L^p_yH^s_x = L^p_y(0.5,1;H^s).
    $$
\end{notation}


\subsection{Preliminaries}

In this part, we collect some technical lemmas as well as some known results about the regularity of the scalar potential.

\begin{lemma}\label{lem:SolReg}
    Consider real numbers $s_0, s\in\xR$, and decreasing spaces $\{X^s\}_{s\in\xR}$, $\{Y^s\}_{s\in\xR}$, $\{Z^s\}_{s\in\xR}$. We assume that:
    \begin{enumerate}
        \item $f=f(\eta,\psi)$ belongs to $\mathcal{T}_r^\infty(s_0;X^\bullet)$ and is linear in $\psi$; $g=g(\eta)$ belongs to $\mathcal{T}^\infty(s_0;H^{\bullet+1};Y^\bullet)$.

        \item $b=b(g,w)$ is a bounded bilinear form from $Y^{s_0} \times Z^{s_0}$ to $X^{s_0}$, satisfying the regularity condition
        $$
            \Ta^\infty(s_0;Y^\bullet \times Z^\bullet, X^\bullet).
        $$

        \item For all $s \ge s_0$, $\mathscr{R}[\eta]$ is a linear form in the class $\mathcal{L}(X^s;Z^s)$ depending implicitly on $\eta\in H^{s+1}$. For all $h\in X^{s}$, there holds
        $$
        \| \mathscr{R}[\eta]h \|_{Z^s} \leq K\big(\|\eta\|_{H^{s_0+1}}\big) \Big( \|h\|_{X^s} + \|\eta\|_{H^{s+1}} \|h\|_{X^{s_0}} \Big).
        $$
    \end{enumerate}
    Under these assumptions, if the following identity holds for all $(\eta,\bm{\delta}\eta,\psi)\in H^{s_0+1}\times H^{s_0+1} \times H^{s_0+1/2}$,
    \begin{equation}\label{eq:SolReg-Cond}
        \begin{aligned}
            \mathscr{R}[\eta+\bm{\delta}\eta]&f(\eta+\bm{\delta}\eta,\psi) - \mathscr{R}[\eta]f(\eta,\psi)  \\
            &= \mathscr{R}[\eta]\Big( f(\eta+\bm{\delta}\eta,\psi) - f(\eta,\psi) - b\big( g(\eta+\bm{\delta}\eta)-g(\eta), \mathscr{R}[\eta]f(\eta+\bm{\delta}\eta,\psi) \big) \Big),
        \end{aligned}
    \end{equation}
    then the mapping $v = v(\eta,\psi)$ of $(\eta,\psi)$ defined by
    $$
        (\eta,\psi) := \mathscr{R}[\eta]f(\eta,\psi)
    $$
    belongs to the class $\mathcal{T}_r^\infty(s_0;Z^\bullet)$. The derivatives of $v$ in $\eta$ can be calculated via the following iteration formula,
    \begin{equation}\label{eq:SolReg-DerFormu}
    	\der^n_\eta v(\eta,\psi) = \mathscr{R}[\eta]\left( \der^n_\eta f(\eta,\psi) - \sum_{k=1}^{n} \binom{n}{k} b( \der^k_\eta g(\eta), \der^{n-k}_\eta v(\eta,\psi) ) \right).
    \end{equation}
    Recall that the class of regular mappings are defined in Definition~\ref{def:Reg} and the class $\Ta_r^\infty$ is defined in Notation~\ref{note:Reg}.
\end{lemma}

This lemma aims to explain that the solution to some evolution equation (resp. elliptic equations) depends smoothly on initial data (resp. boundary value) and the source term. Meanwhile, the regularity (of data and source term) can be preserved by the solution. In application, $\mathscr{R}$ will be chosen as the resolution operator, while $g$ and $f$ serve as the coefficients and source term, respectively. The bilinear form $b$ is usually the action of (para)differential operators on the solution. This lemma will be used to prove the elliptic regularity in~\ref{step:Paralin1} and the paradifferential parabolic equation in~\ref{step:Paralin4}.

\begin{proof}
    To begin with, we notice that $v$ is linear (thus smooth) in $\psi$. The tame estimates can be deduced directly from the formula \eqref{eq:SolReg-DerFormu} as well as the tames estimates for $f,g,b$, and $\mathscr{R}[\eta]$. Therefore, we only focus on the proof of smoothness in $\eta$, which can be deduced from the difference formula \eqref{eq:SolReg-DerFormu} through an iteration argument. More precisely, we will prove that, for any $m\in\xN$, if $v$ is $m$ times differentiable in any directions with \eqref{eq:SolReg-DerFormu} holding for $n=0,1,...m$, then $v$ is $(m+1)$ times differentiable in any directions with \eqref{eq:SolReg-DerFormu} for $n=m+1$. 
	
	If $v$ is $m$-times differentiable (in all directions) with \eqref{eq:SolReg-DerFormu}, we can apply $m$ times directional derivative in $\bm{\delta}\eta$ to the difference formula \eqref{eq:SolReg-Cond} and obtain that
	\begin{align*}
		\der^m_\eta v(\eta+\bm{\delta}\eta,\psi) =& \mathscr{R}[\eta]\left( \der_\eta^mf(\eta+\bm{\delta}\eta,\psi) - \sum_{k=1}^m b\big( \der^k_\eta g(\eta+\bm{\delta}\eta), \mathscr{R}[\eta]\der^{m-k}_{\eta}f(\eta+\bm{\delta}\eta,\psi) \big) \right) \\
		&- \mathscr{R}[\eta] b\big( g(\eta+\bm{\delta}\eta)-g(\eta), \mathscr{R}[\eta]\der^{m}_{\eta}f(\eta+\bm{\delta}\eta,\psi) \big),
	\end{align*}
	whose difference with $\der^m_\eta v(\eta,\psi)$ thus reads
	\begin{align*}
		&\der^m_\eta v(\eta+\bm{\delta}\eta,\psi) - \der^m_\eta v(\eta,\psi) \\
		=& \mathscr{R}[\eta]\left( \der_\eta^mf(\eta+\bm{\delta}\eta,\psi) - \der_\eta^mf(\eta,\psi) - \sum_{k=0}^m b\big( \der^k_\eta g(\eta+\bm{\delta}\eta) - g(\eta), \mathscr{R}[\eta]\der^{m-k}_{\eta}f(\eta+\bm{\delta}\eta,\psi) \big) \right) \\
		& - \mathscr{R}[\eta]\left( \sum_{k=1}^m b\Big( \der^k_\eta g(\eta+\bm{\delta}\eta), \mathscr{R}[\eta]\big(\der^{m-k}_{\eta}f(\eta+\bm{\delta}\eta,\psi) - \der^{m-k}_{\eta}f(\eta,\psi)\big) \Big) \right).
	\end{align*}
	For one thing this formula implies that the $m$-order directional derivatives of $v$ is continuous in $\eta$, thanks to the regularity of $f,g$ and the boundedness of $b,\mathscr{R}[\eta]$. For another thing, we can replace $\bm{\delta}\eta$ by $\epsilon\bm{\delta}\eta$ and take the limit $\epsilon\to 0$. Again through the regularity in $(\eta,\psi)$ of $f,g$ with the boundedness of $b,\mathscr{R}[\eta]$, we can conclude that, for all $s \ge s_0$, $\der_\eta^m v$ admits directional derivative in $\eta$ along all directions $\bm{\delta}\eta\in H^{s+1}$ with 
	\begin{align*}
		\der^{m+1}_\eta v =& \mathscr{R}[\eta]\left( \der_\eta^{m+1}f(\eta,\psi) - \sum_{k=0}^m b\big( \der^{k+1}_\eta g(\eta), \mathscr{R}[\eta]\der^{m-k}_{\eta}f(\eta,\psi) \big) \right) \\
		&\hspace{8em} - \mathscr{R}[\eta]\left( \sum_{k=1}^m b\big( \der^k_\eta g(\eta), \mathscr{R}[\eta]\der^{m+1-k}_{\eta}f(\eta+\bm{\delta}\eta,\psi) \big) \right).
	\end{align*}
	By applying the formula \eqref{eq:SolReg-DerFormu} for $n=0,1,...m$, we can conclude \eqref{eq:SolReg-DerFormu} for $n=m+1$.
\end{proof}

Now, let us recall the elliptic regularity proved in Section 2.2 of \cite{HK2023}.

\begin{proposition}[Corollary 2.12 and Lemma 2.15 of \cite{HK2023}]\label{prop-PLDtN:EllipRegOri}
    Consider a real number $s>2$. Let $(\eta,\psi)\in H^{s+1} \times H^{s+1/2}$ satisfy $\eta+\rho\ge c$ for some constant $c>0$. Then the elliptic problem \eqref{eq-PLDtN:Ellip} admits a unique solution $\Phi$. Moreover, there exists a positive increasing function $K\colon\xR_+\to\xR_+$ depending on $s,c$ such that, for all $0\leq \sigma \leq s$, there holds
    \begin{equation*}
        \int_0^1 \left( y\|\Phi\|_{H^{\sigma+1}_x}^2 + y\|\partial_y\Phi\|_{H^{\sigma}_x}^2 + y^3\|\partial_{y}^2\Phi\|_{H^{\sigma-1}_x}^2 \right) dy \leq K\big( \|\eta\|_{H^{s+1}} \big) \|\psi\|_{H^{\sigma+1/2}}^2.
    \end{equation*}
\end{proposition}

This result is no more than the resolution of the elliptic problem~\eqref{eq-PLDtN:Ellip}. According to Lemma~\ref{lem:SolReg}, this procedure preserves the regularity in the sense of Definition~\ref{def:Reg}. Alternatively, such result can also be obtained through iteration and the paralinearization in this appendix. We refer to Section 3.3.1 of~\cite{ABZ2014} for this type of argument. To sum up, we have

\begin{proposition}\label{prop-PLDtN:EllipReg}
    Consider real numbers $s_0>2$ and assume that $(\eta,\psi)\in H^{s_0+1} \times H^{s_0+1/2}$ satisfies $\eta+\rho\ge c$ for some constant $c>0$. Then the unique solution $\Phi$ to the elliptic problem \eqref{eq-PLDtN:Ellip}, together with its derivatives $\nabla_{x,y}\Phi, \nabla_{x,y}^2\Phi$, is regular in $(\eta,\psi)$. More precisely,
    $$
        \nabla_{x,y}^k\Phi \in \Ta^\infty_r\big(s_0;L^2_yH^{\bullet-k+1}_x\big), \quad k=0,1,2.
    $$
    In particular, we also have
    $$
        \nabla_{x,y}^k\Phi \in \mathcal{T}^\infty_r\big(s_0;L^\infty_yH^{\bullet-k+1/2}_x\big), \quad k=0,1.
    $$
    Recall that the class $\mathcal{T}^\infty_r$ is defined in Definition~\ref{def:Reg} (see also Notation~\ref{note:Reg}).
\end{proposition}

\subsection{An Equivalent Form}

Before starting the paralinearization, we first clarify that the quadratic decay $\R_\DN^{[\leq 1]} =0$ in Theorem~\ref{thm-PLDtN:Main} can be automatically obtained by the classical paralinearization~\eqref{eq-PLDtN:Basic} (see also~\eqref{eq-PLDtN-Pf:Goal} below). More precisely, we can reduce ourselves to the following result.

\begin{proposition}\label{prop-PLDtN-Pf:Main}
    Under the conditions of Theorem~\ref{thm-PLDtN:Main}, we have
    \begin{equation}\label{eq-PLDtN-Pf:Goal}
        G[\eta]\psi = T_{\tilde{\lambda}}w - \partial_x(T_V \eta) - T_{(\rho+\eta)^{-1}B} \eta + \tilde{\R}_\DN(\eta,\psi), 
    \end{equation}
    where $w = \psi - T_B\eta$ is the good unknown and the symbol $\tilde{\lambda}$ reads
    \begin{equation}\label{eq-PLDtN-Pf:PrSym}
        \begin{aligned}
            \tilde{\lambda} =& |\xi| - \frac{1}{2(\rho+\eta)} - \frac{\partial_x\eta }{2(\rho+\eta)}i\sgn(\xi) \\
            &\hspace{6em}+ \sum_{1\leq j <s_0-3/2}  \left( \tilde{f}_j(\eta,\partial_x\eta,\cdots,\partial_x^{j+2}\eta)\xi|\xi|^{-j-1} + \tilde{g}_j(\eta,\partial_x\eta,\cdots,\partial_x^{j+2}\eta)|\xi|^{-j} \right),
        \end{aligned}
    \end{equation}
    with $\tilde{f}_j,\tilde{g}_j\in C^\infty$. The remainder $\tilde{\R}_\DN(\eta,\psi)$ is linear in $\psi$ and belongs to $\mathcal{T}^\infty_r(s_0;H^{\bullet+s_0-1})$ (see Definition~\ref{def:Reg} and Notation~\ref{note:Reg}). 
\end{proposition}

\begin{proof}[Proof of Theorem~\ref{thm-PLDtN:Main}]
    By denoting $\tilde{\lambda}_0:= \tilde{\lambda}|_{\eta=0}$, we can rewrite formula \eqref{eq-PLDtN-Pf:Goal} as
    $$
    G[\eta]\psi = ( G[0] + T_{\tilde{\lambda}-\tilde{\lambda}_0} )w - \partial_x(T_V\eta) - T_{(\rho+\eta)^{-1}B}\eta + (G[0]-T_{\tilde{\lambda}_0})T_B\eta + \tilde{\R}_\DN(\eta,\psi) - \tilde{\R}_\DN(0,\psi).
    $$
    A simple calculus together with the fact
    $$
    G[0] = \frac{I_1(\rho|D_x|)}{I_0(\rho|D_x|)}|D_x|
    $$
    guarantees that $T_{\lambda} = G[0] + T_{\tilde{\lambda}-\tilde{\lambda}_0}$, where $\lambda$ is given by \eqref{eq-PLDtN:Symb} with $f = \tilde{f}-\tilde{f}(0)$ and $g = \tilde{g}-\tilde{g}(0)$. Note that in the subprincipal term
    $$
        - \frac{1}{2(\rho+\eta)} - \frac{\partial_x\eta }{2(\rho+\eta)}i\sgn(\xi),
    $$
    it is harmless to replace $\sgn(\xi)$ by $\partial_\xi\lambda^{(1)}$ since their difference has any polynomial decay in $\xi$ and thus contributes a smooth enough remainder, vanishing as $\eta=0$. Clearly, the remainder $\R_\DN(\eta,w)$ in~\eqref{eq-PLDtN:Main} is given by
    $$
        \R_\DN(\eta,\psi) = (G[0]-T_{\tilde{\lambda}_0})T_B\eta + \tilde{\R}_\DN(\eta,\psi) - \tilde{\R}_\DN(0,\psi),
    $$
    which clearly satisfies the quadratic decay $\R_\DN^{[\leq 1]}=0$ since it is linear in $\psi$.
    
    Now, we check $\R_\DN(\eta,\psi)$ is regular with $\R_\DN\in\Ta^\infty_r(s_0;H^{\bullet+s_0-1})$. The regularity of $\tilde{\R}_\DN(\eta,\psi) - \tilde{\R}_\DN(0,\psi)$ follows from Proposition~\ref{prop-PLDtN-Pf:Main}. As for the remaining part, we observe that, by taking $\eta=0$, the formula \eqref{eq-PLDtN-Pf:Goal} implies
    $$
        (G[0] - T_{\tilde{\lambda}_0})\psi = \tilde{\R}_\DN(0,\psi) \in H^{s+s_0-1}, \quad \forall \psi\in H^s.
    $$
    Thus, $G[0] - T_{\tilde{\lambda}_0}$ is no more than a Fourier multiplier bounded from $H^{s+1/2}$ to $H^{s+s_0-1}$. The boundedness of paraproduct (see Proposition~\ref{prop:PDReg}) and the estimate for $B$ (see Proposition~\ref{prop-PLDtN:BdDtN}) imply the desired regularity:
    \begin{equation*}
        \| (G[0]-T_{\tilde{\lambda}_0})T_B\eta \|_{H^{s+s_0-1}} \lesssim \| T_B\eta \|_{H^{s+1/2}} \lesssim |B|_{L^\infty} \|\eta\|_{H^{s+1/2}} \lesssim \|B\|_{H^{s_0-1}} \|\eta\|_{H^{s+1}} <+\infty.
    \end{equation*}
    The smoothness and tame estimates for its derivatives follow from those of $\tilde{\R}_\DN$ (see Proposition~\ref{prop-PLDtN-Pf:Main}) and $B$ (see Proposition~\ref{prop-PLDtN:BdDtN}), since these properties can be preserved by paraproduct (see Proposition~\ref{prop:PDReg}).
\end{proof}

In what follows, we will follow the guideline \ref{step:Paralin2}- \ref{step:Paralin5} to complete the proof of Proposition~\ref{prop-PLDtN-Pf:Main}.

\subsection{Paralinearization of the Elliptic Operator}

Recall that the scalar potential $\Phi$ verifies the following elliptic equation \eqref{eq-PLDtN:Ellip}
$$
\mathscr{L}[\eta]\Phi := \big( \alpha\partial_y^2 + \beta\partial_x\partial_y + \partial_x^2 - \gamma \partial_y \big) \Phi = 0
$$
with coefficients 
$$
\alpha = \frac{1+y^2|\partial_x\eta|^2}{(\rho+\eta)^2}, \quad \beta = -\frac{2y\partial_x\eta}{\rho+\eta}, \quad \gamma = -\frac{1}{y(\rho+\eta)^2} + y(\rho+\eta)\partial_x\left(\frac{\partial_x\eta}{(\rho+\eta)^2}\right)
$$
and that the paralinearization of $\mathscr{L}[\eta]$ reads
$$
\mathscr{P}[\eta] = \big( T_\alpha\partial_y^2 + T_{\beta i\xi}\partial_y - T_{\xi^2} - T_\gamma \partial_y \big).
$$
The purpose of this step is to show that the difference between $\mathscr{L}[\eta]$ and $\mathscr{P}[\eta]$ is negligible. More precisely, we have
\begin{proposition}\label{prop-PLDtN-Pf:S2}
    Under the conditions of Theorem~\ref{thm-PLDtN:Main}, we have
    \begin{equation}\label{eq-PLDtN-Pf:S1}
        \mathscr{P}[\eta] W = r_1(\eta,\psi),
    \end{equation}
    where the good unknown $W$ is defined by
    \begin{equation}\label{eq-PLDtN-Pf:GoodUnk}
        W:=  \Phi - T_{(\rho+\eta)^{-1}\partial_y\Phi} (y\eta)
    \end{equation}
    and the remainder $r_1(\eta,\psi)$ belongs to the class $\mathcal{T}^\infty_r(s_0;L^2_yH^{\bullet+s_0-3/2}_x)$ (see Definition~\ref{def:Reg} and Notation~\ref{note:Reg}).
\end{proposition}

As a preparation, we check the regularity of the coefficients $\alpha,\beta,\gamma$ and the good unknown $W$, which is due to Proposition~\ref{prop:CompReg} and Proposition~\ref{prop-PLDtN:EllipReg}, respectively.
\begin{lemma}\label{lem-PLDtN-Pf:Coeff}
    Consider a real numbers $s_0>2$. If $\rho+\eta \ge c$ for some constant $c>0$, then the coefficients $\alpha,\beta,\gamma$, together with their derivatives in $y\in[0.5,1]$ are regular in $\eta$. More precisely, we have, for all integers $k\in\xN$,
    $$
        \partial_y^k\alpha,\partial_y^k(\alpha^{-1}),\partial_y^k\beta\in \Ta^\infty(s_0;H^{\bullet+1},L^\infty_yH^{\bullet}_x), \quad \partial_y^k\gamma\in \Ta^\infty(s_0;H^{\bullet+1},L^\infty_yH^{\bullet-1}_x).
    $$
    Furthermore, we have the following paralinearization formulas
    \begin{align*}
        \alpha =& \rho^{-2} - \Op^\PM\left( 2(1+y^2|\partial_x\eta|^2)(\rho+\eta)^{-3} \right)\eta + \Op^\PM\left( 2y^2\partial_x\eta(\rho+\eta)^{-2} \right)\partial_x\eta + \tilde{r}_{\alpha}, \\
        \beta =& \Op^\PM\left( 2y\partial_x\eta(\rho+\eta)^{-2} \right)\eta - \Op^\PM\left( 2y\eta(\rho+\eta)^{-1} \right)\partial_x\eta + \tilde{r}_{\beta}, \\
        \gamma =& - y^{-1}\rho^{-2} + \Op^\PM\left( 2y^{-1}(\rho+\eta)^{-3} + 4y|\partial_x\eta|^2(\rho+\eta)^{-3} - y\partial_x^2\eta(\rho+\eta)^{-2} \right) \eta \\
        &- \Op^\PM\left( 4y\partial_x\eta(\rho+\eta)^{-2} \right) \partial_x\eta + \Op^\PM\left( y(\rho+\eta)^{-1} \right)\partial_x^2\eta + \tilde{r}_{\gamma},
    \end{align*}
    with
    $$
        \tilde{r}_\alpha, \tilde{r}_\beta \in \Ta^\infty(s_0;H^{\bullet+1},L^\infty_yH^{\bullet+s_0-1/2}_x), \quad \tilde{r}_\gamma \in \Ta^\infty(s_0;H^{\bullet+1},L^\infty_yH^{\bullet+s_0-1/2}_x).
    $$
\end{lemma}

\begin{lemma}\label{lem-PLDtN-Pf:GoodUnk}
    Under the condition of Proposition~\ref{prop-PLDtN:EllipReg}, the good unknown $W$ satisfies
    $$
        \nabla_{x,y}^kW \in \Ta^\infty_r(s_0;L^2_yH^{\bullet-k+1}_x),\quad k=0,1,2.
    $$
    Consequently, for $k=0,1$, $\nabla_{x,y}^kW$ also belongs to $\Ta^\infty_r(s_0;L^\infty_yH^{\bullet-k+1/2}_x)$.
\end{lemma}

For the sake of simplicity, during the proof of Proposition~\ref{prop-PLDtN-Pf:S2}, we use the notation:
$$
    w \sim v \Leftrightarrow w-v \in \Ta^\infty_r(s_0;L^2_yH^{\bullet+s_0-3/2}_x).
$$
Namely, our goal is to show that $\mathscr{P}[\eta] W \sim 0$.

We begin with the following identity, which is simply an application of paraproduct decomposition (see Definition~\ref{sec.2.3}):
\begin{align*}
    \mathscr{L}[\eta]\Phi =& \mathscr{P}[\eta]\Phi + T_{\partial_y^2\Phi} \alpha + T_{\partial_y\partial_x\Phi} \beta - T_{\partial_y\Phi}\gamma + r_{1,1}(\eta,\psi), \\
    r_{1,1}(\eta,\psi) = & R_\PM(\alpha,\partial_y^2\Phi) + R_\PM(\beta,\partial_y\partial_x\Phi) - R_\PM(\gamma,\partial_y\Phi) + (\partial_x^2 + T_{\xi^2}) \Phi.
\end{align*}
By combining Proposition~\ref{prop:PMReg} and~\ref{lem-PLDtN-Pf:Coeff}, together with Proposition~\ref{prop-PLDtN:EllipReg}, it is clear that the remainder $r_{1,1}$ belongs to $\Ta^\infty_r(s_0;L^2_yH^{s+s_0-3/2}_x)$. Notice that, by applying the paralinearization of coefficients $\alpha,\beta,\gamma$ (see Lemma~\ref{lem-PLDtN-Pf:Coeff}) with Proposition~\ref{prop:PDReg}, we have
\begin{align*}
    \mathscr{L}[\eta]\Phi \sim& \mathscr{P}[\eta]\Phi + T_{\partial_y^2\Phi} \alpha + T_{\partial_y\partial_x\Phi} \beta - T_{\partial_y\Phi}\gamma \\
    \sim& \mathscr{P}[\eta]\Phi - T_{\partial_y^2\Phi} \Op^\PM\left( 2(1+y^2|\partial_x\eta|^2)(\rho+\eta)^{-3} \right)\eta + T_{\partial_y^2\Phi}\Op^\PM\left( 2y^2\partial_x\eta(\rho+\eta)^{-2} \right)\partial_x\eta \\
    &+ T_{\partial_y\partial_x\Phi}\Op^\PM\left( 2y\partial_x\eta(\rho+\eta)^{-2} \right)\eta - T_{\partial_y\partial_x\Phi}\Op^\PM\left( 2y\eta(\rho+\eta)^{-1} \right)\partial_x\eta \\
    &- T_{\partial_y\Phi}\Op^\PM\left( 2y^{-1}(\rho+\eta)^{-3} + 4y|\partial_x\eta|^2(\rho+\eta)^{-3} - y\partial_x^2\eta(\rho+\eta)^{-2} \right) \eta \\
    &+ T_{\partial_y\Phi}\Op^\PM\left( 4y\partial_x\eta(\rho+\eta)^{-2} \right) \partial_x\eta - T_{\partial_y\Phi}\Op^\PM\left( y(\rho+\eta)^{-1} \right)\partial_x^2\eta.
\end{align*}
From the composition of paradifferential operators (see Proposition~\ref{prop:SymReg}), we can deduce that
\begin{equation}\label{eq-PLDtN-Pf:S2-1}
    \mathscr{L}[\eta]\Phi \sim \mathscr{P}[\eta]\Phi + T_{\iota_0}\eta + T_{\iota_1}\partial_x\eta + \T_{\iota_2}\partial_x^2\eta,
\end{equation}
where $\iota_j$'s are functions of $(\eta,\partial_x\eta,\partial_x^2\eta)$ defined by
\begin{align*}
    \iota_0 =& - 2(1+y^2|\partial_x\eta|^2)(\rho+\eta)^{-3} \partial_y^2\Phi + 2y\partial_x\eta(\rho+\eta)^{-2} \partial_y\partial_x\Phi   \\
    &\hspace{4em}- \left( 2y^{-1}(\rho+\eta)^{-3} + 4y|\partial_x\eta|^2(\rho+\eta)^{-3} - y\partial_x^2\eta(\rho+\eta)^{-2} \right) \partial_y\Phi, \\
    \iota_1 =& 2y^2\partial_x\eta(\rho+\eta)^{-2}\partial_y^2\Phi - 2y\eta(\rho+\eta)^{-1} \partial_y\partial_x\Phi + 4y\partial_x\eta(\rho+\eta)^{-2} \partial_y\Phi, \\
    \iota_2 =& -y(\rho+\eta)^{-1}\partial_y\Phi.
\end{align*}

The cancellation of these remainders is due to the correction $T_{(\rho+\eta)^{-1}\partial_y\Phi}(y\eta)$ in the good unknown. In fact, a direct calculus indicates that
\begin{align*}
    \mathscr{P}[\eta]\big( T_{(\rho+\eta)^{-1}\partial_y\Phi}(y\eta) \big) =& T_\alpha \Big[ \Op^\PM\Big( \partial_y^2 \big( (\rho+\eta)^{-1}\partial_y\Phi \big) \Big) \eta \Big] \\
    &+ T_\beta \Big[ \Op^\PM\Big( \partial_y\partial_x\big(y(\rho+\eta)^{-1}\partial_y\Phi\big) \Big) \eta + \Op^\PM\Big( \partial_y\big(y(\rho+\eta)^{-1}\partial_y\Phi\big) \Big) \partial_x\eta \Big] \\
    &+ \Op^\PM\Big( \partial_x^2\big( y(\rho+\eta)^{-1}\partial_y\Phi \big) \Big) \eta + \Op^\PM\Big( 2\partial_x\big( y(\rho+\eta)^{-1}\partial_y\Phi \big) \Big) \partial_x\eta \\
    &+ \Op^\PM\Big(  y(\rho+\eta)^{-1}\partial_y\Phi \Big) \partial_x^2\eta - T_\gamma \Big[ \Op^\PM\Big( \partial_y\big( (\rho+\eta)^{-1}\partial_y\Phi \big) \Big)\eta \Big].
\end{align*}
We apply again the composition of paradifferential operators (see Proposition~\ref{prop:SymReg}) and conclude that
\begin{equation}\label{eq-PLDtN-Pf:S2-2}
    \mathscr{P}[\eta]\big( T_{(\rho+\eta)^{-1}\partial_y\Phi}(y\eta) \big) \sim T_{\iota_0'}\eta + T_{\iota_1'}\partial_x\eta + \T_{\iota_2'}\partial_x^2\eta,
\end{equation}
where 
\begin{align*}
    \iota_0' =& \alpha \partial_y^2 \big( (\rho+\eta)^{-1}\partial_y\Phi \big) + \beta \partial_y\partial_x\big(y(\rho+\eta)^{-1}\partial_y\Phi\big) + \partial_x^2\big( y(\rho+\eta)^{-1}\partial_y\Phi \big) - \gamma \partial_y\big( (\rho+\eta)^{-1}\partial_y\Phi \big), \\
    \iota_1' =& \beta \partial_y\big(y(\rho+\eta)^{-1}\partial_y\Phi\big) + 2\partial_x\big( y(\rho+\eta)^{-1}\partial_y\Phi \big), \\
    \iota_2' =& y(\rho+\eta)^{-1}\partial_y\Phi.
\end{align*}
Then one show by direct calculus that
$$
\iota_0+\iota_0' = (\rho+\eta)^{-1}\partial_y\big( \mathscr{L}[\eta]\Phi \big), \quad \iota_1+\iota_1' = \iota_2+\iota_2' =0.
$$
As a result, Proposition~\ref{prop-PLDtN-Pf:S2} follows from~\eqref{eq-PLDtN-Pf:S2-1} and~\eqref{eq-PLDtN-Pf:S2-2}.

\subsection{Factorization}

The aim of this step is to decompose the paralinear operator $\mathscr{P}[\eta]$ as 
$$
T_{\alpha}(\partial_y + T_a)(\partial_y - T_A),
$$
modulo some smooth remainders, where $T_A,T_a$ are elliptic paradifferential operators. More precisely, we have
\begin{proposition}\label{prop-PLDtN-Pf:S3}
    Given $s_0>3/2$, there exist homogeneous symbols
    \begin{equation}\label{eq-PLDtN-Pf:S3-symb}
        \begin{aligned}
            a =& \sum_{-1 \leq j <s_0-3/2} a^{(-j)}, \quad A = \sum_{-1 \leq j <s_0-3/2} A^{(-j)}, \\
            a^{(-j)} =& p_a^{(-j)}(y,\eta,...\partial_x^{j+2}\eta) \xi|\xi|^{-j-1} + q_a^{(-j)}(y,\eta,...\partial_x^{j+2}\eta) \xi|\xi|^{-j-1}, \\
            A^{(-j)} =& p_A^{(-j)}(y,\eta,...\partial_x^{j+2}\eta) \xi|\xi|^{-j-1} + q_A^{(-j)}(y,\eta,...\partial_x^{j+2}\eta) \xi|\xi|^{-j-1},
        \end{aligned}
    \end{equation}
    where $p_a^{(-j)},q_a^{(-j)},p_A^{(-j)},p_A^{(-j)}$ belong to $C^\infty(\xR_y\times \xR^{j+3};\xC)$ ($-1 \leq j < s_0-3/2$), such that
    \begin{equation}\label{eq-PLDtN-Pf:S3-Decomp}
        \mathscr{P}[\eta] = T_{\alpha}(\partial_y + T_a)(\partial_y - T_A) + R_0\partial_y + R_1.
    \end{equation}
    Here $R_0$ and $R_1$ are operators depending smoothly on $y$ and $\eta$, such that
    \begin{equation}\label{eq-PLDtN-Pf:S3-Esti}
        R_0 \partial_y W, R_1 W \in \Ta^\infty_r(s_0;H^{\bullet+s_0-1-j}), \quad j=0,1.
    \end{equation}
    See Definition~\ref{def:Reg} and Notation~\ref{note:Reg} for the class $\Ta^\infty_r$.

    In particular, under the assumptions of Theorem~\ref{thm-PLDtN:Main}, we have
    \begin{equation}\label{eq-PLDtN-Pf:S3}
        (\partial_y + T_a)(\partial_y - T_A) W = r_2(\eta,\psi)
    \end{equation}
    where the remainder $r_2(\eta,\psi)$ belongs to the class $\Ta^\infty(s_0;H^{\bullet+1},L^2_yH^{\bullet+s_0-3/2}_x)$. Recall that the good unknown $W$ is defined by \eqref{eq-PLDtN-Pf:GoodUnk}.
\end{proposition}

\begin{proof}
Let us begin with the following computation:
\begin{align*}
    \mathscr{P}[\eta] - T_{\alpha}(\partial_y + T_a)(\partial_y - T_A) =& \mathscr{P}[\eta] - \big( T_{\alpha}\partial_y^2 + T_\alpha T_{a-A}\partial_y - T_\alpha T_a T_A - T_\alpha T_{\partial_yA} \big) \\
    =& \big( T_{\beta i\xi -\gamma} - T_{\alpha}T_{a-A} \big)\partial_y - \big( T_{\xi^2} - T_{\alpha}T_{a}T_{A} - T_{\alpha} T_{\partial_yA} \big),
\end{align*}
which indicates that the operators $R_0,R_1$ should be chosen as
$$
R_0 = T_{\beta i\xi -\gamma} - T_{\alpha}T_{a-A}, \quad R_1 = T_{\xi^2} - T_{\alpha}T_{a}T_{A} - T_{\alpha} T_{\partial_yA}.
$$
In what follows, we will construct $a,A$ in the form described in Proposition~\ref{prop-PLDtN-Pf:S3} such that
\begin{equation}\label{eq-PLDtN-Pf:S3-CondSymb}
    \left\{
    \begin{aligned}
        \beta i \xi - \gamma =& \alpha (a-A), \\
        \xi^2 =& \alpha (a \# A) + \alpha \partial_y A,
    \end{aligned}
    \right.
\end{equation}
where $a \# A$ should be understood as
$$
a \# A = \sum_{|\alpha|+k+k'<s_0-1/2} \frac{1}{i^{|\alpha|}\alpha!} \partial_\xi^{\alpha} a^{(1-k)} \partial_x^{\alpha} A^{(1-k')}.
$$
Note that once such $a,A$ exist, we have
$$
R_0 = T_{\alpha(a-A)} - T_{\alpha}T_{a-A}, \quad R_1 = T_{\alpha (a \# A)} - T_{\alpha}T_{a}T_{A} + T_{\alpha \partial_y A} - T_{\alpha} T_{\partial_yA}
$$
Thus the desired estimate \eqref{eq-PLDtN-Pf:S3-Esti} can be obtained through the composition of paradifferential operators (see Proposition~\ref{prop:SymReg}). As for the dependence in $y$, since all the symbols are smooth in $y\in[0.5,1]$, the application of derivative $\partial_y$ has no impact in the following algebric arguments. 

In order to find a solution to the equation \eqref{eq-PLDtN-Pf:S3-CondSymb} in the form \eqref{eq-PLDtN-Pf:S3-symb}, we can compare the homengenity in $\xi$ and reduce to
\begin{equation*}
    \left\{
    \begin{aligned}
        \beta i \xi =& \alpha (a^{(1)} - A^{(1)}),  \quad \xi^2 = \alpha a^{(1)} A^{(1)},  \quad -\gamma = \alpha (a^{(0)} - A^{(0)}),  \\
        0 =& a^{(-j)} - A^{(-j)}, \quad 1 \leq j < s_0-3/2, \\
        0 =& \sum_{|\alpha|+k+k'=j+1} \frac{1}{i^{|\alpha|}\alpha!} \partial_\xi^{\alpha} a^{(1-k)} \partial_x^{\alpha} A^{(1-k')} + \partial_y A^{(1-j)}, \quad 0 \leq j < s_0-3/2
    \end{aligned}
    \right.
\end{equation*}
From the first two equations, one solves 
\begin{equation}\label{eq-PLDtN-Pf:S3-aA1}
    a^{(1)} = \frac{\rho+\eta}{1+y^2|\partial_x\eta|^2} |\xi| - \frac{y(\rho+\eta) \partial_x\eta}{1+y^2|\partial_x\eta|^2} i\xi, \quad A^{(1)} = \frac{\rho+\eta}{1+y^2|\partial_x\eta|^2} |\xi| + \frac{y(\rho+\eta) \partial_x\eta}{1+y^2|\partial_x\eta|^2} i\xi.
\end{equation}
Then, by applying this to the third equation as well as the last one with $j=0$, we obtain a linear system for $a^{(0)},A^{(0)}$:
\begin{equation}\label{eq-PLDtN-Pf:S3-aA0}
    \left\{
    \begin{aligned}
        -\gamma =& \alpha (a^{(0)} - A^{(0)}),  \\
        0 =& a^{(1)}A^{(0)} + A^{(1)} a^{(0)} + \frac{1}{i} \partial_\xi a^{(1)} \partial_x A^{(1)} + \partial_y A^{(1)}.
    \end{aligned}
    \right.
\end{equation}
Since the determinant equals $\alpha(A^{(1)}+a^{(1)})$, which is non-zero, one could solve directly $a^{(0)},A^{(0)}$ in the form~\eqref{eq-PLDtN-Pf:S3-symb}. Similarly, once $a^{(1)},A^{(1)},...a^{(-j'+1)},A^{(-j'+1)}$ have been determined ($j'<s_0-3/2$), the last two equations become a linear system of $a^{(-j')},A^{(-j')}$,
\begin{equation*}
    \left\{
    \begin{aligned}
        0 =& a^{(-j')} - A^{(-j')}, \\
        0 =& a^{(1)}A^{(-j')} + A^{(1)} a^{(-j')} +  \sum_{|\alpha|+k+k'=j'+1,\ |\alpha|\ge 1} \frac{1}{i^{|\alpha|}\alpha!} \partial_\xi^{\alpha} a^{(1-k)} \partial_x^{\alpha} A^{(1-k')} + \partial_y A^{(1-j')},
    \end{aligned}
    \right.
\end{equation*}
with the same non-zero determinant $\alpha(A^{(1)}+a^{(1)})$. In this way, $a^{(-j)},A^{(-j)}$'s can be solved iteratively in the form~\eqref{eq-PLDtN-Pf:S3-symb}.

To deduce the identity \eqref{eq-PLDtN-Pf:S3}, we apply the decomposition \eqref{eq-PLDtN-Pf:S3-Decomp} to the identity \eqref{eq-PLDtN-Pf:S1} obtained in the previous step. Based on Proposition~\ref{prop-PLDtN-Pf:S2}, the estimates \eqref{eq-PLDtN-Pf:S3-Esti}, as well as the regularity of $W,\partial_yW$ (see Lemma~\ref{lem-PLDtN-Pf:GoodUnk}), guarantee that 
$$
    T_\alpha(\partial_y + T_a)(\partial_y - T_A) W = r_{2,1}(\eta,\psi) \in \Ta^\infty_r(s_0;L^2_yH^{\bullet+s_0-3/2}_x).
$$
Recall that, in Lemma~\ref{lem-PLDtN-Pf:Coeff}, we have seen that $\alpha,\alpha^{-1}$ and their derivatives in $y\in[0.5,1]$ belong to $\Ta^\infty(s_0;H^{\bullet+1},\mathcal{M}^0_{\bullet-1/2})$ uniformly in $y\in[0.5,1]$. Through applying $T_{\alpha^{-1}}$ to both sides of the equation above, we conclude that
$$
    (\partial_y + T_a)(\partial_y - T_A) W = r_2(\eta,\psi) := (\Id - T_{\alpha}T_{\alpha^{-1}}) (\partial_y + T_a)(\partial_y - T_A) W + T_{\alpha^{-1}} r_{2,1}(\eta,\psi),
$$
where the operator $(\Id - T_{\alpha}T_{\alpha^{-1}})$ is of order $-(s_0-1/2)$ thanks to the composition of paradifferential operators (see Proposition~\ref{prop:SymReg}). Finally, we make use of the boundedness of paradifferential operators (Proposition~\ref{prop:PDReg}) and the regularity of $W$ (Lemma~\ref{lem-PLDtN-Pf:GoodUnk}) to complete the proof of \eqref{eq-PLDtN-Pf:S3}.
\end{proof}

\subsection{Boundary Regularity}

Equation \eqref{eq-PLDtN-Pf:S3} in the previous step can be re-written as the following paradifferential parabolic equation:
\begin{equation*} 
    \partial_y U + T_a U = r, \quad U|_{y=0.5} = U_0,
    \quad \text{where }U := \partial_yW - T_AW.
\end{equation*}
Here
\begin{equation}\label{eq-PLDtN-Pf:S4-Rev}
    r=r_2 \in \Ta^\infty_r(s_0;L^2_yH^{\bullet+s_0-3/2}_x),\quad U_0 = \partial_yW|_{y=0.5} - (T_AW)|_{y=0.5} \in \Ta^\infty_r(s_0;H^{\bullet-1/2}).
\end{equation}

The purpose of this step is to study the regularity of $U$ at the boundary $y=1$. More precisely, we will show that the radial derivative of the good unknown $W$ can be replaced by its axial derivatives with an error possessing $(s_0-3/2)$-more regularity than $W$ itself.
\begin{proposition}\label{prop-PLDtN-Pf:S4}
    Under the assumptions of Theorem~\ref{thm-PLDtN:Main}, we have that
    \begin{equation}\label{eq-PLDtN-Pf:S4}
       (\partial_y - T_A)W = r_0,
    \end{equation}
    where the remainder $r_0$ belongs to the class $\Ta^\infty_r(s_0;L^2_y(0.6,1;H^{\bullet+s_0-1/2}))$ and $\Ta^\infty_r(s_0;L^\infty_y(0.6,1;H^{\bullet+s_0-1}))$. See Notation~\ref{note:Reg} for the definition of $\Ta_r^\infty$.
\end{proposition}

The key of the proof is to solve a paradifferential parabolic equation, which is standard. For example, as a consequence of Proposition 2.18 of \cite{ABZ2014}, we have the following result.
\begin{lemma}\label{lem-PLDtN-Pf:SolPara}
    Consider real numbers $\sigma_0\leq \sigma$, $0<r<1$, and an interval $I = [y_0,1]\subset\xR$. Let $p = p(y,x,\xi)\in \Gamma^1_{r}$, $q = q(y,x,\xi)\in \Gamma^0_0$ uniformly in $y\in I$, in the sense that 
    $$
    \sup_{y\in I}\mathcal{M}^1_r(p(y)) + \sup_{y\in I}\mathcal{M}^0_0(q(y)) < \infty.
    $$
    Assume that $p$ is elliptic, namely there exists constant $c>0$ such that, for all $y\in I$, $x\in\xT$, and $\xi \neq 0$,
    $$
    \RE p(y,x,\xi) \ge c|\xi|.
    $$
    Then, for all $r\in L^2_y(I;H^{\sigma-1/2}_x)$ and $U_0\in H_x^{\sigma_0}$, the equation
    \begin{equation}\label{eq-PLDtN-Pf:SolPara-Eq}
        \partial_y U + T_p U + T_q U = r, \quad U|_{y=y_0}=U_0
    \end{equation}
    admits a unique solution $U$ satisfying, for all $y_0<y_1 < 1$,
    \begin{equation}\label{eq-PLDtN-Pf:SolPara-Esti}
        \|U\|_{C^0_y(y_1,1;H_x^{\sigma})} + \|U\|_{L^2_y(y_1,1;H_x^{\sigma+1/2})} \leq K \Big( \|U_0\|_{H_x^{\sigma_0}} + \|r\|_{L^2_y(I;H^{\sigma-1/2}_x)} \Big),
    \end{equation}
    where $K$ is a positive constant depending on $\sigma,\sigma_0,r,y_0,y_1$ and the semi-norms $\sup_{y\in I}\mathcal{M}^1_r(p(y))$ and $\sup_{y\in I}\mathcal{M}^0_0(q(y))$. Moreover, the solution is unique in $C^0_y(I;H_x^{\sigma_1})\cap L^2_y(I;H_x^{\sigma_1+1/2})$ for any $\sigma_1\leq \sigma_0$. 
\end{lemma}
\begin{proof}
    Recall that Proposition 2.18 of \cite{ABZ2014} already covers the case with $y_0=y_1$, $q=0$, and $\sigma_0=\sigma$. 
    
    Firstly, we clarify that the presence of non-zero $q$ has no impact in the proof since it is of order $0$. In fact, since $q$ belongs to $\Gamma^0_0$, $T_q$ is of order $0$. This means that the G{\aa}rding's inequality (see, for example, Section 6.3.2 of \cite{MePise}) for $p+q$ and $p$ is the same: there exists $C_1,C_2>0$ such that
    $$
    \langle T_{p+q}U, U \rangle_{H^{\sigma}_x} \leq C_1 \|U\|_{H_x^{\sigma+1/2}}^2 - C_2 \|U\|_{H_x^{\sigma+(1-\rho)/2}}^2,
    $$
    which is the kernel of the proof of Proposition 2.18 of \cite{ABZ2014}. Then one may repeat exactly the same arguments as in \cite{ABZ2014} to conclude the desired result, which we omit for simplicity.

    As for the cases with $\sigma_0<\sigma$, we can decompose the interval $[y_0,0]$ into many subintervals and apply the estimate \eqref{eq-PLDtN-Pf:SolPara-Esti} on each subinterval, where the regularity of initial data increases until $H^{\sigma}_x$. As a consequence, the target estimate \eqref{eq-PLDtN-Pf:SolPara-Esti} only holds for a slightly smaller interval $[y_1,1]$ than $I$.
\end{proof}

Applying Lemma~\ref{lem-PLDtN-Pf:SolPara} with
$$
p = a^{(1)},\quad q = \sum_{0\leq j < s_0-2} a^{(-j)},
$$
it is clear that the quantity $(\partial_y-T_A)W$ is in $L^\infty_yH^{s+s_0-1}_x$ as $(\eta,\psi)\in H^{s+1}\times H^{s+1/2}$ and it remains to check the regularity in $(\eta,\psi)$. Alternatively, it suffices to check that the regularity is preserved when solving the paradifferential parabolic equations, which can be seen from Lemma~\ref{lem:SolReg}: simply take 
\begin{align*}
    &X^s = H_x^{s-1/2}\times L^2_yH_x^{s+s_0-3/2},\quad Y^s = \mathcal{M}^1_0,\quad Z^s = L_y^\infty H_x^{s+s_0-1}\cap L_y^2 H_x^{s+s_0-1/2}, \\
    &f = (U_0,r), \quad g = p+q = a,\quad b(g,w) = (0,T_{g}w), \\
    &\mathscr{R}[\eta]f = \mathscr{R}[\eta](U_0,r) := U,
\end{align*}
where $\mathscr{R}[\eta]$ is the resolution operator of \eqref{eq-PLDtN-Pf:SolPara-Eq}. The regularity of $f,g$ in $\eta$ has been proved (see \eqref{eq-PLDtN-Pf:S4-Rev}), while the estimate for the bilibear form $b$ follows from Proposition~\ref{prop:PDReg}. The estimate for $\mathscr{R}[\eta]$ is no more than the consequence of Lemma~\ref{lem-PLDtN-Pf:SolPara} above. To check the identity \eqref{eq:SolReg-Cond}, we use tilde to express the quantities perturbated in $\eta$. Then a direct calculus gives
$$
    \big(\partial_y+T_a\big) (\tilde{U}-U) = \tilde{r}-r - T_{\tilde{a}-a}\tilde{U}, \quad (\tilde{U}-U)|_{y=0.5} = \tilde{U}_0-U_0,
$$
implying that
$$
    \tilde{U}-U = \mathscr{R}[\eta]\Big( \big(\tilde{U}_0-U_0,\tilde{r}-r\big) - \big( 0, T_{\tilde{a}-a}\tilde{U} \big) \Big) = \mathscr{R}[\eta]\big( \tilde{f}-f - b(\tilde{a}-a,\tilde{U}) \big).
$$

In conclusion, the regularity of $r_0$ in~\eqref{eq-PLDtN-Pf:S4} follows from Lemma~\ref{lem:SolReg}.

\subsection{Paralinearization of Dirichlet-Neumann Operator}

In this step, we will complete the proof of Proposition~\ref{prop-PLDtN-Pf:Main}. Recall that, in terms of $\Phi$ and $\eta$, the Dirichlet-Neumann operator can be written as (see \eqref{eq:Formu-DtN})
\begin{equation*}
	G[\eta]\psi = \left.\left( \frac{1+|\partial_x\eta|^2}{\rho+\eta} \partial_y\Phi - \partial_x\eta\partial_x\Phi \right)\right|_{y=1}.
\end{equation*}
In order to obtain the desired paralinearization \eqref{eq-PLDtN-Pf:Goal}, the main difficulty is to replace the radial derivative $\partial_y\Phi$ can be replaced by axial ones, which has been done in the previous step (see Proposition~\ref{prop-PLDtN-Pf:S4}) for the good unknown $W = \Phi - T_{(\rho+\eta)^{-1}\partial_y\Phi}(y\eta)$. Thus, we need to express the Dirichlet-Neumann operator in terms of $W$ and $\eta$, up to smooth remainders.

For simplicity, in the following, we use the notation
$$
    u \sim v \Leftrightarrow (u-v)|_{y=1} \in \mathcal{T}^\infty_r(s_0;H^{\bullet+s_0-1}_x),
$$
to emphasize that the remainder is regular (recall Definition~\ref{def:Reg} and Notation~\ref{note:Reg} for the class $\Ta^\infty_r$) as in the conclusion of Proposition~\ref{prop-PLDtN-Pf:Main}.

Firstly, one observes that
$$
	G[\eta]\psi = \left.\left( (\rho+\eta)\alpha\partial_y\Phi + \frac{(\rho+\eta)\beta}{2}\partial_x\Phi \right)\right|_{y=1}.
$$
Thus, it suffices to paralinearize the right-hand side without taking trace. Through paraproduct decomposition, the right-hand side verifies
$$
	(\rho+\eta)\alpha\partial_y\Phi + \frac{(\rho+\eta)\beta}{2}\partial_x\Phi \sim T_{(\rho+\eta)\alpha}\partial_y\Phi + T_{\partial_y\Phi}\big((\rho+\eta)\alpha\big) + \frac{1}{2}T_{(\rho+\eta)\beta}\partial_x\Phi + \frac{1}{2}T_{\partial_x\Phi}\big( (\rho+\eta)\beta \big),
$$
since the difference of both sides equals
$$
	R_\PM\big( (\rho+\eta)\alpha, \partial_y\Phi \big) + \frac{1}{2}R_\PM\big( (\rho+\eta)\beta, \partial_x\Phi \big) \in \Ta^\infty_r(s_0;L^\infty_yH^{\bullet+s_0-2}_x),
$$
which follows from the regularity of $\alpha,\beta,\Phi$ (Lemma~\ref{lem-PLDtN-Pf:Coeff} and Proposition~\ref{prop-PLDtN:EllipReg}) and the boundedness of paraproduct (Proposition~\ref{prop:PMReg}). Now, we replace $\Phi$ by $W+ T_{(\rho+\eta)^{-1}\partial_y\Phi}(y\eta)$ and apply Proposition~\ref{prop-PLDtN-Pf:S4},
\begin{align*}
	(\rho+\eta)\alpha\partial_y\Phi + \frac{(\rho+\eta)\beta}{2}\partial_x\Phi \sim& \left(T_{(\rho+\eta)\alpha}T_A + \frac{1}{2}T_{(\rho+\eta)\beta}\partial_x \right) W + T_{(\rho+\eta)\alpha} \partial_y\big(T_{(\rho+\eta)^{-1}\partial_y\Phi}(y\eta)\big) \\
	&+ \frac{1}{2}T_{(\rho+\eta)\beta}\partial_x\big( T_{(\rho+\eta)^{-1}\partial_y\Phi}(y\eta) \big) + T_{\partial_y\Phi}\big((\rho+\eta)\alpha\big) + \frac{1}{2}T_{\partial_x\Phi}\big( (\rho+\eta)\beta \big) \\
	=:& I_1 + I_2 + I_3, 
\end{align*}
where
\begin{align*}
	I_1 =& \left(T_{(\rho+\eta)\alpha}T_A + \frac{1}{2}T_{(\rho+\eta)\beta}\partial_x \right) W, \\ 
	I_2 =& T_{(\rho+\eta)\alpha} \partial_y\big(T_{(\rho+\eta)^{-1}\partial_y\Phi}(y\eta)\big) + \frac{1}{2}T_{(\rho+\eta)\beta}\partial_x\big( T_{(\rho+\eta)^{-1}\partial_y\Phi}(y\eta) \big), \\
	I_3 =& T_{\partial_y\Phi}\big((\rho+\eta)\alpha\big) + \frac{1}{2}T_{\partial_x\Phi}\big( (\rho+\eta)\beta \big).
\end{align*}

In the following, we will simplify the three terms $I_1,I_2,I_3$ separately. To control $I_1$, it suffices to use the composition of paradifferential operators (Proposition~\ref{prop:SymReg}) and the regularity of $W$ (Lemma~\ref{lem-PLDtN-Pf:GoodUnk}),
\begin{equation}\label{eq-PLDtN-Pf:S5-PrSym}
	I_1 \sim T_{\tilde{\Lambda}} W, \quad  \tilde{\Lambda} = \frac{1+y^2|\partial_x\eta|^2}{\rho+\eta} A - y\partial_x\eta i\xi.
\end{equation}
As for the second term $I_2$, we compute directly the derivative in $x,y$ and then use again the composition of paradifferential operators (Proposition~\ref{prop:SymReg}),
\begin{align*}
	I_2 =& T_{(\rho+\eta)\alpha}T_{(\rho+\eta)^{-1}\partial_y^2\Phi}(y\eta) + T_{(\rho+\eta)\alpha} T_{(\rho+\eta)^{-1}\partial_y\Phi}\eta + \frac{1}{2}T_{(\rho+\eta)\beta}  T_{(\rho+\eta)^{-1}\partial_y\Phi}(y\partial_x\eta) \\
	&+ \frac{1}{2}T_{(\rho+\eta)\beta} T_{(\rho+\eta)^{-1}\partial_{y}\partial_x\Phi}(y\eta) - \frac{1}{2}T_{(\rho+\eta)\beta} T_{(\rho+\eta)^{-2}\partial_x\eta\partial_y\Phi}(y\eta) \\
	\sim& T_{\alpha\partial_y^2\Phi}(y\eta) + T_{\alpha\partial_y\Phi}\eta + \frac{1}{2}T_{\beta\partial_y\Phi}(y\partial_x\eta) + \frac{1}{2}T_{\beta\partial_{y}\partial_x\Phi}(y\eta) - \frac{1}{2}T_{(\rho+\eta)^{-1}\beta\partial_x\eta\partial_y\Phi}(y\eta).
\end{align*}
The treatment of $I_3$ requires a paralinearization (see Proposition~\ref{prop:PLReg}) for $(\rho+\eta)\alpha$,
$$
	(\rho+\eta)\alpha - \Big( - T_{(\rho+\eta)^{-2}(1+y^2|\partial_x\eta|^2)}\eta + 2T_{y^2(\rho+\eta)^{-1}\partial_x\eta}\partial_x\eta \Big) \in \mathcal{T}^{\infty}(s_0;L^\infty_yH_x^{\bullet+s_0-1/2}),
$$
which, combined with the composition of paradifferential operators (Proposition~\ref{prop:SymReg}), implies that
\begin{align*}
	I_3 \sim& T_{\partial_y\Phi} \Big( - T_{(\rho+\eta)^{-2}(1+y^2|\partial_x\eta|^2)}\eta + 2T_{y^2(\rho+\eta)^{-1}\partial_x\eta}\partial_x\eta \Big) - T_{\partial_x\Phi} (y\partial_x\eta) \\
	\sim& - T_{(\rho+\eta)^{-2}(1+y^2|\partial_x\eta|^2)\partial_y\Phi} \eta + 2T_{y^2(\rho+\eta)^{-1}\partial_x\eta\partial_y\Phi}\partial_x\eta - T_{\partial_x\Phi} (y\partial_x\eta) \\
	=& - T_{\alpha\partial_y\Phi} \eta - T_{\beta\partial_y\Phi}(y\partial_x\eta) - T_{\partial_x\Phi} (y\partial_x\eta).
\end{align*}

Now, from these simplifications of $I_j$'s, we are able to deduce that
\begin{align*}
	&\hspace{-2em}(\rho+\eta)\alpha\partial_y\Phi + \frac{(\rho+\eta)\beta}{2}\partial_x\Phi \sim I_1 + I_2 + I_3 \\
	\sim& T_{\tilde{\Lambda}}W + \Op^\PM\left( \alpha\partial_y^2\Phi + \frac{1}{2}\beta\partial_y\partial_x\Phi - \frac{1}{2}(\rho+\eta)^{-1}\beta\partial_x\eta\partial_y\Phi \right) (y\eta) - \Op^\PM\left( \frac{1}{2}\beta\partial_y\Phi + y\partial_x\Phi \right) \partial_x\eta,
\end{align*}
where, through the elliptic equation~\eqref{eq-PLDtN:Ellip} and the formulas~\eqref{eq-PLDtN:EllipCoeff} for the coefficients, the symbol of the first paradifferential operator reads
\begin{align*}
	&\hspace{-2em}\mathscr{L}[\eta]\Phi - \frac{1}{2}\beta\partial_y\partial_x\Phi - \partial_x^2\Phi + \left( \gamma + \frac{y|\partial_x\eta|^2}{(\rho+\eta)^2} \right)\partial_y\Phi \\
	=& - \frac{1}{2}\beta\partial_y\partial_x\Phi - \partial_x^2\Phi - \frac{1}{2}\partial_x\beta\partial_y\Phi - \frac{1}{y(\rho+\eta)^2}\partial_y\Phi \\
	=& -\partial_x\left( \partial_x\Phi + \frac{1}{2}\beta\partial_y\Phi \right)  - \frac{1}{y(\rho+\eta)^2}\partial_y\Phi.
\end{align*}
Therefore, we can conclude that
\begin{align*}
	&\hspace{-2em}(\rho+\eta)\alpha\partial_y\Phi + \frac{(\rho+\eta)\beta}{2}\partial_x\Phi \\
	\sim& T_{\tilde{\Lambda}}W - \Op^\PM\left( \partial_x\left( \partial_x\Phi + \frac{1}{2}\beta\partial_y\Phi \right) \right)(y\eta) - \Op^\PM\left( \frac{1}{2}\beta\partial_y\Phi + y\partial_x\Phi \right) \partial_x\eta  \\
	&- \Op^\PM\left((\rho+\eta)^{-2}\partial_y\Phi\right)\eta \\
	\sim& T_{\tilde{\Lambda}}W - \partial_x\left( \Op^\PM\left(\partial_x\Phi - (\rho+\eta)^{-1}\partial_x\eta\partial_y\Phi\right) \eta \right) - \Op^\PM\left((\rho+\eta)^{-2}\partial_y\Phi\right)\eta.
\end{align*}

Finally, we take the trace $y=1$ of both side and use the fact that
$$
    B = (\rho+\eta)\partial_y\Phi|_{y=1},\quad V = \partial_x\Phi|_{y=1} - \partial_x\eta(\rho+\eta)^{-1}\partial_y\Phi|_{y=1},
$$
which can be deduced from their relations \eqref{DN_V_B} with the Dirichlet-Neumann operator, to conclude the desired paralinearization formula
$$
    G[\eta]\psi - \left( T_{\tilde{\lambda}} - \partial_x(T_V\eta) - T_{(\rho+\eta)^{-1}B}\eta \right) \in \Ta^\infty_r(s_0;H^{\bullet+s_0-1}_x).
$$
Here the symbol $\tilde{\lambda}$ is given by $\tilde{\Lambda}|{y=1}$, whose principal and subprincipal part, due to~\eqref{eq-PLDtN-Pf:S5-PrSym}, \eqref{eq-PLDtN-Pf:S3-aA1}, and~\eqref{eq-PLDtN-Pf:S3-aA0}, read
\begin{equation*}
    \begin{aligned}
        \tilde{\lambda}^{(1)} =& \frac{1+y^2|\partial_x\eta|^2}{\rho+\eta} A^{(1)}|_{y=1} - y\partial_x\eta i\xi = |\xi|, \\
        \tilde{\lambda}^{(0)} =& \frac{1+y^2|\partial_x\eta|^2}{\rho+\eta} A^{(0)}|_{y=1} = - \frac{1}{2(\rho+\eta)} - \frac{\partial_x\eta }{2(\rho+\eta)}i\sgn(\xi).
    \end{aligned}
\end{equation*}
while the remaining parts take the form described in \eqref{eq-PLDtN-Pf:PrSym}.

\vspace{1em}

\noindent
\textbf{Acknowledgments.} The authors would like to warmly thank Thomas Alazard, Zexing Li, and Jean-Marc Delort for numerous discussions on this problem. This work is partially completed during the second author's PhD thesis at Universit{\'e} Paris-Salcay and Universit{\'e} Sorbonne Paris-Nord.

\bibliographystyle{abbrv}\bibliography{References}

\end{spacing}
\end{document}